\documentclass[a4paper,10pt]{article}

 \usepackage[utf8]{inputenc}
 \usepackage{amssymb}
 \usepackage{amsmath}
 \usepackage{graphicx}
 \usepackage[T1]{fontenc}
 \usepackage{hyperref}
 \usepackage{xcolor}
 \usepackage{natbib}

 \newcommand{\bayes}{s}
 \newcommand{\cT}{\mathcal{T}}
 \newcommand{\crit}{\text{crit}}
 \newcommand{\risk}[1]{\mathcal{L}(#1)}
 \newcommand{\loss}[1]{\ell(s,#1)}
 
 \newcommand{\cfun}{\gamma}
 
 \newcommand{\Norm}[1]{\left\|#1\right\|}
 \newcommand{\NormInfinity}[1]{\left\|#1\right\|_{\infty}}
 \newcommand{\BlackBox}{\rule{1.5ex}{1.5ex}} 
 
 \newcommand{\ho}[1]{\hat{R}_{#1}^{ho}}
 \newcommand{\cv}{\hat{R}_{\cT}^{cv}}
 
 \newcommand{\oracle}{\mathrm{or}}
 \newcommand{\Deltak}{\Delta}
 \newcommand{\Deltaor}{\mathcal{E}}
 \newcommand{\Deltal}{\mathfrak{e}}
 \newcommand{\CV}[2]{\mathrm{CV}_{#1}(#2)}
 \newcommand{\cM}{\mathcal{M}}
 \newcommand{\funal}{H}
 \newenvironment{proof}{\par\noindent{\bf Proof\ }}{\hfill\BlackBox\\[2mm]}

 \newtheorem{lemma}{Lemme}[section]
 \newtheorem{theorem}{Theorem}
 \newtheorem{corollary}[theorem]{Corollary}
 \newtheorem{claim}[theorem]{Claim}

 \newtheorem{proposition}[lemma]{Proposition}
 \newtheorem{definition}{Definition}
 \newtheorem{hyp}{Hypothesis}

\newcommand{\HO}[2]{\text{HO}_{#1}\left(#2\right)}

 \newcommand{\Ethet}[2]{\hat{\theta}_{#1}^{#2}}
 \newcommand{\ERMtrigo}[2]{\hat{s}_{#1}^{#2}}

 \newcommand{\numberthis}{\addtocounter{equation}{1}\tag{\theequation}}
 \DeclareMathOperator*{\argmin}{argmin}
 
 \DeclareMathOperator{\Var}{Var}
 \DeclareMathOperator{\Cov}{Cov}

 \title{Local asymptotics of cross-validation in least-squares density estimation}
 
 \author{Guillaume Maillard}
 
\begin{document}


\maketitle

 \begin{abstract}
 In model selection, several types of cross-validation are commonly used and many variants have been introduced. While consistency of some of these methods has been proven, their rate of convergence to the oracle is generally still unknown. Until now, an asymptotic analysis of cross-validation  able to answer this question has been lacking. Existing results focus on the "pointwise" estimation of the risk of a single estimator, whereas analysing model selection requires understanding how the CV risk varies with the model. In this article, we investigate the asymptotics of the CV risk in the neighbourhood of the optimal model, for trigonometric series estimators in density estimation. Asymptotically, simple validation and "incomplete" $V-$fold CV behave like the sum of a convex function $f_n$ and a symmetrized Brownian changed in time $W_{g_n/V}$.   
We argue that this is the right asymptotic framework for studying model selection.
\end{abstract}


\section{Introduction}

The problem of selecting a model, a hyperparameter or, more generally, an estimator, is ubiquitous in statistics and a large number of methods have been proposed. Cross-validation is certainly one of the most popular and most widly used in practice.  However, there are many ways to perform cross-validation, depending on how the data is split: V-fold cross-validation, leave-p-out, Monte Carlo CV, etc. \cite[Section 4.3]{Arl_Cel:2010:surveyCV} . Each of these methods has hyperparameters that must be set by the user, such as $V$ for $V-$fold cross-validation and $p$ for the leave-p-out. In addition, many variants and alternatives to CV have been proposed in the litterature, including $V-$fold penalties \citep{arlot:penVF}, AmCV \citep{lecue2012} and Aggregated hold-out \citep{agghoo_rkhs}. The practitioner is faced with the problem of choosing one of these methods and selecting its hyperparameters. Since these are already model selection methods, using data to select among them threatens to lead to an infinite regress. In order to make a principled choice, one must therefore find some way to compare their performance à priori.

However, theoretical results on model selection generally lack the precision necessary to compare the methods which are used in practice. As noted by \citet{Arl_Ler:2012:penVF:JMLR}, when selecting a model $m \in \cM$ to minimize the risk $\risk{m}$ using estimators $\hat{m}_1, \hat{m}_2$, existing results only compare $\hat{m}_1, \hat{m}_2$ at first order --- that is to say, in terms of the ratio $\frac{\risk{\hat{m}_1}}{\risk{\hat{m}_2}}$ and its limit.  This is insufficient to distinguish the performance of $V$-fold from that of hold-p-out with $p = \frac{n}{V}$, for example, even though the first of these is generally much better in practice.

In order to perform fine comparisons between different CV methods, theoricians have mainly focused on calculating the bias and variance of the CV risk estimator, in a variety of contexts (such as linear regression \citep{burman1989}, kernel density estimation \citep{Celisse2014} or $k-$nearest neighbour classification \citep{Celisse2018TheoreticalAO}). An in-depth review can be found in \citet{Arl_Cel:2010:surveyCV}. Recently, asymptotic normality results have also been proved under stability assumptions \citep{Austern_Zhou2020, bayle2020}. However, this analyzes CV purely as a \emph{risk estimator}, not as a \emph{model selection} method. Risk estimation and model selection behave differently in theory and in practice. In practice, \citet{Breiman1992} report that $10-$fold CV is generally better than leave-one-out for model selection, whereas the reverse is true for risk estimation. In theory, as explained by \citet{Arl_Ler:2012:penVF:JMLR}, what matters for model selection performance is that the \emph{sign} of $\crit(m_1) - \crit(m_2)$ be the same as that of $\risk{m_1} - \risk{m_2}$, with errors occurring only for values of $\risk{m}$ close to the optimum $\risk{m_*}$.  
Moreover, the definition of $\hat{m},m_*$ implies the inequality
\[ \risk{\hat{m}} - \risk{m_*} \leq \left| \crit(\hat{m}) - \risk{\hat{m}} - (\crit(m_*) - \risk{m_*})  \right|. \]
 As $\hat{m}$ will typically concentrate around $m_*$ (in some sense), this bound can be much smaller than $|\crit(m_*) - \risk{m_*}|$ and $|\crit(\hat{m}) - \risk{\hat{m}}|$. Thus, an approach to model selection based on a central limit theorem for $\crit(m) - \risk{m}$ will not yield the right rate of convergence in general.
 
 One potential solution is to perform a local approximation of $\crit$ in the vicinity of $m_*$ or in other words, to replace the study of $\crit(m)$ with that of a centered and rescaled process $\crit(m_* + \alpha \Deltak) - \crit(m_*)$. In order to study the model selection problem, the scale $\Deltak$ should  reflect the order of magnitude of $\hat{m} - m_*$, where $\hat{m}$ denotes the selected parameter, i.e the minimizer of $\crit$. 
 
 This approach was successfully applied to kernel density estimation in a classical paper by 
 \citet{Hall1987}, which establishes the asymptotic normality of the cross-validated kernel parameter $\hat{h}_{CV}$. The method used by these authors is, however, highly specific to kernel density estimation.
As for parametric M-estimators (proof of \cite[Theorem 5.3]{vdV1998}, for example), they show that the CV criterion can be locally approximated by a process of the form $c(h - h_*)^2 + (h-h_*)Z$, where $Z$ is a normal random variable and $c$ is a non-negative constant. This crucially uses the differentiability of kernel methods with respect to their parameter. Moreover, to prove that the risk is locally quadratic, they assume that the true density $s$ is twice differentiable and that the kernel $K$ is such that $\int z^2 K(z) dz \neq 0$. This effectively sets the rate of convergence of the kernel method \citep[Proposition 1.6]{Tsybakov2009}, whereas one of the main theoretical advantages of CV is that it \emph{adapts} to various non-parametric convergence rates.
   
In model selection, these arguments do not apply. The set of models $\cM$ is discrete, so there is no differentiability; moreover, since one model $m_*$ may be much better than all others, smooth, locally quadratic behaviour of $\risk{m}$ around $m_*$ cannot be expected in general. The point of the present article is to find a \emph{local approximation} to cross-validation in a simple non-parametric model selection problem (least-squares density estimation using Fourier series), while taking into account a variety of possible behaviours of the true density $s$ (in particular, different convergence rates). The key assumption made about $s$ is that its Fourier coefficients are non-increasing, which, roughly speaking, guarantees uniqueness of the optimum $m_*$.
The program raises a number of technical issues. First, since the class of functions the risk of which must be estimated is random and grows with $n$, standard empirical process theory cannot be used. Moreover, it is by no means clear that the process, once appropriately centered and scaled, actually converges to a limit (in fact, we conjecture that this is false in general).  

Instead, we prove that simple validation, at the relevant scale, is approximately the sum of a convex function $f_n$ and a Brownian motion changed in time $W_{g_n}$. The same result holds for "incomplete" $V$-fold cross-validation (Definition \ref{3.def_vfcv}), as long as $V$ remains constant when $n \to + \infty$. The approximating process depends on $n$ (this is not a limit), but several inequalities on $f_n$ and $g_n$ show that it does not become trivial when $n$ tends to $+ \infty$. Interestingly, in the case of simple validation, the process is independent from the training data constituting the "training sample" of the hold-out. As a consequence, for any fixed $V$, (incomplete) $V$-fold cross-validation behaves as if the folds were independent: in particular, the asymptotic variance is reduced by a factor $V$ relative to the hold-out. When $V \to +\infty$, $V-$fold CV concentrates around the deterministic $f_n$, which suggests that the CV parameter concentrates faster than the hold-out parameter. 

The results of this article allow to make use of the abundant theory available on Brownian motion in order to study
the parameter selection step. They yield a (heuristic) formula for the rate of convergence of the CV parameter $\hat{m}$ to $m_*$, and open the way for proving \emph{second-order optimal} oracle inequalities for the hold-out and incomplete $V-$fold cross-validation, which will be the subject of future work.


\section{\texorpdfstring{$L^2$}{L2} density estimation} \label{3.chap3.secL2density}
Let $s \in L^2([0;1])$ be a probability density function.
Given 
a sample $X_1, \ldots, X_n$ drawn according to the density $s$, the $L^2$ density estimation problem consists in constructing an estimator $\hat{s}_n$ that approaches $s$ in terms of the $L^2$ norm.

Although it is not obvious at first glance (this is not true for the other $L^p$ norms),
this non-parametric density estimation problem can be reformulated as a risk minimization problem, with a contrast function:
$\cfun(t,x) = \Norm{t}^2 - 2t(x),$
which yields the \emph{risk} $\mathbb{E}[\cfun(t,X)] = \Norm{t}^2 - 2 \int s(x) t(x) dx = \Norm{t-s}^2 - \Norm{s}^2$.
It follows that $s$ is indeed the minimizer of the risk corresponding to the $\cfun$ contrast function, and furthermore the \emph{excess risk} $\loss{t} := \mathbb{E}[\cfun(t,X)] - \mathbb{E}[\cfun(s,X)]$ coïncides with the $L^2$ norm: 
\[ \loss{t}  = \Norm{t-s}^2. \]
As a result, it is possible to construct an \emph{empirical risk estimator}, 
\[ P_n \cfun(t) = \frac{1}{n} \sum_{i = 1}^n \cfun(t,X_i) = \Norm{t}^2 - \frac{2}{n} \sum_{i = 1}^n t(X_i), \]
which can in particular be used to perform cross-validation.

Here we will consider as a family of non-parametric estimators  the \emph{empirical orthogonal series} estimators \cite[Section 3.1]{Efromovich1999} on a trigonometric basis.
To ease the presentation, we consider only cosine functions, which is equivalent to assuming that $s$ is symmetrical with respect to $\frac{1}{2}$. 
This restriction is of no fundamental importance --- it is reasonable to conjecture that the results remain valid with the complete trigonometric basis.

For every $j \in \mathbb{N}^*$, let $\psi_j: x \mapsto \sqrt{2} \cos (2\pi j x)$ and let $\psi_0: x \mapsto 1$.
The collection $(\psi_j)_{j \in \mathbb{N}}$ is an orthonormal basis of the subset of $L^2([0;1])$ of functions symmetrical with respect to $\frac{1}{2}$. 

Let $D_n = (X_1,...,X_n)$ be a sample. 
For any $n \in \mathbb{N}$ and any $T \subset \{1,\ldots,n\}$, we will denote, for any real-valued measurable function $t$,
\[ P_n^T(t) = \frac{1}{|T|} \sum_{i \in T} t(X_i). \]
Consider the estimators defined as follows.
\begin{definition}
 For all $k \in \mathbb{N}$ and all $T \subset \{1,\ldots,n\}$,
 \[ \ERMtrigo{k}{T} = \sum_{j = 0}^k P_n^T(\psi_j) \psi_j, \]
 where $\psi_0 = 1$ and for all $j \geq 1$, $\psi_j(x) = \sqrt{2} \cos(2 \pi j x)$.
\end{definition}

The estimators $\ERMtrigo{k}{T}$ are empirical risk minimizers on the \emph{models} 
\[ E_k = \left\{ \sum_{j = 0}^k v_j \psi_j : v \in \mathbb{R}^{k+1} \right\}. \]
The problem of parameter choice $k$ is therefore a problem of \emph{model selection} 
within the model collection $(E_k)_{k \geq 0}$. 
Here, the models are nested, meaning $E_k \subset E_{k'}$ for every $k \leq k'$.

\section{Risk estimation for the hold-out}
The larger $k$ is, the better the approximation of $s$ by the functions of $E_k$, but the more difficult it is to estimate the best approximation to $s$ within $E_k$. The choice of $k$ is therefore
subject to a bias-variance trade-off which, if properly carried out, allows adaptation to the smoothness
of~$s$, simultaneously reaching the minimax risk on Lipschitz spaces of periodic functions \citep{Efromovich1999}.

\subsection{Cross-validation}
Since the risk, except for a constant, is expressed as the expectation of a contrast function \[P \cfun(\ERMtrigo{k}{T}) := E_X \left[ \cfun(\ERMtrigo{k}{T},X) \right] = \Norm{\ERMtrigo{k}{T} - s}^2 - \Norm{s}^2, \] 
it can be estimated by hold-out and cross-validation as
in regression and classification. 

This is the subject of the following definition.
\begin{definition} \label{3.def_ho}
Let $D_n$ be an i.i.d sample drawn from the distribution $s(x) dx$. 
Let $n_t \in \{1, \ldots, n-1\}$.
 Let $T \subset \{1..n \}$ be a subset with cardinality $|T| = n_t$.
 Then, for all $k \in \mathbb{N}$, we define the hold-out estimator 
 of the risk of $\ERMtrigo{k}{}$ with training sample indices $T$ by 
 \[ \HO{T}{k} = \Norm{\ERMtrigo{k}{T}}^2 - 2P_n^{T^c} (\ERMtrigo{k}{T}). \]
\end{definition}

$\HO{T}{\cdot}$ is indeed an estimator since the norm $\Norm{\cdot}$ is computed
with respect to a known dominating measure (in this case the Lebesgue measure) and so does not depend on the distribution of $X$. 
Moreover,
\[ \HO{T}{k} =  \Norm{\ERMtrigo{k}{T} - s}^2 - 2(P_n^{T^c} - P) (\ERMtrigo{k}{T}): \]
the hold-out risk estimator can be expressed as the sum of the excess risk and 
a centered empirical process.

The hold-out risk estimator depends on the choice of a subset $T$ of $\{1,\ldots,n\}$, 
but its distribution depends only on the cardinality of that subset. 
The precise choice of a subset $T$ of cardinality $n_t$ will thus play no role in the sequel. 
We will therefore denote by $T$ any subset of $\{1, \ldots, n\}$ of cardinality $n_t$. 

Since the distribution of $\HO{T}{\cdot}$ only depends on $T$ through its cardinality $n_t$, it is possible to construct an estimator with smaller variance by averaging several $\HO{T_i}{\cdot}$. This is the idea behind cross-validation.  In the V-fold scheme presented below, the $T_i$ are chosen such that the test sets $T_i^c$ are disjoint.

\begin{definition} \label{3.def_vfcv}
Let $n_t,V$ be integers such that $\frac{V-1}{V} n \leq n_t \leq n-1$. Let $(I_i)_{1 \leq i \leq V}$ be a collection of disjoint subsets of $\{1,\ldots,n\}$ of equal cardinality $|I_i| = n - n_t$. For all $i \in \{1,\ldots,V\}$, let $T_i = \{1,\ldots,n \} \backslash I_i$. Let $\cT = (T_1,\ldots,T_V)$.
The "incomplete" V-fold CV risk estimator is
\[ \CV{\cT}{k} = \frac{1}{V} \sum_{i = 1}^V \HO{T_i}{k}. \]
\end{definition} 

Similarly to the hold-out, the distribution of $\CV{\cT}{\cdot}$ only depends on $n_t,V$ and in the rest of this article, we will denote by $\cT$ any collection $T_1,\ldots,T_V$ which satisfies the assumptions of Definition \ref{3.def_vfcv}.
Compared to standard $V-$fold, the cross-validation scheme defined above retains the constraint that the $I_i$ be disjoint and of equal size, but decouples the size of the test sets $|I_i| = n - n_t$ from the number of splits $V$. In particular, the hold-out (Definition \ref{3.def_ho}) is a special case of Definition \ref{3.def_vfcv} (for $V = 1$). 
As the collection $(I_i)_{1 \leq i \leq V}$ may be "completed" into a partition by adding sets $I_j$, we shall call $\CV{\cT}{k}$ "incomplete $V-$fold cross-validation".

Conditionally on $D_n^{T^c}$, $\HO{T}{k}$ is an unbiased, consistent estimator of 
$\Norm{\ERMtrigo{k}{T} - s}^2$, which "converges" to $\Norm{\ERMtrigo{k}{T} - s}^2$ at rate $\frac{1}{\sqrt{n - n_t}}$  by the central-limit theorem. $\Norm{\ERMtrigo{k}{T} - s}^2$ itself is known to concentrate around its expectation \cite[Lemma 14]{Arl_Ler:2012:penVF:JMLR}, so $\HO{T}{k}$ does too. $\CV{\cT}{k}$ is an unbiased estimator of $\mathbb{E} \left[ \Norm{\ERMtrigo{k}{T} - s}^2 \right]$, and concentrates at least as fast as $\HO{T}{k}$ by Jensen's inequality.

\subsection{What CV estimates: the oracle}
Since the purpose of cross-validation
is to select the parameter $k$, we want to understand
the behavior of $\HO{T}{\cdot}, \CV{\cT}{\cdot}$ in a region where their optima $\hat{k}_T^{ho}, \hat{k}^{cv}_{\cT}$ can be found 
with high probability. 
 The consistency and unbiasedness of $\HO{T}{k}, \CV{\cT}{k}$ suggest that $\hat{k}_T^{ho}, \hat{k}^{cv}_{\cT}$ should be "close" to $\argmin_{k \in \mathbb{N}} \mathbb{E} \left[\Norm{\ERMtrigo{k}{T} - s}^2\right]$, the 
"optimal" parameter (for a sample of size $n_t$).

Under a few conditions, it is possible to give a simple, deterministic approximant for this optimal parameter.  
More precisely, let $n_t(n)$ be a sequence of integers such that, for all $n \in \mathbb{N}$, $1 \leq n_t(n) \leq n$, and define $n_v(n) = n - n_t(n)$. In the following, we shall denote $n_t = n_t(n)$ and $n_v = n_v(n)$ for a generic value of $n$. Whenever $n, n_v,n_t$ appear in the same expression, it will be understood that $n_t = n_t(n)$ and $n_v = n_v(n) = n - n_t(n)$.

For any $j \in \mathbb{N}$, let $\theta_j = \langle s, \psi_j \rangle$ denote the Fourier coefficients 
of $s$ on the cosine basis. The expected $L^2$ risk can be approximated as follows (see also claim \ref{3.claim_approx_ex_risk}):
\[ \Norm{\ERMtrigo{k}{T} - s}^2 \sim  \frac{k}{n_t} + \sum_{j = k+1}^{+ \infty} \theta_j^2.  \] 
If the squared Fourier coefficients $\theta_j^2$ form a non-increasing sequence, then the approximating function $k \mapsto \frac{k}{n_t} + \sum_{j = k+1}^{+ \infty} \theta_j^2$ is convex and, in particular, has a unique minimizer $k_*(n_t)$. As a further consequence, the level sets of "near-optimal" values of $k$ form a nested collection of intervals. These properties greatly simplify the analysis of the hold-out procedure by avoiding situations where the hold-out "jumps" between two widely separated regions. For this reason, in the remainder of the article, we shall always assume that the sequence $\theta_j^2$ is non-increasing. We will discuss later how this assumption might be relaxed.

Assuming now that $\theta_j^2$ is non-decreasing, it is possible to give simple approximate formulas for the argmin and minimum of the true risk, $\Norm{\ERMtrigo{k}{T} - s}^2$ and its expectation.

\begin{definition} \label{3.def_or}
For all $n \in \mathbb{N}$, let 
 \begin{align*}  
 k_*(n) &= \max \left\{k \in \mathbb{N} : \theta_k^2 \geq \frac{1}{n} \right\} \\
\text{and} \qquad 
\oracle(n) &= \inf_{k \in \mathbb{N}} \left\{ \sum_{j = k+1}^{+ \infty} \theta_j^2 + \frac{k}{n} \right\} . 
  \end{align*}
Equivalently, \begin{align*}
    k_*(n) &= \max \argmin_{k \in \mathbb{N}} \left\{ \sum_{j = k+1}^{+ \infty} \theta_j^2 + \frac{k}{n} \right\}
    \\
\text{and} \qquad     \oracle(n) &= \sum_{j = k_*(n)+1}^{+ \infty} \theta_j^2 + \frac{k_*(n)}{n} 
    \, . 
    \end{align*}
\end{definition}


$k_*(n_t)$ and $\oracle(n_t)$ are thus, approximately, the minimizer and the minimum in $k$ of the $L^2$ risk of the estimators $\ERMtrigo{k}{T}$, 
which explains the name $\oracle(n_t)$ (oracle). 
Thus, it is to be expected that the minima of $\HO{T}{k}, \CV{\cT}{k}$ lie close to $k_*(n_t)$. For this reason, 
$k_*(n_t)$ will be the most relevant value of $k_*$ in the following, and we will often omit the argument $n_t$, with the 
understanding that $k_* = k_*(n_t)$.

Since $\CV{\cT}{\cdot}$ can be expressed as an average of $\HO{T_i}{\cdot}$, we first focus on analyzing the hold-out risk estimator $\HO{T}{\cdot}$. Consequences for cross-validation will be derived in section \ref{sec.thm_cv} .
Assuming to simplify that 
$k_*$ minimizes the $L^2$ risk,
\begin{equation} 
   \HO{T}{k} - \HO{T}{k_*(n_t)} =  \Norm{\ERMtrigo{k}{T} - s}^2 - \Norm{\ERMtrigo{k_*(n_t)}{T} - s}^2 - 2(P_n^{T^c} - P)(\ERMtrigo{k}{T} - \ERMtrigo{k_*(n_t)}{T}) 
\end{equation}
is the sum of a non-negative term (the excess risk) and a centered empirical process. Both tend to $0$ as $k$ tends to $k_*$. 

\subsection{Scaling} \label{3.subsec.scaling}
The relevant scale at which to study the hold-out procedure, which minimizes $\HO{T}{k}$, is the scale of the fluctuations $\hat{k} - k_*(n_t)$ of the argmin $\hat{k}$ of $\HO{T}{k}$. Since the asymptotic study of $\HO{T}{\cdot}$ precedes that of $\hat{k}$, the correct scaling must be deduced \emph{a priori}. 

 Consider the following generic scaled and centered hold-out process:
 \begin{equation} \label{3.eq_dvp_ho_estim_risk}
 \begin{split}
     \frac{1}{\Deltal} \left( \HO{T}{k_* + \alpha \Deltak} - \HO{T}{k_*(n_t)} \right) 
    &= \frac{1}{\Deltal} \left(\Norm{\ERMtrigo{k_* + \alpha \Deltak}{T} - s}^2 - \Norm{\ERMtrigo{k_*(n_t)}{T} - s}^2 \right) \\ 
    &\quad - \frac{2}{\Deltal} (P_n^{T^c} - P)(\ERMtrigo{k_* + \alpha \Deltak}{T} - \ERMtrigo{k_*(n_t)}{T}), 
 \end{split}
 \end{equation}
 where $\alpha \in \{\tfrac{k - k_*}{\Deltak}: k \in \mathbb{N} \}$ and $\Deltak, \Deltal$ are values which may depend on $s,n$ and $n_t$. 
 The relative size of the excess risk and the empirical process in \eqref{3.eq_dvp_ho_estim_risk} depends on $\Deltak$. If the excess risk is much larger than the empirical process in equation \eqref{3.eq_dvp_ho_estim_risk}, then the argmin of the process will concentrate around $0$, which implies that $|\hat{k} - k_*| = o(\Deltak)$: $\Deltak$ is then too large. In contrast, if the centered empirical process is dominant in equation \eqref{3.eq_dvp_ho_estim_risk}, then the argmin diverges, since the variance of $(P_n^{T^c} - P)(\ERMtrigo{k}{T} - \ERMtrigo{k_*}{T})$ grows with $|k - k_*|$. This means that $\Deltak = o(|\hat{k} - k_*|)$, so $\Deltak$ is too small.
 Thus, $\Deltak$ should be chosen based on the following principle:
 \emph{The correct scaling $\Deltak$ for $|\hat{k} - k_*|$ is such that the excess risk $\Norm{\ERMtrigo{k_* \pm \Deltak}{T} - s}^2 - \Norm{\ERMtrigo{k_*}{T} - s}^2$ and the centered empirical process $2(P_n^{T^c} - P)(\ERMtrigo{k_* \pm \Deltak}{T} - \ERMtrigo{k_*}{T})$ have the same order of magnitude.}
  In other words, the bias of the rescaled process should be of the same order of magnitude as its standard deviation. 
 $\Deltal$ should then be chosen so that bias and standard deviation both remain of order $1$, in order to avoid divergence of the scaled process or convergence to $0$ (which would be uninformative).
 
 The appropriate choice of $\Deltak, \Deltal$ is given in the following definition.
%
\begin{definition} \label{3.def_odgs}
For all $n \in \mathbb{N}$, let 
 \begin{align*}  
  \Deltak_d(s, n_t,n) &= \max \left\{l \in \mathbb{N}: \theta_{k_*(n_t) + l}^2 \geq 
  \left[1 - \sqrt{\frac{n_t}{n - n_t}} \frac{1}{\sqrt{l}} \right] \frac{1}{n_t} \right\} \\
  \Deltak_g(s,n_t, n) &= \min \left\{l \in \{0,\ldots,k_*(n_t)\} : \theta_{k_*(n_t) - l}^2 \geq 
  \left[1 + \sqrt{\frac{n_t}{n - n_t}} \frac{1}{\sqrt{l}} \right] \frac{1}{n_t}  \right\} \\
  \Deltak (s, n_t, n) &= \max \left( \Deltak_d (s, n_t, n), \Deltak_g(s, n_t, n) \right) \\
  \Deltaor(s, n_t, n) &= \frac{\Deltak(s, n_t, n)}{n_t}. \\
  \Deltal (s, n_t, n) &= \sqrt{\frac{\Deltaor(s, n_t, n)}{n - n_t}} .
 \end{align*}
\end{definition}
Definition \ref{3.def_odgs} also introduces the quantity $\Deltaor(s,n_t,n)$. This quantity appears often in the proofs, so it is helpful to have notation for it; it also has an interpretation as the order of magnitude of the fluctuations in the variance term, 
\[\mathbb{E}\left[\Norm{\ERMtrigo{k}{T} - \mathbb{E}[\ERMtrigo{k}{T}]}^2 \right] - \mathbb{E}\left[\Norm{\ERMtrigo{k_*}{T} - \mathbb{E}[\ERMtrigo{k_*}{T}]}^2 \right], \] 
and the bias term, 
\[ \Norm{\mathbb{E}[\ERMtrigo{k}{T}] - s}^2 - \Norm{\mathbb{E}[\ERMtrigo{k_*}{T}] - s}^2 = - \text{sign}(k-k_*) \sum_{j = k_*\wedge k + 1}^{k\vee k_*} \theta_j^2, \]
of the estimators $\ERMtrigo{k}{T}$, for $k - k_*$ "of order" $\Deltak$ (in a sense to be made precise later).

As the sequence $n_t(n)$ and the density $s$ are considered to be fixed once and for all, 
the notation $\Deltak(s,n_t,n), \Deltaor(s,n_t,n), \Deltal(s,n_t,n)$ will frequently be replaced 
by the abbreviations $\Deltak, \Deltaor, \Deltal$.

Definition \ref{3.def_odgs} does not make clear how large $\Deltak, \Deltaor$ and $\Deltal$ are. 
Their order of magnitude may depend on the sequence $(\theta_j)_{j \in \mathbb{N}}$ of Fourier coefficients of $s$ as well as on $n_t(n)$. However, the following inequalities always hold.
\begin{lemma} \label{3.lem_odgs}
 For any density $s$ such that the sequence  $\theta_j^2 = \langle s, \psi_j \rangle^2$ is non-increasing,
 \begin{align}
  \Deltak &\geq \frac{n_t}{n - n_t} \\
  \Deltaor &\geq \frac{1}{n - n_t} \\
  \Deltal &\geq \frac{1}{n - n_t} \label{inlem_lb_deltal} \\
  \Deltal &\leq \Deltaor \label{inlem_comp_deltal_deltaor}\\
  \Deltaor &\leq 2 \oracle(n_t) + \frac{1}{n - n_t} \label{inlem_ub_deltaor}.
 \end{align}
\end{lemma}

This lemma is proved in section \ref{3.sec_dem_lem_odgs}. 
The following two examples show that in extreme cases, lemma \ref{3.lem_odgs} may be optimal, at least up to constants.

\paragraph{Two examples}
\begin{itemize}
 \item Let $n_t(n)$ and $u_n$ be two integer sequences, such that $\frac{n_t(n)}{n} \to 1$, $u_n \to + \infty$ and $u_n \leq \frac{\sqrt{n}}{2}$ for all $n$. 
 Assume also that $\frac{n - n_t}{n_t} = o(\tfrac{u_{n_t}}{n_t})$.
 Let for all $j \in \mathbb{N}$
 \begin{equation} \label{ex_s_plate}
  \theta_{j,n}^2 = \begin{cases}
                  & 1 \text{ if } j = 0 \\
                  & \frac{1}{n_t} \text{ if } 1 \leq j \leq u_{n_t}\\
                  & 0 \text{ if } j \geq u_{n_t} + 1,
                 \end{cases}
 \end{equation}
 corresponding for example to the pdf $s_n = 1 + \sum_{j = 1}^{u_{n_t}} \sqrt{\frac{1}{n_t}} \psi_j$.
 Remark that equation \eqref{ex_s_plate} implies that $k_*(n_t) = u_{n_t}$.
Then as $n \to + \infty$, $\Deltaor(s_n,n_t,n) \sim \frac{u_{n_t}}{n_t} \sim \oracle(n_t)$ and $\frac{n - n_t}{n_t} = o(\oracle(n_t))$, so $\Deltal(s_n,n_t,n) = o(\Deltaor(s_n,n_t,n))$.
\item Let $s$ be the pdf associated with the Fourier coefficients 
\begin{equation} \label{ex_s_exp}
\forall j \in \mathbb{N}, \langle s, \psi_j \rangle = \theta_j = \frac{1}{3^j}.
\end{equation}
Let $n_t(n)$ be a sequence of integers 
such that $\frac{n_t(n)}{n} \to 1$.
Then by Lemma \ref{3.lem_odgs}, $\Deltak \geq \frac{n_t}{n - n_t}$, but as 
\[ 9^{- \frac{n_t}{n - n_t}} = o \left(1 - \sqrt{\frac{1}{1 + \frac{n - n_t}{n_t}}} \right),\]
it follows that $\Deltak(s,n_t,n) \sim \frac{n_t}{n - n_t}$, hence 
\[ \Deltaor(s,n_t,n) \sim \frac{n_t}{(n-n_t) n_t} \sim \frac{1}{n - n_t}. \]
As a result, $\Deltaor(s,n_t,n) \sim \Deltal(s,n_t,n) \sim \frac{1}{n - n_t}$, and this for \emph{any} sequence $n_t(n)$ 
such that $n_t(n) \sim n$. 
\end{itemize}

Now that $\Deltak, \Deltal$ are defined, the hold-out process can be rescaled as in equation \eqref{3.def_ho_scaled}. More precisely, the rescaled hold-out process is given by Definition~\ref{3.def_ho_scaled} below.

\begin{definition} \label{3.def_ho_scaled}
 For all $j \in [-k_*; + \infty[ \cap \mathbb{Z}$, let
\[ \ho{T} \left( \frac{j}{\Deltak} \right) =  \frac{1}{\Deltal} \left(\HO{T}{k_* + j} - \HO{T}{k_*} \right),\]
in other words (by definition \ref{3.def_ho})
 \[ \ho{T} \left( \frac{j}{\Deltak} \right) = \frac{1}{\Deltal} \left( \Norm{\ERMtrigo{k_* + j}{T} - s}^2 - \Norm{\ERMtrigo{k_*}{T} - s}^2 \right) 
 -  \frac{2}{\Deltal} \left(P_n^{T^c} - P \right) \left( \ERMtrigo{k_* + j}{T} - \ERMtrigo{k_*}{T} \right). \]
 The $\ho{T}$ function is extended by linear interpolation to all $\alpha \in \left[ \frac{-k_*(n_t)}{\Deltak}; +\infty \right[$.
 Let $\cv$ be defined in a similar manner, i.e
 \[ \cv(\tfrac{j}{\Deltak}) = \frac{1}{\Deltal} \left(\CV{\cT}{k_* + j} - \CV{\cT}{k_*} \right) \]
 for all $j \in [-k_*; + \infty[ \cap \mathbb{Z}$,
 extended by linear interpolation to $\left[ \frac{-k_*(n_t)}{\Deltak}; +\infty \right[$.
 Note that by linearity of the interpolation operation,
 $\cv = \frac{1}{V} \sum_{i = 1}^V \ho{T_i}$.
\end{definition}

The extension of $\ho{T}, \CV{\cT}$ by linear interpolation simplifies their approximation by a continuous process. 
Notice that any minimizer of $\ho{T}$ (resp. $\CV{\cT}$) on the grid $\frac{1}{\Deltak} \left([-k_*(n_t);+ \infty[ \cap \mathbb{Z} \right)$ remains a minimizer of $\ho{T}$ (resp. $\CV{\cT}$) on the interval $\left[ \frac{-k_*(n_t)}{\Deltak}; +\infty \right[$.
In particular, this applies to the hold-out parameter obtained by minimisation of the hold-out risk estimator.

The process $\ho{T}$ can be expressed as the sum of the standardized excess risk, 
\[ \frac{1}{\Deltal} \left( \Norm{\ERMtrigo{k_* + j}{T} - s}^2 - \Norm{\ERMtrigo{k_*}{T} - s}^2 \right) \]
and a centered empirical process: $\frac{2}{\Deltal} \left(P_n^{T^c} - P \right) \left( \ERMtrigo{k_* + j}{T} - \ERMtrigo{k_*}{T} \right)$.
Though the excess risk is a priori random (it depends on $D_n^T$), the proof will show that it concentrates around a deterministic function $f_n$, depending on $n$, which is given by 
definition \ref{3.def_fn} below.

\begin{definition} \label{3.def_fn}
 For all $k \in \mathbb{N}$, let $R(k) = \sum_{j = k + 1}^{+ \infty} \theta_j^2$.
 Extend $R$ to $\mathbb{R}_+$ by linear interpolation: 
 \[ \forall x \in \mathbb{R}_+, 
 R(x) = (1 + \lfloor x \rfloor - x) R(\lfloor x \rfloor) + (x - \lfloor x \rfloor) R(\lfloor x \rfloor + 1). \]
 $f_n: ]-\tfrac{k_*(n_t)}{\Deltak}; +\infty[ \rightarrow \mathbb{R}_+$ is now defined by:
 \begin{equation}
  f_n(\alpha) = \frac{1}{\Deltal} \left( R(k_*(n_t) + \alpha \Deltak) - R(k_*(n_t)) + \frac{\alpha \Deltak}{n_t} \right).
 \end{equation}
Thus, for all $k \in \mathbb{N}$, $k \neq k_*(n_t)$,
\begin{equation} \label{3.indef_fn_abs}
 \Deltal f_n \left( \frac{k - k_*(n_t)}{\Deltak} \right) = \sum_{j = k \wedge k_*(n_t) + 1}^{k \vee k_*(n_t)} \left| \theta_j^2 - \frac{1}{n_t} \right|. 
\end{equation}
\end{definition}
It is clear by equation \eqref{3.indef_fn_abs} that $f_n$ reaches its minimum at $0$, moreover the assumption that the sequence $(\theta_j^2)_{j \in \mathbb{N}}$ is non-increasing implies that $f_n$ is convex. In particular, $f_n$ is non-increasing on $]-\tfrac{k_*(n_t)}{\Deltak};0]$
and non-decreasing on $[0;+\infty[$. Moreover, the definition of $\Deltak$ and $\Deltal$ (Definition \ref{3.def_odgs}) implies the following bounds on the increments of $f_n$:
\begin{lemma} \label{3.claim_bd_diff_fn}
For any $\alpha_1, \alpha_2 \in \mathbb{R}$ such that $\alpha_1 \alpha_2 \geq 0$ and $|\alpha_2| \geq |\alpha_1| \geq 1$,
\[ f_n(\alpha_2) - f_n(\alpha_1) \geq |\alpha_2| - |\alpha_1|. \]
In particular, since $f_n(0) = 0$, for all $\alpha \in \mathbb{R}$,
\[ f_n(\alpha) \geq (|\alpha| - 1)_+. \]
Moreover, using the notation from Definition \ref{3.def_odgs},
 \begin{itemize}
  \item If $\Deltak = \Deltak_d$, then for any $\alpha_1, \alpha_2 \in [0;1]$ such that $\alpha_1 \leq \alpha_2$,
  $f_n(\alpha_2) - f_n(\alpha_1) \leq \alpha_2 - \alpha_1$.
  \item If $\Deltak = \Deltak_g$, then for any $\alpha_1, \alpha_2 \in [-1;0]$ such that $\alpha_1 \leq \alpha_2$,
  $f_n(\alpha_1) - f_n(\alpha_2) \leq \alpha_2 - \alpha_1$.
 \end{itemize}
\end{lemma}
This lemma is proved in section \ref{3.sec_dem_bd_diff_fn}.
It guarantees that $f_n$ remains in a sense of "finite order" and "non zero" as $n \to +\infty$, which means that
$f_n$ remains uniformly bounded on $[-1;0]$ or on $[0;1]$, and is lower-bounded on $\mathbb{R}$ by the non-zero function $(|x|-1)_+$.

\section{Hypotheses} \label{3.sec.hyp}
The main hypothesis of this article is that the squared Fourier coefficients $\theta_j^2$ are non-increasing, which is admittedly strong, though a natural assumption in the context of this article, as discussed above. It seems likely that the desirable effects of this hypothesis can be retained under weaker conditions, as we will discuss later (section \ref{3.subsec.disc_hyp}). 

In addition to this "shape constraint" on the $\theta_j^2$ , the main theorem of this article requires a number of more technical assumptions. First, approximating the process $HO_T(k)$, which is a sum over Fourier coefficients, inevitably involves bounding "tail sums" of the Fourier series, such as $\sum_{j = k+1} \theta_j^2$. To control these, we assume that the Fourier coefficients decay fast enough.

\begin{hyp} \label{3.inthm_hyp_ub_sum_varphi}
There exists constants $c_1 \geq 0$ and $\delta_1 \geq 0$
 such that for all $k \in \mathbb{N}$, $\sum_{j = k+1}^{+ \infty} \theta_j^2 \leq \frac{c_1}{k^{2 + \delta_1}}$. 
\end{hyp}

This upper bound is satisfied for some $c_1,\delta_1$ if and only if the smoothness assumption $s \in H^{\beta}$ holds for some $\beta > 1$, where $H^\beta$ denotes the Sobolev Hilbert space. This implies in particular that
($k_*(n_t) \leq n_t$) is satisfied for all sufficiently large $n_t$ (thus also for all large enough $n$).

Secondly, since we seek to approximate the discrete process $HO_T(k)$ by a continuous one, it is necessary to make sure that the number of "points" $k \in [k_* - \alpha \Deltak, k_* + \alpha \Deltak]$ tends to infinity as $n \to + \infty$, for all $\alpha$. Similarly, one must assume that $f_n(\tfrac{j}{\Deltak}) \sim f_n(\tfrac{j-1}{\Deltak})$, otherwise the continuous function $f_n$ is not sufficiently close to its discrete version $j \mapsto f_n(\tfrac{j}{\Deltak})$. For technical reasons, we will assume a polynomial growth rate, using the following three hypotheses.

\begin{hyp} \label{3.inthm_hyp_lb_sum_varphi}
There exists constants $c_2 \geq 0$, $\delta_2 \geq 0$ such that
 for all $k \in \mathbb{N}$, $\sum_{j = k+1}^{+ \infty} \theta_j^2 \geq \frac{c_2}{k^{\delta_2}}$  
\end{hyp}

This hypothesis states that the Fourier coefficients $\theta_j^2$ cannot decay faster than polynomially, and excludes in particular analytic functions. This guarantees that $k_*(n_t) \to +\infty$ at a polynomial rate. Hypothesis \ref{3.inthm_hyp_ub_sum_varphi} holds for example if $s$ or one of its derivatives has a point of discontinuity.

\begin{hyp} \label{3.inthm_hyp_lb_rapport_varphi}
There exists constants $c_3 > 0$, $\delta_3 > 0$ such that for all $k \geq 1$,
\[ \theta_{k + k^{\delta_3}}^2 \geq c_3 \theta_{k - k^{\delta_3}}^2. \]  
\end{hyp}

Hypothesis \ref{3.inthm_hyp_lb_rapport_varphi} means that the sequence $\theta_j^2$ cannot decrease too abruptly, excluding in particular a locally exponential decrease such as $\theta_{k_n + j}^2 = 2^{-j}w_n$, for $j \in \{1,\ldots, \varepsilon \log k_n\}$ and $k_n \to + \infty$. 
Without this hypothesis, $f_n$ may have "asymptotically sharp" discontinuities at $\pm 1$, violating the condition $f_n(\tfrac{j}{\Deltak}) \sim f_n(\tfrac{j-1}{\Deltak})$, see claim \ref{3.claim_bd_diff_fn} and example \eqref{ex_s_exp} .
Hypothesis \ref{3.inthm_hyp_lb_rapport_varphi} is satisfied by polynomially decreasing sequences, $ \theta_j^2 = \kappa j^{-\beta}$, with $\delta_3 = 1$, but also by sequences $\theta_j^2  = \kappa \exp(- j^\alpha)$, as long as $\alpha < 1$. Locally, $\theta_j^2$ can thus decrease much faster than the polynomial lower bound given by hypothesis \ref{3.inthm_hyp_lb_sum_varphi}. 

Together, hypotheses \ref{3.inthm_hyp_ub_sum_varphi}, \ref{3.inthm_hyp_lb_sum_varphi}, \ref{3.inthm_hyp_lb_rapport_varphi}
basically mean that the Fourier coefficients of $s$ on the cosine basis decrease polynomially. For example, they are satisfied if there are two strictly positive constants $\mu, L$ 
and a constant $\beta > 1$, such that
\[ \forall k \in \mathbb{N}, \mu k^{-2\beta} \leq \sum_{j = k+1}^{+ \infty} \theta_j^2 \leq L k^{- 2 \beta}. \]

The two remaining hypotheses \ref{3.inthm_hyp_ub_nv} and \ref{3.inthm_hyp_lb_nv} do not bear on $s$, 
 but on the parameter $n_t$ which is chosen by the statistician. They serve to establish upper and lower bounds on $\Deltak$ using lemma \ref{3.lem_odgs}.
 
\begin{hyp} \label{3.inthm_hyp_ub_nv}
There exists a constant $\delta_4 > 0$ such that $n - n_t \leq n^{1-\delta_4}$. 
\end{hyp}

By lemma \ref{3.lem_odgs}, hypothesis \ref{3.inthm_hyp_ub_nv} guarantees that $\Deltak$ grows at least at a polynomial rate $n^{\delta_4}$, as announced above. For technical reasons, it is also necessary to upper bound $\Deltak$, which is accomplished using hypothesis \ref{3.inthm_hyp_lb_nv} below.

\begin{hyp} \label{3.inthm_hyp_lb_nv}
There exists a constant $\delta_5 > 0$ such that $n_v = n - n_t \geq n^{\frac{2}{3} + \delta_5}$. 
\end{hyp}

The statistician can always choose $n_t$ such that hypotheses \ref{3.inthm_hyp_ub_nv} and \ref{3.inthm_hyp_lb_nv} hold.
One should however check that this is compatible with good performance of the hold-out. The oracle inequalities of \citet{Arl_Ler:2012:penVF:JMLR} show that the risk of the hold-out in model selection for $L^2$ density estimation is (at most) of order
 $\frac{n}{n_t} \oracle(n) + \frac{\log (n - n_t)}{n - n_t}$. 
 If $\oracle(n)$ decreases in $n$ with rate $\frac{1}{n^\alpha}$ ($\alpha \in (0,1)$), 
 which is the case under assumptions \ref{3.inthm_hyp_ub_sum_varphi} and \ref{3.inthm_hyp_lb_sum_varphi}, $n - n_t$ can be chosen within the interval $[\tfrac{1}{2} n^{\frac{2 + \alpha}{3}}; n^{\frac{4 + \alpha}{5}}] $ 
 ---so that assumptions \ref{3.inthm_hyp_ub_nv} and \ref{3.inthm_hyp_lb_nv} are satisfied--- 
 without changing the order of magnitude of the risk.


\section{Main theorems}
The purpose of this article is to find simple approximants $Y_{n,V}(u)$ to the rescaled CV estimators $\cv(u)$ (including the hold-out when the number of splits $V = 1$). 

\subsection{Domain of approximation}
In order to later derive results about the argmin $\hat{k}^{cv}_{\cT}$ selected by CV, a uniform approximation is desirable; however it is unrealistic to expect a uniform approximation to the unbounded processes $\cv$ on the whole real line. Instead, it is sufficient to uniformly approximate $\cv$ on compact sets which contain the argmin with high probability. At first order, the probability for the argmin to lie in a given region should depend primarily on the size of $f_n$, which approximates the rescaled excess risk: we cannot expect $\hat{k}^{cv}_{\cT}$ to lie where the excess risk is large, even if $|\hat{k}^{cv}_{\cT} - k_*|$ is small. From a technical viewpoint, the size of the coefficients $\theta_j^2$ matters for bounding the approximation error, and that can be taken into account through $f_n(\alpha)$.  For these reasons, the "natural" sets on which to approximate $\cv$ are the intervals $[a_x,b_x]$ defined below.

\begin{definition}
Let $T \subset \{1 \ldots n\}$ be a subset of cardinality $n_t$ and let $k_* = k_*(n_t)$.
 For any $x > 0$, let 
 \begin{align*}
  a_x &= \min \{ \tfrac{j}{\Deltak} : j \in \{-k_*(n_t), \ldots, 0\}, f_n(\tfrac{j}{\Deltak}) \leq x\} \\
  b_x &= \max \{ \tfrac{j}{\Deltak} : j \in \mathbb{N}, f_n(\tfrac{j}{\Deltak}) \leq x\}. 
 \end{align*}
\end{definition}
Up to the constraint that $a_x,b_x \in \tfrac{1}{\Deltak} \mathbb{Z}$, the sets $[a_x,b_x]$ are just the sublevel sets of $f_n$. Moreover, since $\hat{k}^{cv}_{\cT}  - k_* \in \mathbb{Z}$, $\frac{\hat{k}^{cv}_{\cT} - k_*}{\Deltak} \in [a_x,b_x]$ if and only if $f_n \left( \frac{\hat{k}^{cv}_{\cT} - k_*}{\Deltak} \right) \leq x$. 

By lemma \ref{3.claim_bd_diff_fn}, for $x \geq 1$, $[a_x,b_x]$ either contains $[0,1]$ (if $\Deltak = \Deltak_d$) or $[-1,0]$ (if $\Deltak = \Deltak_g$). On the other hand, by lemma \ref{3.lem_ub_fen_risk}, $b_x - a_x \leq 2(1+x)$. This shows that the definition of $a_x,b_x$ is consistent with the choice to center $\CV{\cT}{\cdot}$ at $k_*$ and rescale by $\Deltak$.

\subsection{Simple validation} \label{3.sec.thm_ho}
The following theorem shows that the process $\ho{T}(\cdot)$ can be approximated on $[a_x,b_x]$ by the sum of $f_n$ and a time-changed Brownian motion. 

\begin{theorem} \label{3.thm_approx_ho}
Assume that the assumptions of section \ref{3.sec.hyp} hold.
 There exists a non-decreasing function $g_n: [- \frac{k_*(n_t)}{\Deltak}; + \infty[ \to \mathbb{R}$ 
and for any $x > 0$, there exists a two-sided Brownian motion $(W_t)_{t \in [a_x;b_x]}$ independent from $D_n^T$ 
such that, for any $y > 0$, with probability greater than $1 - e^{-y}$,
 \begin{equation} \label{3.inthm_approx_ho} 
  \mathbb{E} \left[ \sup_{u \in [a_x;b_x]} \left| \ho{T}(u) - (f_n(u) - W_{g_n(u)}) \right| \,\bigg\vert\, D_n^T \right] \leq \kappa_0 (1+y)^2 (1+x)^{\frac{3}{2}}  n^{-u_1},
 \end{equation}
 where $u_1 > 0$ and $\kappa_0 \geq 0$ are two constants which depend only on $\delta_1,\delta_3,\delta_5,\delta_2$ 
 and $c_1,c_2,\delta_1,c_3$, respectively.
 Moreover, $g_n$ and $W$ can be chosen so as to satisfy the following conditions.
  \begin{enumerate}
  \item $g_n(0) = 0$, $W_0 = 0$,
  \item $\forall (\alpha_1, \alpha_2) \in \left[ \frac{-k_*}{\Deltak}; + \infty \right[^2$, 
  $\alpha_2 < \alpha_1 \implies g_n(\alpha_1) - g_n(\alpha_2) \geq 4\Norm{s}^2 [\alpha_1 - \alpha_2]$. \label{3.inthm_lb_diff_gn}
  \item For all $(\alpha_1, \alpha_2) \in \left[ \frac{-k_*}{\Deltak}; + \infty \right[^2$ such that
  $\alpha_1 < \alpha_2 < 0$ or $0 < \alpha_1 < \alpha_2$,
  \begin{equation} \label{3.eq.inthm_ub_diff_gn} g_n(\alpha_2) - g_n(\alpha_1) \leq - \frac{8 \NormInfinity{s}}{(n-n_t) \Deltal}  [f_n(\alpha_2) - f_n(\alpha_1)] 
  + \left(8\NormInfinity{s} + 4\Norm{s}^2 \right) [\alpha_2 - \alpha_1].\end{equation} \label{3.inthm_ub_diff_gn}
 \end{enumerate}
\end{theorem}
This theorem is proved in section \ref{3.sec_dem_thm}. It states that the rescaled hold-out process, $\ho{T}$, can be approximated uniformly in expectation on $[a_x,b_x]$ by a continuous process $Y_{n,1}$,which is the sum of a convex non-negative function $f_n$ and a time-changed Brownian motion $W_{g_n}$. 
$f_n$ and $g_n$ depend on $n_t$ and $n$, but not on the data (they are deterministic functions), 
while $W$ depends on the data only through the test sample $D_n^{T^c}$. In particular, in this asymptotic setting, $\ho{T}$ doesn't depend on $D_n^T$, the training data.

The function $g_n$ increases on its domain and has a Lipschitz-continuous inverse. By lemma \ref{3.claim_bd_diff_fn} and equation \eqref{3.inthm_ub_diff_gn}, $f_n,g_n$ are Lipschitz continuous either on $[-1,0]$ or on $[0,1]$ (depending on whether $\Deltak = \Deltak_d$ or $\Deltak_g = \Deltak$), with Lipschitz constants that depends only on $\Norm{s}^2, \Norm{s}_{\infty}$. 
In particular, $f_n, g_n$ are both of constant order on $[a_x,b_x]$ for $x > 1$. Figure \ref{3.rep_f_g_bds} illustrates the bounds that hold on $f_n,g_n$ in a situation where $\Deltak = \Deltak_d$.

\begin{figure} \begin{center}
 \includegraphics[scale = 0.5]{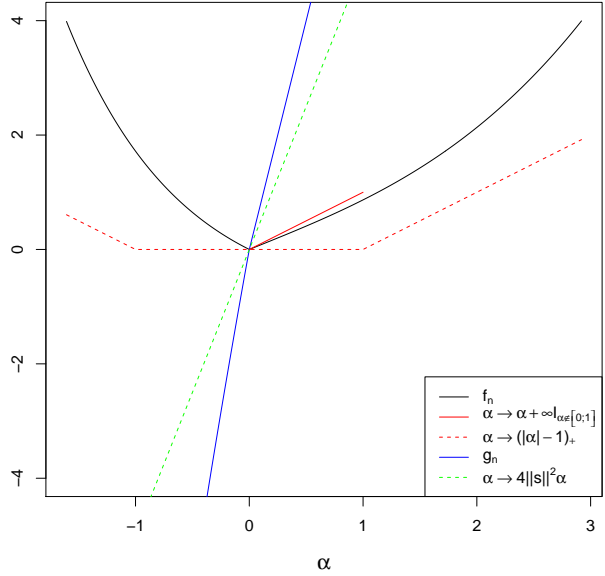}
 \caption{A plot of $f_n, g_n$ on $[a_6;b_6]$ with upper and lower bounds, for $\Norm{s}^2 = 1.2$. \label{3.rep_f_g_bds}}
 \end{center}
\end{figure}

On the one hand, by definition of $a_x,b_x$, $0 \leq f_n \leq x$ on $[a_x,b_x]$ and by equation \eqref{3.inthm_ub_diff_gn}, 
\begin{equation}\label{3.inthm_ub_gn}
\sup_{\alpha \in [a_x,b_x]} |g_n(\alpha)| \leq 8\NormInfinity{s}x + \left(8\NormInfinity{s} + 4\Norm{s}^2 \right) \max(|a_x|,|b_x|) \leq 20 \NormInfinity{s}(1+x)
\end{equation}
On the other hand, $|g_n(\alpha)| \geq 4\Norm{s}^2 |\alpha|$ by equation \eqref{3.inthm_lb_diff_gn} and $f_n(b_x) \geq x - o(1)$ by claim \ref{3.lem_ub_varphi_fen} . Thus, the principle discussed in section 2 is satisfied on $[a_x,b_x]$: the mean and standard deviation of $Y_{n,1}$ are both of constant order. 

It remains to see that the intervals $[a_x,b_x]$ are the "largest" on which this is true.
Figure \ref{3.rep_f_g} gives an illustration of the situation for $x = 25$ and 
\begin{align*}
 f_n: \alpha &\mapsto \begin{cases}
                      &e^{- \alpha} - 1 \text{ if } \alpha \leq 0 \\
                      & \frac{8}{10} \alpha + \frac{8}{30} \alpha^3 \text{ if } \alpha \geq 0
                     \end{cases} \\
 g_n: \alpha &\mapsto \begin{cases}
                 &7.8 \alpha \text{ if } \alpha \geq 0 \\
                 &7.8 \alpha - 3 f_n(\alpha) \text{ if } \alpha \leq 0              
                   \end{cases}
\end{align*}
(which satisfy the properties of lemma \ref{3.claim_bd_diff_fn} and Theorem \ref{3.thm_approx_ho} when $\Norm{s}^2 \leq 1.2$ and  $\NormInfinity{s} \leq 1.5$).

\begin{figure}\begin{center}
 \includegraphics[width=.7\textwidth]{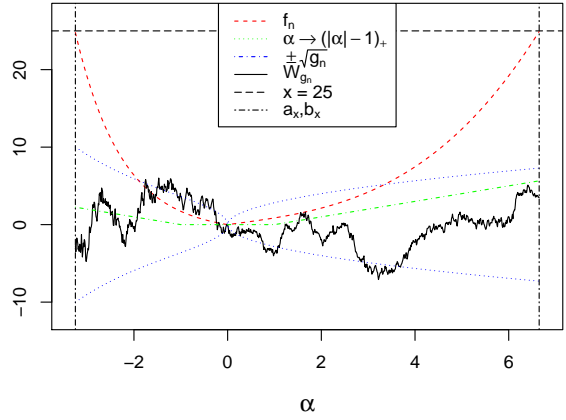}
 \caption{ \label{3.rep_f_g} A plot of $f_n,W_{g_n}$ on $[a_x;b_x]$, for $x = 25$, $g_n: \alpha \mapsto 7.8 \alpha - 3 f_n(\alpha) \mathbb{I}_{\alpha < 0}$.}
 \end{center}
\end{figure}

Figure \ref{3.rep_f_g} suggests that for large $x$, $\sqrt{g_n}$ and hence $W_{g_n}$ become negligible compared to $f_n$ 
outside the interval $[a_x,b_x]$. This intuition can be theoretically justified:
if $\alpha \in \frac{1}{\Deltak} \mathbb{Z}$ does not belong to $[a_x;b_x]$, then $f_n(\alpha) \geq x$ by definition of $a_x,b_x$,
while on the other hand, by equation \eqref{3.eq.inthm_ub_diff_gn}, there exists a constant $\kappa$, depending only on $\NormInfinity{s}, \Norm{s}^2$, such that
\begin{align*}
 \frac{\sqrt{\Var(W_{g_n(\alpha)})}}{f_n(\alpha)} &\leq \frac{\sqrt{g_n(\alpha)}}{f_n(\alpha)} \\
 &\leq \frac{\sqrt{\kappa f_n(\alpha)}}{f_n(\alpha)} + \frac{\sqrt{\kappa |\alpha|}}{f_n(\alpha)}.
\end{align*}
By lemma \ref{3.claim_bd_diff_fn}, $f_n(\alpha) \geq (|\alpha| - 1)_+$ therefore $|\alpha| \leq 2f_n(\alpha)$ whenever $f_n(\alpha) \geq 1$, i.e for $\alpha \notin [a_1;b_1]$. 
This yields:
\[ 
\forall x \geq 1 , \, \forall \alpha \in \frac{1}{\Deltak} \mathbb{Z} \backslash [a_x;b_x] \, , 
\qquad 
\frac{\sqrt{\Var(W_{g_n(\alpha)})}}{f_n(\alpha)} \leq \sqrt{\frac{\kappa}{x}} + \sqrt{\frac{2\kappa}{x}} 
= \frac{\sqrt{\kappa} (1+\sqrt{2})}{\sqrt{x}}
. \]
Hence, for sufficiently large $x$, the random term $W_{g_n}$ becomes negligible relative to the deterministic $f_n$ outside the interval $[a_x;b_x]$. 
\subsection{Incomplete $V-$fold cross-validation} \label{sec.thm_cv}
Since the cross-validation risk estimator $\CV{\cT}{k}$ can be written as an average of hold-out risk estimators $\HO{T_i}{k}$, Theorem \ref{3.thm_approx_ho} has direct implications for CV.

\begin{corollary} \label{3.cor_cv}
Assume that the hypotheses of section \ref{3.sec.hyp} hold.
Let $f_n,g_n$ be as in definition \ref{3.def_fn} and theorem \ref{3.thm_approx_ho}. For any $x > 0$,
\[ \mathbb{E} \left[ \sup_{u \in [a_x;b_x]} \left| \cv(u) - (f_n(u) - W_{\frac{g_n(u)}{V}}) \right| \right] \leq 5 \kappa_0 (1+x)^{\frac{3}{2}}  n^{-u_1}, \]
with the same constants $\kappa_0,u_1$ as in Theorem \ref{3.thm_approx_ho}.
\end{corollary}

\begin{proof}
By integrating the bound of Theorem \ref{3.thm_approx_ho},
\[ \mathbb{E} \left[ \sup_{u \in [a_x;b_x]} \left| \ho{T_i}(u) - (f_n(u) - W^i_{\frac{g_n(u)}{V}}) \right|  \right] \leq 5 \kappa_0 (1+x)^{\frac{3}{2}}  n^{-u_1}  \]
for each $i \in \{1,\ldots,V\}$, where the $W^i$ are symmetrical BMs that are independent of $D_n^{T_i}$. We can construct $W^i$ such that $W^i = H(D_n^{T_i^c},U_i)$, where $H$ is a measurable function and $U_i$ is an auxiliary uniform random variable. Taking independent $U_i$ yields i.i.d $W^i$, since the sets $T_i^c = I_i$ are disjoint. 
Let $\bar{W} = \frac{1}{V} \sum_{i = 1}^V W^i$.
By Jensen's inequality,
\[ \mathbb{E} \left[ \sup_{u \in [a_x;b_x]} \left| \cv(u) - (f_n(u) - \bar{W}_{g_n(u)}) \right| \right] \leq 5 \kappa_0 (1+x)^{\frac{3}{2}}  n^{-u_1}. \]
Conclude by noting that $(\bar{W}_t)_{t \in \mathbb{R}}$ is equal in distribution to $(W_{t/V})_{t \in \mathbb{R}}$, as a continuous random process.
\end{proof}

Corollary \ref{3.cor_cv} proves that cross-validation is effective at reducing the variance of risk estimation, compared to simple validation (the hold-out).
 The process approximating $\cv$ is of the same form as that approximating $\ho{T}$, but with its variance reduced by a factor $V$, as would be the case if the hold-out estimators $\HO{T_i}{\cdot}$ were independant. Importantly, this reduction in variance occurs for the rescaled process $\cv$, and so can be expected to reflect the model selection performance. Corollary \ref{3.cor_cv} is sharp when $V$ is fixed as $n \to +\infty$, since in that case, the approximating process $f_n - W_{g_n/V}$ remains nontrivial (random and of bounded size) as $n \to +\infty$. When $V = V_n \to +\infty$, corollary \ref{3.cor_cv} is still valid, but $\Var(W_{g_n/V_n}) = g_n/V_n \to 0$, which means that $\cv$ concentrates around the deterministic function $f_n$. However, corollary \ref{3.cor_cv} cannot tell us the rate at which this convergence occurs, unless $V_n = o(n^{u_1})$.

\section{Discussion}
In this section, we interpret the results of the article and discuss how they could be extended.

\subsection{Implications of our results}
Theorem \ref{3.thm_approx_ho} provides an approximation to the rescaled hold-out process, with scale factor $\Deltak$ given by Definition \ref{3.def_odgs}. 
The fact that the approximating process is asymptotically random, but trends away from zero at $\pm \infty$ (as discussed at length in section \ref{3.sec.thm_ho}) supports the conjecture that $\hat{k}_T - k_*$ is of order $\Deltak$.  In the case of "incomplete" cross-validation for fixed $V$, the process is of identical type, but with a smaller variance, which strongly suggests better model selection performance. If $V \to + \infty$, then rescaled cross-validation concentrates around the deterministic function $f_n$, which by the argument of section \ref{3.subsec.scaling}, suggests that the scale $\Deltak$ is "asymptotically larger" than the fluctuations $\hat{k}_{\cT}^{cv} - k_*$ of the CV parameter. This supports the belief that cross-validation improves when $V$ is increased.
 Though $k_*(n_t)$ is not the oracle $k_*(n)$, the greater concentration of CV around $k_*(n_t)$ makes it possible to choose $n_t$ closer to $n$ than would be reasonable for the hold-out, resulting in improved overall performance.  

\subsection{Hypotheses} \label{3.subsec.disc_hyp}
Theorem \ref{3.thm_approx_ho} relies on the assumptions of section \ref{3.sec.hyp}, most importantly on the assumption that the sequence of squared Fourier coefficients $\theta_j^2$ is non-increasing. 

This could be weakened in various ways. First, given the local nature of our analysis, what is really required is that the coefficients $\theta_j^2$ be non-increasing in a neighbourhood  $[k_* - r_n, k_* + r_n]$ of $k_*$ of radius $r_n$ which dominates $\Deltak$ ($\Deltak = o(r_n)$).
The only difference in that case is that the selected parameter $\hat{k}_{\cT}^{cv}$ may lie outside $[k_* - r_n, k_* + r_n]$, which diminishes the usefulness of Theorem \ref{3.thm_approx_ho} for studying model selection.

  Moreover, since the process $\cv(\alpha)$ consists of sums $\sum_{j = 1}^{k_* + \alpha \Deltak}$, it is probably sufficient to replace the hypotheses on the individual coefficients $\theta_j^2$ with hypotheses bearing on local averages $ \bar{\theta_j^2} = \frac{1}{2m_n} \sum_{r = j - m_n}^{j + m_n} \theta_r^2$, at some scale $m_n = o(\Deltak)$.  
Depending on the scale $m_n$, the hypothesis that a smoothed sequence $\bar{\theta_j^2}$ is non-decreasing may be quite plausible, considering the fact that the Fourier coefficients tend to $0$ at a prescribed rate for sufficiently smooth functions $s$.

\subsection{Perspectives}

\paragraph{Model selection}
The CV risk estimator is usually used to select a model, $\hat{k}^{cv}_{\cT}$, which minimizes it. The final result of CV is then the estimator $\hat{s}_{\hat{k}^{cv}_{\cT}}$, or $\hat{s}^T_{\hat{k}_T}$ in the case of simple validation. Thus, what we are most interested in practice is the risk $\Norm{\hat{s}_{\hat{k}^{cv}_{\cT}} - s}^2$ of this final estimator, and how it depends on the CV method used (at least when CV is used with a goal of \emph{estimation}, as opposed to \emph{identification} of the best model). There are several ways our results can contribute to answering these questions. 
First, since $\cv(u)$ can be uniformly approximated by $f_n(u) - W_{g_n(u)/V}$, it is natural to approximate $\frac{\hat{k}^{cv}_{\cT} - k_*}{\Deltak}$ (the minimizer of $\cv(u)$) by $\hat{\alpha}_{n,V}$ (the minimizer of $f_n - W_{g_n/V}$). The minima of $f_n - W_{g_n/V}$ can be studied using the theory of Wiener processes. Together with our results about $f_n$ and $g_n$ (in lemma \ref{3.claim_bd_diff_fn} and Theorem \ref{3.thm_approx_ho}), this makes the analysis of $\hat{\alpha}_{n,V}$ much easier than that of $\hat{k}^{cv}_{\cT}$. 
Secondly, claim \ref{3.claim_approx_ex_risk} of this article proves that the (excess) risk $\Norm{\hat{s}_k^T - s}^2 - \Norm{\hat{s}_{k_*}^T - s}^2$ concentrates around $\Deltal f_n \left( \frac{k - k_*}{\Deltak} \right)$ for $k$ "close enough" to $k_*$. This removes the dependency on the training sample $D_n^T$ and reduces the analysis of $\Norm{\hat{s}_k^T - s}^2 - \Norm{\hat{s}_{k_*}^T - s}^2$ to that of the deterministic function $f_n \left( \tfrac{k - k_*}{\Deltak} \right)$.


\paragraph{Other model selection methods}
In this article, we only considered a particular type of cross-validation.
Corollary \ref{3.cor_cv} relies on decomposing "incomplete" $V-$fold CV as a finite average of asymptotically independent hold-out estimators, which can be analysed more easily than general cross-validation because conditionally on their training data, they are empirical processes. 

In general, for projection estimators in $L^2$ density estimation, the cross-validation risk estimator is not a (conditional) empirical process but a (weighted) U-statistic of order $2$ (more precisely, a weighted sum of the terms $\psi_k(X_i) \psi_k(X_j)$). Thus, new methods are required to approximate general CV estimators by gaussian processes. We conjecture that more general cross-validation methods also behave locally like the sum of the (rescaled) excess risk and a time-changed Wiener process, though we expect the scaling and the time-change $g_n$ to be different for different versions of CV.

Theorem \ref{3.thm_approx_ho} can also shed light on the behaviour of other methods which use simple validation as a key ingredient, such as Aggregated hold-out \citep{agghoo_rkhs}. 

\section*{Acknowledgements}
While finishing the writing of this article, the author (Guillaume Maillard) has received funding from the
European Union's Horizon 2020 research and innovation programme under grant agreement No 811017.

\section{Proofs} \label{3.sec_dem_thm}

In this section, the term constant means a function of $\NormInfinity{s}, \Norm{s}^2$ 
and the constants $c_1,c_2, c_3, \delta_1$, 
$\delta_2, \delta_3, \delta_4, \delta_5$. 
which appear in the hypotheses of Theorem \ref{3.thm_approx_ho}.
Note that by hypothesis \eqref{3.inthm_hyp_ub_sum_varphi}, $\Norm{\theta}_{\ell^1},\NormInfinity{s}, \Norm{s}^2$ are finite and can be bounded by functions of $c_1, \delta_1$.
The letter $u$ will denote strictly positive constants that only depend on $(\delta_i)_{1 \leq i \leq 5}$ 
(they will generally appear as exponents of $\frac{1}{n}$). The letter $\kappa$ denotes a non-negative constant. The notation $n_v = n - n_t$ will also be used frequently.

\subsection{Preliminary results}
The results of this section are independent from the rest. 
They will be used in the rest of the proof of Theorem \ref{3.thm_approx_ho}, as well as in the Appendix.
Let's start by proving some basic properties of $a_x,b_x$ and $f_n$ that will be used repeatedly in the main proofs.

\subsubsection{Proof of lemma \ref{3.lem_odgs}} \label{3.sec_dem_lem_odgs}
\begin{itemize}
 \item By definition and non-negativity of $\theta_j^2$, $\sqrt{\frac{n_t}{n - n_t}} \frac{1}{\sqrt{\Deltak_d}} \leq 1$, therefore 
 $\Deltak \geq \Deltak_d \geq \frac{n_t}{n - n_t}$. 
 \item $\Deltaor = \frac{\Deltak}{n_t} \geq \frac{1}{n - n_t}$.
 \item $\Deltal = \sqrt{\frac{\Deltaor}{n - n_t}} \geq \sqrt{\frac{1}{(n - n_t)^2}} = \frac{1}{n - n_t}$.
 \item $\frac{\Deltaor}{\Deltal} = \Deltaor \sqrt{\frac{n - n_t}{\Deltaor}} = \sqrt{(n - n_t) \Deltaor} \geq 1$.
 \item By definition, $\Deltak_g \leq k_*$.
 Thus $\frac{\Deltak_g}{n_t} \leq \frac{k_*}{n_t} \leq \oracle(n_t)$.
 Moreover, 
 \[ \Deltak_d \left[ 1 - \sqrt{\frac{n_t}{n - n_t}} \frac{1}{\sqrt{\Deltak_d}} \right] \frac{1}{n_t} 
 \leq \sum_{j = k_* + 1}^{k_* + \Deltak_d} \theta_j^2 \leq \sum_{j = k_* + 1}^{+\infty} \theta_j^2 \leq \oracle(n_t). \]
Thus
\begin{align*}
 n_t \oracle(n_t) &\geq \Deltak_d - \sqrt{\frac{n_t}{n - n_t}} \sqrt{\Deltak_d} \\
 &\geq \Deltak_d - \frac{1}{2} \frac{n_t}{n - n_t} - \frac{1}{2} \Deltak_d \\
 &\geq \frac{1}{2} \Deltak_d - \frac{1}{2} \frac{n_t}{n - n_t}.
\end{align*}
It follows that
\[ \Deltak_d \leq 2 n_t \oracle(n_t) + \frac{n_t}{n - n_t}, \]
so since $\frac{n_t}{n - n_t} \frac{1}{n_t} = \frac{1}{n - n_t}$,
\[ \frac{\Deltak_d}{n_t} \leq 2 \oracle(n_t) + \frac{1}{n - n_t}, \]
 which proves the result.
 \end{itemize}

\subsubsection{Proof of lemma \ref{3.claim_bd_diff_fn}} \label{3.sec_dem_bd_diff_fn}
$f_n$ is continuous and piecewise linear by definition \ref{3.def_fn}.
 $f_n$ is convex because the sequence $\theta_j^2$ is non-increasing by assumption. 
 Let $j \in \mathbb{Z}$ and $\alpha \in \left] \tfrac{j}{\Deltak}; \tfrac{j+1}{\Deltak} \right[$ be two numbers.
 By definition, $f_n$ is linear on the interval $\left] \tfrac{j}{\Deltak}; \tfrac{j+1}{\Deltak} \right[$, 
 in particular $f_n$ is differentiable on this interval and
 \begin{align}
  f_n'(\alpha) &= \Deltak \left[f_n \left( \tfrac{j+1}{\Deltak} \right) - f_n \left( \tfrac{j}{\Deltak} \right) \right] \nonumber \\
  &= \frac{\Deltak}{\Deltal} \left[ \frac{1}{n_t} - \theta_{k_* + j+1}^2 \right]. \label{eq_deriv_fn}
 \end{align}
Because the sequence $\theta_j^2$ is non-increasing, it follows from the definition of $k_*(n_t)$ that $f_n$ is increasing on $\left] \tfrac{j}{\Deltak}; \tfrac{j+1}{\Deltak} \right[$ if $j \geq 0$ and non-increasing if $j < 0$. This implies that $f_n$ reaches its minimum at $k_*(n_t)$. 
If $\alpha \geq 1$, then $j = \lfloor \alpha \Deltak \rfloor \geq \Deltak$, therefore by definition of 
$\Deltak_d \leq \Deltak$,
\begin{align*}
 f_n'(\alpha) &\geq \frac{\Deltak}{\Deltal} \left[ \frac{1}{n_t} - \theta_{k_* + \Deltak + 1}^2 \right] \\
 &\geq \frac{\Deltak}{\Deltal} \sqrt{\frac{n_t}{n - n_t}} \frac{1}{\sqrt{\Deltak}} \frac{1}{n_t} \\
 &= \Deltak \sqrt{\frac{(n-n_t) n_t}{\Deltak}} \sqrt{\frac{n_t}{n - n_t}} \frac{1}{\sqrt{\Deltak}} \frac{1}{n_t}\\
 &= 1.
\end{align*}
In the same way, if $\alpha < -1$, then $j+1 = \lceil \alpha \Deltak \rceil \leq - \Deltak \leq -\Deltak_g$, so
\begin{align*}
 f_n'(\alpha) &\leq \frac{\Deltak}{\Deltal} \left[ \frac{1}{n_t} - \theta_{k_* - \Deltak}^2 \right] \\
 &\leq - \frac{\Deltak}{\Deltal} \sqrt{\frac{n_t}{n - n_t}} \frac{1}{\sqrt{\Deltak}} \frac{1}{n_t} \\
 &\leq -1.
\end{align*}
Furthermore,
\begin{itemize}
 \item If $\Deltak = \Deltak_d$, then for all $\alpha \in [0;1]$, $j+1 = \lceil \alpha \Deltak \rceil \leq \Deltak = \Deltak_d$,
 therefore by definition of $\Deltak_d$,
 \begin{align*}
   f_n'(\alpha) &\leq \frac{\Deltak}{\Deltal} \left[ \frac{1}{n_t} - \theta_{k_* + \Deltak}^2 \right] \\
 &\leq \frac{\Deltak}{\Deltal} \sqrt{\frac{n_t}{n - n_t}} \frac{1}{\sqrt{\Delta}} \frac{1}{n_t} \\
 &\leq 1.
 \end{align*}
 \item If $\Deltak = \Deltak_g$, then for all $\alpha \in [-1;0]$, $j = \lfloor \alpha \Deltak \rfloor \geq - \Deltak = - \Deltak_g$,
 therefore by definition of $\Deltak_g$ and since the sequence $(\theta_j^2)_{j \in \mathbb{N}}$ is non-increasing,
 \begin{align*}
   f_n'(\alpha) &\geq \frac{\Deltak}{\Deltal} \left[ \frac{1}{n_t} - \theta_{k_* - \Deltak + 1}^2 \right] \\
 &\geq - \frac{\Deltak}{\Deltal} \sqrt{\frac{n_t}{n - n_t}} \frac{1}{\sqrt{\Deltak}} \frac{1}{n_t} \\
 &\geq - 1.
 \end{align*}
\end{itemize}
By continuity of $f_n$, this proves the lemma.

\subsubsection{Properties of the interval $[a_x;b_x]$}

\begin{lemma} \label{3.lem_ub_fen_risk}
Let $a_x,b_x$ be as defined in Theorem\ref{3.thm_approx_ho}.
 Then for all $x > 0$,
 \begin{align*}
 [b_x - a_x] &\leq 2 (1+x) \\
  \sum_{k_* + a_x \Deltak}^{k_* + b_x \Deltak} \theta_j^2 &\leq 4 (1+x)\Deltaor.   
 \end{align*}
\end{lemma}

\begin{proof}
Either $b_x \leq 1$, or $b_x > 1$ and by lemma \ref{3.claim_bd_diff_fn}, 
$f_n(b_x) \geq b_x - 1$ which implies that $b_x \leq f_n(b_x) + 1 
\leq x + 1$. In all cases, $b_x \leq x + 1$. In the same way, $a_x > -1 - x$. 
Thus $b_x - a_x \leq 2(1 + x)$.
Moreover,
 \begin{align*}
  \sum_{k_* + a_x \Deltak}^{k_* + b_x \Deltak} \theta_j^2 &\leq \sum_{k_* + a_x \Deltak}^{k_* + b_x \Deltak} \left| \theta_j^2 - \frac{1}{n_t} \right| 
  + \frac{(b_x - a_x)\Deltak}{n_t} \\
  &\leq \Deltal [f_n(a_x) + f_n(b_x)] +  [b_x - a_x] \Deltaor \\ 
  &\leq 2x \Deltal + 2 \left[1 + x \right] \Deltaor.
\end{align*}
Since $\Deltal \leq \Deltaor$ by  lemma \ref{3.lem_odgs}, 
\[ \sum_{k_* + a_x \Deltak}^{k_* + b_x \Deltak} \theta_j^2 \leq (4x + 2) \Deltaor. \]
This proves lemma \ref{3.lem_ub_fen_risk}. 
\end{proof}

We now introduce some notation which will be used in the remainder of this chapter.
\begin{definition}
Let an i.i.d sample $D_n$ be given, with distribution $P$ and pdf $s$ on $[0;1]$.
 For all $j \in \mathbb{N}$ and any $T \subset \{1,\ldots,n\}$, let 
 \begin{align*}
  \theta_j &= P \psi_j = \langle s, \psi_j \rangle \\
  \Ethet{j}{T} &= P_n^T(\psi_j).
 \end{align*}
 This notation will be used very often in the remainder of the chapter.
\end{definition}

The hold-out risk estimator can be expressed as the sum of two terms.
Definition \ref{3.def_Z} below gives a name to each of these terms.
\begin{definition} \label{3.def_Z}
For all $j \in [-k_*(n_t); + \infty[ \cap \mathbb{Z}$, let 
\[ L(\tfrac{j}{\Deltak}) = \frac{1}{\Deltal} \left( \Norm{\ERMtrigo{k_* + j}{T} - s}^2 - 
 \Norm{\ERMtrigo{k_*}{T} - s}^2 \right).\]
 The function $L$ is extended to the interval $[- \tfrac{k_*(n_t)}{\Deltak}; + \infty [$ by linear interpolation.
 Let $Z$ be the random function defined for all $j \in [- \tfrac{k_*(n_t)}{\Deltak}; + \infty[ \cap \mathbb{Z}$ by
 \[ Z \left( \frac{j}{\Deltak} \right) = \frac{2}{\Deltal} \left(P_n^{T^c} - P \right) \left( \ERMtrigo{k_* + j}{T} - \ERMtrigo{k_*}{T} \right) \]
 and extended by linear interpolation to the interval $[-\tfrac{k_*(n_t)}{\Deltak}; + \infty[$, 
 so that for all $\alpha$, $\ho{T} \left( \alpha \right) = L(\alpha) - Z_{\alpha}$.
\end{definition}
Thus, $L$ is the rescaled excess risk, and $Z$ 
is a centered empirical process. These two terms will be approximated separately.

\subsection{Approximation of the excess risk} \label{3.sec_approx_exrisk}
Let $x > 0$ be fixed for the entirety of this section.
We now prove the following claim.
\begin{claim} \label{3.claim_approx_ex_risk}
Let $L$ be the function introduced in definition \ref{3.def_Z}, and $f_n$ be 
given by definition \ref{3.def_fn}.
There exists a constant $\kappa_1$ such that, for any $y > 0$,
 with probability greater than $1 - e^{-y}$,
 \[ \sup_{\alpha \in [a_x;b_x]} \left|L(\alpha) - f_n(\alpha)  \right| \leq \kappa_1 (1+x) [\log(2+x) + y + \log n]^2 n^{- \min(\frac{1}{12}, \frac{\delta_4}{2})}. \]
\end{claim}
\begin{proof}
Let $j \in \{a_x \Deltak, \ldots, b_x \Deltak \}$.
Since $\ERMtrigo{k}{T} = \sum_{j = 1}^k P_n^T(\psi_j) \psi_j = \sum_{j = 1}^k \Ethet{j}{T} \psi_j$, 
\[
 \Norm{\ERMtrigo{k_* + j}{T} - \bayes}^2 -  \Norm{\ERMtrigo{k_*}{T} - \bayes}^2 
 = \text{sgn}(j) \sum_{i = k_* + (j)_- + 1}^{k_* + (j)_+} \left( \Ethet{i}{T} - \theta_i \right)^2 - \theta_i^2  .
\]
It is known \cite[Lemma 14]{Arl_Ler:2012:penVF:JMLR} \cite[Proposition 6.3]{lerasle2011} that the process 
\[ \sum_{i = k_* + (j)_- + 1}^{k_* + (j)_+} \left( \Ethet{i}{T} - \theta_i \right)^2 = \sum_{i = k_* + (j)_- + 1}^{k_* + (j)_+} (P^T - P)(\psi_j)^2 \]
concentrates around its expectation, so that 
\[ \sum_{i = k_* + (j)_- + 1}^{k_* + (j)_+} \left( \Ethet{i}{T} - \theta_i \right)^2 \sim \sum_{i = k_* + (j)_- + 1}^{k_* + (j)_+} \frac{\Var(\psi_j)}{n_t}.\]
Furthermore, by  lemma \ref{3.lem_lim_var} in the appendix, $\Var(\psi_j) \sim 1$, therefore 
\[\sum_{i = k_* + (j)_- + 1}^{k_* + (j)_+} \left( \Ethet{i}{T} - \theta_i \right)^2 \sim \frac{|j|}{n_t}. \]
More precisely, proposition \ref{3.prop_approx_rsk} in the appendix and a union bound show that, with probability greater than $1 - e^{-y}$,
 for any $j \in \mathbb{Z} \cap [a_x \Deltak; b_x \Deltak[$,
 \[ \left| \sum_{i = k_* + (j)_- + 1}^{k_* + (j)_+} \left( \Ethet{i}{T} - \theta_i \right)^2 - \frac{|j|}{n_t}  \right| 
 \leq  \kappa_1 (y + \log n + \log((b_x - a_x)\Deltak \wedge 1))^2 n^{- \min(\frac{1}{12}, \frac{\delta_4}{2})}\frac{j}{\Deltak} \Deltal. \]
 Let $r_n = \kappa_1 (y + \log n + \log((b_x - a_x)\Deltak \wedge 1))^2 n^{- \min(\frac{1}{12}, \frac{\delta_4}{2})}$. Then for any $j \geq 1$,
 \begin{align*}
   \Norm{\ERMtrigo{k_* + j}{T} - \bayes}^2 -  \Norm{\ERMtrigo{k_*}{T} - \bayes}^2 
 &= - \sum_{i = k_* + 1}^{k_* + j} \theta_i^2 + \sum_{i = k_* + 1}^{k_* + j} \left( \Ethet{i}{T} - \theta_i \right)^2 \\
 &= - \sum_{i = k_* + 1}^{k_* + j} \theta_i^2 + \frac{j}{n_t} \pm \frac{j}{\Deltak} r_n \Deltal \\
 &= \sum_{i = k_* + 1}^{k_* + j} \left[ \frac{1}{n_t} - \theta_i^2 \right] \pm \frac{j}{\Deltak} r_n \Deltal \\
 &= \Deltal f_n \left( \frac{j}{\Deltak} \right) \pm \frac{j}{\Deltak} r_n \Deltal.
 \end{align*}
On this same event, for any $j \in \{-k_*(n_t),\ldots, -1\}$,
\begin{align*}
 \Norm{\ERMtrigo{k_* + j}{T} - \bayes}^2 -  \Norm{\ERMtrigo{k_*}{T} - \bayes}^2 
 &=  \sum_{i = k_* + j + 1}^{k_*} \theta_i^2 - \sum_{i = k_* + j + 1}^{k_*} \left( \Ethet{i}{T} - \theta_i \right)^2 \\
 &=  \sum_{i = k_* + j + 1}^{k_*} \theta_i^2 - \frac{|j|}{n_t} \pm \frac{|j|}{\Deltak} r_n \Deltal \\
 &= \sum_{i = k_* + j + 1}^{k_*} \left[ \theta_j^2 - \frac{1}{n_t} \right] \pm \frac{j}{\Deltak} r_n \Deltal \\
 &= \Deltal f_n \left( \frac{j}{\Deltak} \right) \pm \frac{j}{\Deltak} r_n \Deltal.
\end{align*}
Thus, since $f_n$ and $\ho{T}$ are linear between the points of $\frac{1}{\Deltak} \mathbb{Z}$,
\begin{align*}
 \sup_{\alpha \in [a_x;b_x]} \left|L(\alpha) - f_n(\alpha)  \right| &= \frac{1}{\Deltal} \max_{a_x \Deltak \leq j \leq b_x \Deltak} 
 \Bigl| \Norm{\ERMtrigo{k_* + j}{T} - \bayes}^2 -  \Norm{\ERMtrigo{k_*}{T} - \bayes}^2 
 - \Deltal f_n \left( \tfrac{j}{\Deltak} \right) \Bigr| \\
 &\leq \max(|a_x|,|b_x|) r_n.
\end{align*}
By lemma \ref{3.lem_ub_fen_risk}, $\max(|a_x|,|b_x|) \leq b_x - a_x \leq 2(1+x)$ so
\begin{align*}
  \max(|a_x|,|b_x|)  r_n  &\leq 2 (1+x) \kappa_1 (y + \log n + \log(2(1+x)) + \log (1 \wedge \Deltak))^2 n^{- \min(\frac{1}{12}, \frac{\delta_4}{2})} \\
  &\leq \kappa (1+x) [\log(2+x) + \log n + y]^2 n^{- \min(\frac{1}{12}, \frac{\delta_4}{2})}
\end{align*}
for some constant $\kappa$, since by lemma \ref{3.lem_odgs} and hypothesis \eqref{3.inthm_hyp_lb_nv} 
of Theorem \ref{3.thm_approx_ho}, 
\begin{align*}
    \Deltak &= n_t \Deltaor \\
    &\leq 2n_t or(n_t) + \frac{n_t}{n - n_t} \\
    &\leq 2 (\Norm{s}^2-1)n_t + n_t^{\frac{1}{3}}.
\end{align*}
This proves claim \ref{3.claim_approx_ex_risk}.
\end{proof}

We will now seek to approximate the process $Z$ given by definition \ref{3.def_Z}.
\subsection{Strong approximation of the hold-out process}
Let us start by showing that the empirical process $Z$ (definition \ref{3.def_Z}) can be approximated by a gaussian process, 
uniformly on $[a_x;b_x]$. This is the purpose of the following result, which will be proven in this section.
\begin{claim} \label{3.claim_strong_approx_proc_emp}
 Let $Z$ be the process given by definition \ref{3.def_Z}. 
 There exists a gaussian process $(Z^1_\alpha)_{\alpha \in [a_x;b_x]}$ with the same variance-covariance function
 as $Z$: for any $(\alpha_1,\alpha_2) \in [a_x;b_x]^2$, $\Cov(Z^1_{\alpha_1},Z^1_{\alpha_2}) = \Cov(Z_{\alpha_1}, Z_{\alpha_2})$
 and such that for all $n \geq 1$, for all $x > 0$, with probability greater than $1 - e^{-y}$,
 \[ \mathbb{E} \left[\sup_{\alpha \in [a_x;b_x]} |Z_\alpha - Z^1_\alpha| \Bigr| D_n^T \right] \leq 
 \kappa_5(c_1, \delta_5) (1+y) (1+x)^{\frac{3}{2}} n^{-\frac{\delta_5}{3}}. \]
 Furthermore, $Z^1$ can be expressed as $Z^1 = \funal (Z,\nu)$, with $\nu$ a uniform random variable independent from $D_{n}$
 and $\funal$ a measurable function on $C([0;1], \mathbb{R})$.
\end{claim}

Let $n_v = |T^c| = n - |T| = n - n_t$.
Let $F: x \rightarrow \int_{0}^x s(t) dt$ be the cumulative distribution function of the given $X_i$.
Let $F_{T^c}: x \rightarrow \frac{1}{n_v} \sum_{i \notin T} \mathbb{I}_{X_i \leq x}$ be the empirical cumulative distribution function 
of the sample $D_n^{T^c}$.
By the Komlos-Major-Tusnady approximation theorem \cite[Theorem 3]{KMT1975}, there exist a universal constant $C$ 
and a standard Brownian bridge process $B_{T^c}$ such that for all $y > 0$, with probability
greater than $1 - e^{-y}$, $\NormInfinity{B_{T^c} \circ F - \sqrt{n_v}(F_{T^c} - F)} \leq \frac{C(\log n_v + y)}{\sqrt{n_v}}$ (remark that since
$F$ is continuous, $F(X_i) \sim \mathcal{U}([0;1])$, which means that the result for general $F$ follows from the result for the uniform distribution). 
Furthermore, $B_{T^c}$ can always be realized as a measurable function of $D_{n}^{T^c}$ and an auxiliary, uniformly distributed
random variable $\nu$: $B_{T^c} = H(D_n^{T^c}, \nu)$, with $\nu$ independant from $D_n$. 
Let $B^{T^c}$ be obtained in this way.
From $B_{T^c} \circ F$, one can define an operator on the Sobolev space $W^1(\mathbb{R})$:
\begin{definition}
 For any function $f$ such that $f' \in L^1([0;1])$, let
 \[ G_{T^c}(f) = - \int_{0}^{1} f'(x) B_{T^c}(F(x)) dx.  \]
\end{definition}
$G_{T^c}$ "approximates" the empirical process $\sqrt{n_v} (P_n^{T^c} - P)$ on the space $W^1$.
Lemma \ref{3.lem_strong_approx} below gives a bound on the error made with this approximation.

\begin{lemma} \label{3.lem_strong_approx}
 For any function $f$ such that $f' \in L^1([0;1])$,
 \[ \left|G_{T^c}(f) - \sqrt{n_v} (P_n^{T^c} - P)(f)  \right| \leq 
 \NormInfinity{B_{T^c} - \sqrt{n_v}(F_{T^c} - F)} \Norm{f'}_{L^1}.  \]
 Furthermore, for all functions $f,g$ such that $f',g' \in L^1([0;1])$,
 \[ \Cov (G_{T^c}(f), G_{T^c}(g)) = P[fg] - P[f] P[g] = \Cov \left( \sqrt{n_v} (P_n^{T^c} - P)(f) , \sqrt{n_v} (P_n^{T^c} - P)(g)  \right). \]
\end{lemma}

\begin{proof}
Let $f$ be a function such that $f' \in L^1([0;1])$. Then
\begin{align*}
 (P_n^{T^c} - P)(f) &= \int f d(P_n^{T^c} - P) \\
 &= \int [f - f(0)] d(P_n^{T^c} - P) \\
 &= \int_{0}^{1} \int_{0}^{1} \mathbb{I}_{t < x} f'(t) dt \ d(F_{T^c} - F)(x) \\
 &= \int_{0}^{1} f'(t) (P_n^{T^c} - P)((t, +\infty) ) \\
 &= - \int_{0}^{1} f'(t) (F_{T^c} - F)(t) dt \numberthis \label{3.eq_proc_emp}.
\end{align*}
Il follows that for all functions $f$ such that $f' \in L^1([0;1])$,
\begin{align*}
 \left| G_{T^c}(f) - \sqrt{n_v}(P_n^{T^c} - P)(f)  \right| &= \left|\int_{0}^{1} f'(t) \bigl[ \sqrt{n_v}(F_{T^c} - F) - B_{T^c} \circ F \bigr] 
 (t) dt \right| \\
 &\leq \Norm{f'}_{L^1([0;1])} \NormInfinity{B_{T^c} \circ F - \sqrt{n_v}(F_{T^c} - F)}.
\end{align*}
By definition, it is clear that $\mathbb{E}[G_{T^c}(f)] = 0$. Thus,
\begin{align*}
 \Cov \left( G_{T^c}(f), G_{T^c}(g) \right) &= \mathbb{E} \left[ G_{T^c}(f) G_{T^c}(g) \right] \\
 &= \mathbb{E} \left[ \int_0^1 \int_0^1 f'(u) g'(v) B_{T^c} (F(u)) B_{T^c} (F(v)) \right] \\
 &= \int_0^1 \int_0^1 f'(u) g'(v) [F(u) \wedge F(v)] [1 - F(u) \vee F(v)] du dv \\
 &=\int_0^1 \int_0^1 f'(u) g'(v) \left(\mathbb{E}[\mathbb{I}_{X \leq u} \mathbb{I}_{X\leq v}] - 
 \mathbb{E}[\mathbb{I}_{X \leq u}] \mathbb{E}[\mathbb{I}_{X \leq v}] \right) \\
 &= n_v \int_0^1 \int_0^1 f'(u) g'(v) \mathbb{E} \left[ (F_{T^c} - F)(u) (F_{T^c} - F)(v) \right] \\
 &= \Cov \left( \sqrt{n_v} (P_n^{T^c} - P)(f) , \sqrt{n_v} (P_n^{T^c} - P)(g) \right) \text{ by equation } \eqref{3.eq_proc_emp}
\end{align*}

\end{proof}

Let the process $Z^1$ be defined for all $j \in \{a_x \Deltak, \ldots, b_x \Deltak \}$ by
\[ Z^1 \left( \frac{j}{\Deltak} \right) = \frac{2}{\sqrt{n_v} \Deltal} G_{T^c} \left( \ERMtrigo{k_* + j}{T} - \ERMtrigo{k_*}{T} \right). \]
$Z^1$ is extended to the interval $[a_x;b_x]$ by linear interpolation, as for $Z$. By lemma \ref{3.lem_strong_approx},
the variance-covariance function of $Z^1$
conïncides with that of $Z$ at the points $\tfrac{j}{\Deltak}, j \in \mathbb{Z} \cap [a_x \Deltak;b_x \Deltak]$, 
and this property extends by bilinearity to the whole interval $[a_x;b_x]$.
Furthermore,
\begin{align*}
 \sup_{a_x \leq \alpha \leq b_x} |Z^1_\alpha - Z_\alpha| &\leq \max_{j \in \mathbb{Z} \cap [a_x;b_x]} \left|Z^1\left( \frac{j}{\Deltak} \right) 
 - Z\left( \frac{j}{\Deltak} \right) \right| \\
 &\leq \frac{4 \sqrt{2} \pi}{\Deltal \sqrt{n_v}} \NormInfinity{B_{T^c} \circ F - \sqrt{n_v}(F_{n_v} - F)} 
 \times \max_{a_x \Deltak \leq j \leq b_x \Deltak}  \Norm{\sum_{i = k_* + 1}^{k_* + j} i \Ethet{i}{T} \sin(2 i \pi \cdot)}_1 \\
 &\leq \frac{4 \pi}{\Deltal \sqrt{n_v}} \NormInfinity{B_{T^c} \circ F - \sqrt{n_v}(F_{n_v} - F)} 
 \times \sqrt{\sum_{i = k_* + a_x \Deltak + 1}^{k_* + b_x \Deltak} i^2 (\Ethet{i}{T})^2}
\end{align*}
By construction, the process $B_{T^c} \circ F - \sqrt{n_v}(F_{n_v} - F)$ is independent from $D_n^T$.
As a result,
\begin{align*}
 \mathbb{E} \left[ \sup_{a_x \leq \alpha \leq b_x} |Z^1_\alpha - Z_\alpha| | D_n^T \right] 
 &\leq \frac{4 \pi}{\Deltal \sqrt{n_v}} \mathbb{E} \left[ \NormInfinity{B_{T^c} \circ F - \sqrt{n_v}(F_{n_v} - F)} \right] \\
&\quad \times (k_* + b_x \Deltak) \sqrt{\sum_{i = k_* + a_x \Deltak + 1}^{k_* + b_x \Deltak} \left( \Ethet{j}{T} \right)^2} \\
 &\leq \frac{4 \pi C \log n_v}{\Deltal n_v} \times (k_* + b_x \Deltak) \sqrt{\sum_{j = k_* + a_x \Deltak + 1}^{k_* + b_x \Deltak} 
 2 \theta_j^2  + 2 \left( \Ethet{j}{T} - \theta_j \right)^2}. \numberthis \label{3.eq_ub_Z1_Z}
\end{align*}
By proposition \ref{3.prop_approx_rsk}, there exists an event $E_1(y)$ of probability greater than $1 - e^{-y}$ such that,
for all $D_n^T \in E_1(y)$,
\[ \sum_{j = k_* + a_x \Deltak + 1}^{k_* + b_x \Deltak} \left( \Ethet{j}{T} - \theta_j \right)^2 
\leq [b_x - a_x] \frac{\Deltak}{n_t} + \kappa_1 (b_x - a_x) [\log n + y]^2 n^{- \min(\frac{1}{12}, \frac{\delta_4}{2})} \Deltal (n), \]
therefore by  lemma \ref{3.lem_ub_fen_risk} and equation \eqref{3.eq_ub_Z1_Z}, for all $D_n^T \in E_1(y)$,
\[ \mathbb{E} \left[ \sup_{a_x \leq \alpha \leq b_n} |Z^1_\alpha - Z_\alpha| \Bigr| D_n^T \right] \leq 
\frac{4 \pi C \log n_v}{\Deltal n_v} \times (k_* + b_x \Deltak) 2 \sqrt{1+x} \left[ 2 \sqrt{\Deltaor} + \sqrt{2 \kappa_1} (\log n + y) 
n^{- \min(\frac{1}{12}, \frac{\delta_4}{2})} \sqrt{\Deltal} \right]. \]
Since $\Deltal \leq \Deltaor$ and $n^{- \min(\frac{1}{12}, \frac{\delta_4}{2})}\log n \to 0$, there exists therefore a constant $\kappa$ such that 
for all $D_n^T \in E_1(y)$ :
\begin{align*}
 \mathbb{E} \left[ \sup_{a_x \leq \alpha \leq b_n} |Z^1_\alpha - Z_\alpha| \Bigr| D_n^T \right] &\leq \kappa \frac{\log n_v}{\sqrt{n_v}} \frac{\sqrt{\Deltaor}}{\Deltal \sqrt{n_v}} \times (k_* + b_x \Deltak) 
(1+y) \sqrt{1+x} \\
&\leq \kappa \frac{\log n_v}{\sqrt{n_v}} \times (k_* + b_x \Deltak) (1+y)
\sqrt{1+x}. \numberthis \label{3.eq_strong_approx}
\end{align*}
By lemma \ref{3.lem_odgs}, $\Deltaor \leq 2 \oracle(n_t) + \frac{1}{n_v}$ therefore $\Deltak \leq 2n_t \oracle(n_t) + \frac{n_t}{n_v}$ 
and by definition of $or$, $k_*(n_t) = n_t \frac{k_*}{n_t} \leq n_t \oracle(n_t)$
therefore $\frac{k_* + b_x \Deltak}{\sqrt{n_v}} \leq (2b_x + 1) \frac{n_t \oracle(n_t)}{\sqrt{n_v}} + \frac{b_x n_t}{n_v \sqrt{n_v}}$. 
By hypothesis \ref{3.inthm_hyp_lb_nv} of section \ref{3.sec.hyp}, 
$n_v \geq n^{\frac{2}{3} + \delta_5}$, so
\[ \frac{k_* + b_x \Deltak}{\sqrt{n_v}} \leq (2b_x + 1) \frac{n_t \oracle(n_t)}{n^{\frac{1}{3} + \frac{\delta_5}{2}}} + 
\frac{b_x}{n^{\frac{3\delta_5}{2}}}. \]
Moreover, by hypothesis \ref{3.inthm_hyp_ub_sum_varphi} of Theorem \ref{3.thm_approx_ho}, 
$\sum_{j = k+1}^{+\infty} \theta_j^2 \leq \frac{c_1}{k^{2 + \delta_1}}$, therefore 
\[\oracle(n_t) \leq \min_{k \in \mathbb{N}^*} \frac{c_1}{k^{2 + \delta_1}} + \frac{k}{n_t}
\leq 2 \inf_{x \geq 1} \frac{c_1}{x^{2 + \delta_1}} + \frac{x}{n_t} \leq 3 \frac{c_1^{\frac{1}{3 + \delta_1}}}{n_t^{\frac{2 + \delta_1}{3 + \delta_1}}},\]
whence (since $c_1 \geq 1$) 
$n_t \oracle(n_t) \leq 3 (c_1 n_t)^{\frac{1}{3 + \delta_1}}$.
It follows that:
\begin{equation} \label{3.eq_odg_strong_approx}
 \log n \frac{k_* + b_x \Deltak}{\sqrt{n_v}} \leq 3(2b_x + 1) c_1^{\frac{1}{3 + \delta_1}} \log n n^{- \frac{\delta_5}{2}}
 + \frac{b_x \log n}{n_t^{\frac{3\delta_5}{2}}}.
\end{equation}
Since $\frac{\log n}{n^{\frac{\delta_5}{2}}} = o\left( n^{-\frac{\delta_5}{3}} \right)$,
by  equations \eqref{3.eq_strong_approx}, \eqref{3.eq_odg_strong_approx} and lemma \ref{3.lem_ub_fen_risk}, 
there exists a constant $\kappa(c_1, \delta_5)$ such that for any $n$, with probability greater than $1 - e^{-y}$,
\[ \mathbb{E} \left[ \sup_{a_x \leq \alpha \leq b_n} |Z^1_\alpha - Z_\alpha| \Bigr| D_n^T \right] \leq 
\kappa (1+y) (1 + x)^{\frac{3}{2}} n^{-\frac{\delta_5}{3}}. \]

\subsection{Approximation of the covariance function}

We will now seek to approximate the process $Z^1$ given by claim \ref{3.claim_strong_approx_proc_emp}
by a time-changed Wiener process. To this end, we first approximate the variance-covariance function
of $Z^1$ (which is the same as that of $Z$).

\begin{claim} \label{3.claim_def_g}
 There exists a function $g_n$ satisfying the hypotheses of Theorem \ref{3.thm_approx_ho}
 and a constant $u_5 > 0$
 such that, for all $x,y > 0$, with probability greater than $1 - e^{-y}$, 
 \begin{align*}
  \max_{(j_1, j_2) \in \{0,\ldots,b_x \Deltak\}^2} \left| \Cov \left(Z \left( \tfrac{j_1}{\Deltak} \right), Z \left( \tfrac{j_2}{\Deltak} \right) |D_n^T \right) 
  - \min \left( g_n \left(\tfrac{j_1}{\Deltak} \right), g_n \left( \tfrac{j_2}{\Deltak} \right) \right) \right|
  &\leq \kappa_6(1+x)^2 [y + \log n]^2 n^{-u_5} \\
  \max_{(j_1, j_2) \in \{a_x \Deltak,\ldots,0\}^2} \left| \Cov \left(Z \left( \tfrac{j_1}{\Deltak} \right), Z \left( \tfrac{j_2}{\Deltak} \right) | D_n^T  \right) 
  - \max \left( g_n \left(\tfrac{j_1}{\Deltak} \right), g_n \left( \tfrac{j_2}{\Deltak} \right) \right) \right|
  &\leq \kappa_6(1+x)^2 [y + \log n]^2n^{-u_5} \\
  \max_{(j_1, j_2) \in \{a_x \Deltak,\ldots, 0\} \times \{0,\ldots,b_x \Deltak \} } \left| \Cov \left(Z \left( \tfrac{j_1}{\Deltak}\right), 
  Z \left( \tfrac{j_2}{\Deltak} \right) | D_n^T  \right)  \right|
  &\leq \kappa_6(1+x)^2 [y + \log n]^2 n^{-u_5}.
 \end{align*}
\end{claim}

We introduce the following definition.
\begin{definition} \label{3.def_K}
Let $(W_t)_{t \in \mathbb{R}}$ be a two-sided Wiener process such that $W_0 = 0$.
 For any function $g: I \rightarrow \mathbb{R}$, where $I$ is an interval containing $0$, 
 let $K(g): I^2 \rightarrow \mathbb{R}$
 be defined for any $(s, t) \in I^2 $ by
 \begin{equation} \label{3.def_Kg}
  K(g)(s,t) =  \begin{cases}
                             &g(s \wedge t) \text{ if } (s,t) \in (I \cap \mathbb{R}_+)^2 \\ 
                             & - g(s \vee t) \text{ if } (s,t) \in (I \cap \mathbb{R}_-)^2\\
                             &0 \text{ else }.
                            \end{cases}
 \end{equation}
\end{definition}

For all $j \in \mathbb{Z} \cap [a_x \Deltak; b_x \Deltak]$, by definition \ref{3.def_Z} of $Z$:
\begin{equation}
 Z\left( \frac{j}{\Deltak} \right) = \frac{2}{\Deltal} \left( P_n^{T^c} - P \right) \left( \ERMtrigo{k_* + j}{T} - \ERMtrigo{k_*}{T} \right)
= \begin{cases}
&0 \text{ if } j = 0, \\
&\frac{2}{\Deltal} \left( P_n^{T^c} - P \right) \sum_{i = k_*+1}^{k_* + j} \Ethet{i}{T} \psi_i \text{ if } j > 0, \\
&\frac{2}{\Deltal} \left( P_n^{T^c} - P \right) \sum_{i = k_*+j+1}^{k_*} \Ethet{i}{T} \psi_i \text{ if } j < 0
\end{cases}
\end{equation}
In other words, for all $j \in \mathbb{Z} \cap [a_x \Deltak; b_x \Deltak]$,
\[ Z \left( \frac{j}{\Deltak} \right)  =  \text{sgn}(j) \frac{2}{\Deltal} \sum_{i = k_* - (j)_- + 1}^{k_* + (j)_+}  \Ethet{i}{T} 
\left( P_n^{T^c} - P \right)(\psi_i). \]
Let $n_v = |T^c| = n - n_t$.
Thus, for any $(j_1,j_2) \in \{ a_x \Deltak,\ldots, b_x \Deltak \}^2$ and any variable $X$ with distribution $s(x)dx$,
\begin{align*}
&\Cov \left(Z \left( \tfrac{j_1}{\Deltak} \right), Z \left( \tfrac{j_2}{\Deltak} \right) | D_n^T \right) \\ 
&\quad = \text{sgn}(j_1) \text{sgn}(j_2)  
 \frac{4}{n_v \Deltal^2} \sum_{i_1 = k_* - (j_1)_- + 1}^{k_* + (j_1)_+} \sum_{i_2 = k_* - (j_2)_- + 1}^{k_* + (j_2)_+} 
 \Ethet{i_1}{T} \Ethet{i_2}{T} \Cov(\psi_{i_1}(X), \psi_{i_2}(X)).
\end{align*}
Let us now introduce the following definition.
\begin{definition} \label{3.def_E}
 Let $(m_1,m_2,m_3) \in \mathbb{N}^3$ 
 be three integers. Let $m_{(1)} \leq m_{(2)} \leq m_{(3)}$ be their non-decreasing rearrangement. Define $E_{m_1,m_2,m_3} = 0$ 
 if $m_{(1)} = m_{(2)}$ or $m_{(2)} = m_{(3)}$ and
 \begin{equation}
  E_{m_1,m_2,m_3} = \sum_{j_1 = m_{(1)}+1}^{m_{(2)}} \sum_{j_2 = m_{(2)} + 1}^{m_{(3)}} \hat{\theta}_{j_1}^T \hat{\theta}_{j_2}^T 
  \Cov(\psi_{j_1}(X), \Cov(\psi_{j_2}(X))
 \end{equation}
 if $m_{(1)} < m_{(2)} < m_{(3)}$ .
\end{definition}
Let $n_v = n - n_t$.
The covariance can be broken down as follows:
If $0 < j_1 \leq j_2$, conditionally on $D_n^T$.
\begin{align*}
 \Cov \left(Z \left( \frac{j_1}{\Deltak} \right), Z \left( \frac{j_2}{\Deltak} \right) | D_n^T \right) 
 &= \Var \left(Z \left( \frac{j_1}{\Deltak} \right) \right) \\
&\quad +  \frac{4}{n_v \Deltal^2} \sum_{i_1 = k_*+1}^{k_* + j_1} 
\sum_{i_2 = k_* + j_1+1}^{k_* + j_2} \hat{\theta}_{i_1}^T \hat{\theta}_{i_2}^T \Cov(\psi_{i_1}(X), \psi_{i_2}(X)).   
\end{align*}

If $j_1 \leq j_2 < 0$, symetrically,
\begin{align*}
 \Cov \left(Z \left( \frac{j_1}{\Deltak} \right), Z \left( \frac{j_2}{\Deltak} \right) | D_n^T \right) &= \Var\left(Z \left( \frac{j_2}{\Deltak} \right) \right) \\
&\quad + \frac{4}{n_v \Deltal^2} \sum_{i_2 = k_* + j_2 + 1}^{k_*} \sum_{i_1 = k_* + j_1 + 1}^{k_* + j_2} 
\hat{\theta}_{i_1}^T \hat{\theta}_{i_2}^T \Cov(\psi_{i_1}(X), \psi_{i_2}(X)).
\end{align*}
Finally, if $j_1 < 0 < j_2$,
\[ \Cov \left(Z \left( \frac{j_1}{\Deltak} \right), Z \left( \frac{j_2}{\Deltak} \right) | D_n^T \right)= 
\frac{-4}{n_v\Deltal^2} \sum_{i_1 = k_* + j_1 + 1}^{k_*} 
\sum_{i_2 = k_* + 1}^{k_* + j_2} \hat{\theta}_{i_1}^T \hat{\theta}_{i_2}^T \Cov(\psi_{i_1}(X), \psi_{i_2}(X)). \]
It follows from the previous equations that for any $(k_1,k_2) \in \mathbb{N}^2$, 
\begin{equation} \label{3.eq_struct_cov_Z}
 \Cov \left( Z\left( \frac{k_1 - k_*}{\Deltak} \right), Z\left( \frac{k_2 - k_*}{\Deltak} \right) \right) = \begin{cases}
 \Var(Z\left( \frac{k_1 - k_*}{\Deltak} \right)) + 4 \frac{E_{k_*,k_1,k_2}}{n_v \Deltal^2} \text{ if } k_* < k_1 \leq k_2 \\
  4 \frac{E_{k_*,k_1,k_2}}{n_v \Deltal^2} \text{ if } k_1 <  k_* < k_2 \\
  \Var(Z\left( \frac{k_2 - k_*}{\Deltak} \right)) + 4\frac{E_{k_*,k_1,k_2}}{n_v \Deltal^2}\text{ if } k_1 \leq k_2 < k_*
 \end{cases}.
\end{equation}
Let $I,J \subset \{a_x \Deltak, \ldots, b_x \Deltak \}$.
Assuming concentration around the expectation yields
\begin{align*}
   \sum_{i \in I} \sum_{j \in J}\Ethet{i}{T} \Ethet{j}{T} \Cov(\psi_{i}(X), \psi_{j}(X)) 
&\sim \sum_{i \in I} \sum_{j \in J} \theta_{i} \theta_{j} \Cov(\psi_{i}(X), \psi_{j}(X)) \\
&\quad + \frac{1}{n_t} \sum_{i \in I} \sum_{j \in J} \Cov(\psi_{i}(X), \psi_{j}(X))^2. 
\end{align*}
Moreover,
for any $(i_1,i_2) \in \mathbb{N}^2$, 
\[ \psi_{i_1}(X) \psi_{i_2}(X) = 2 \cos(2 i_1 \pi X) \cos(2 i_2 \pi X) = \cos \left( 2 (i_1 + i_2) \pi X \right) 
+ \cos (2(i_1 - i_2) \pi X) \]
and by definition, for all $i \in \mathbb{N}^*$, $\psi_i = \sqrt{2} \cos(2 i \pi X )$,
while $\psi_0 = 1 = \cos(0\pi x)$.
As a result, $\psi_{i_1}(X) \psi_{i_2}(X) = \frac{\psi_{i_1 + i_2}(X) + \psi_{|i_1 - i_2|}(X)}{\sqrt{2}}$ if $i_1 \neq i_2$ and
\[ \Cov(\psi_{i_1}(X), \psi_{i_2}(X)) = \frac{\theta_{i_1 + i_2}}{\sqrt{2}} 
+ \left( \frac{1 - \delta_{i_1,i_2}}{\sqrt{2}} + \delta_{i_1,i_2} \right) \theta_{|i_2 - i_1|} - \theta_{i_1} \theta_{i_2}. \]
By assumption, the sequence $|\theta_k|$ tends to $0$ with a polynomial rate of convergence, hence for sequences $i_1\sim i_2$ tending to $+\infty$, 
$\theta_{|i_1 - i_2|}$ dominates
$\theta_{i_1} \theta_{i_2}$ and $\theta_{i_1 + i_2}$.
Heuristically, it can thus be expected that 
\begin{align*}
 \sum_{i \in I} \sum_{j \in J} \Ethet{i}{T} \Ethet{j}{T} \Cov(\psi_{i}(X), \psi_{j}(X)) 
&\sim \sum_{i \in I} \sum_{j \in J} \theta_i \theta_j \left( \frac{1 - \delta_{i,j}}{\sqrt{2}} + \delta_{i,j} \right) \theta_{|j - i|} \\
&\quad + \sum_{i \in I} \sum_{j \in J} \left( \frac{1 - \delta_{i,j}}{\sqrt{2}} + \delta_{i,j} \right)^2 \theta_{|j - i|}^2.
\end{align*}
This leads to the following proposition, the rigourous proof of which can be found in the appendix (proposition \ref{3.prop_cov_approx}).
\begin{proposition} \label{3.prop_cov_approx1}
   Let $P$ be the probability measure with pdf $s$ on $[0;1]$, let $\theta_j = \langle s, \psi_j \rangle = P(\psi_j)$ and assume that the coefficients $\theta_j$ satisfy the hypotheses of section \ref{3.sec.hyp}.
  Let $\hat{\theta}_j^T = P^T(\psi_j)$. Let $I^1_k, I^2_k \subset \{ k_* + a_x \Deltak, \ldots, k_* + b_x \Deltak \}$
  be two intervals.  Then the statistics
  \[U_{I^1_k, I^2_k} = \sum_{i \in I^1_k} \sum_{j \in I^2_k} \hat{\theta}_i^T \hat{\theta}_j^T 
  [P(\psi_i \psi_j) - P \psi_i P \psi_j]  \]
  can be approximated in the following way: there exists two constants $\kappa_4$ and $u_3 > 0$ such that, with probability greater than $1 - e^{-y}$,
  \begin{align*}
   U_{I^1_k, I^2_k} &=  \frac{1}{2} \frac{|I^1_k \cap I^2_k|}{n_t} +
   \left(1 - \frac{1}{\sqrt{2}} \right) \sum_{i \in I^1_k \cap I^2_k} \theta_i^2 +
  \frac{1}{\sqrt{2}} \sum_{i \in I^1_k} \sum_{j \in I^2_k} \theta_i \theta_j\theta_{|i-j|} 
  + \frac{1}{2 n_t} \sum_{i \in I^1_k} \sum_{j \in I^2_k} \theta_{|i-j|}^2  \\
  &\quad \pm \kappa_4 (y + \log n)^2 (1+x) n^{-u_3} \Deltaor.
  \end{align*}
 \end{proposition}
It is now possible to show that the terms $E_{k_*,k_1,k_2}$ which appear in equation \eqref{3.eq_struct_cov_Z} 
are negligible compared to $\Deltaor$. That is the point of the following claim.
\begin{claim} \label{3.lem_ub_E}
Under the assumptions of Theorem \ref{3.thm_approx_ho}, there exists constants $\kappa_7 \geq 0$ and 
$u_4 > 0$ such that for all $n \in \mathbb{N}$, $x > 0$ and $(m_1,m_2,m_3) \in \{a_x \Deltak, \ldots, b_x \Deltak \}^3$ 
such that $m_1 < m_2 < m_3$, 
\begin{align}
 \sum_{j_1 = m_1+1}^{m_2} \sum_{j_2 = m_2+1}^{m_3} \theta_{j_1} \theta_{j_2} \theta_{|j_1 - j_2|} 
 &\leq \kappa_7 (1+x)^2 n^{-u_4} \Deltaor 
 \label{3.inlem_ub_cov_bias} \\
 \frac{1}{n_t} \sum_{j_1 = m_1+1}^{m_2} \sum_{j_2 = m_2+1}^{m_3} \theta_{|j_1 - j_2|}^2 &\leq  \kappa_7 (1+x)^2 n^{-u_4} \Deltaor. \label{3.inlem_ub_cov_var}
\end{align}
and moreover, for all $x > 0$, with probability greater than $1 - e^{-y}$, for any integers
$(m_1,m_2,m_3) \in \{a_x \Deltak,\ldots, b_x \Deltak \}^3$,
\begin{equation} \label{3.inlem_ub_E_m}
 |E_{m_1,m_2,m_3}| \leq \kappa_7 (1+x)^2 (y + \log n)^2 n^{-u_4} \Deltaor .
\end{equation}
\end{claim}
\begin{proof}
Assume without loss of generality that $m_1 < m_2 < m_3$.
We start by proving equation \eqref{3.inlem_ub_cov_bias}. First, changing variables from $j_1,j_2$ to $i = j_1, r = j_2 - j_1$ yields
\begin{align*}
 \sum_{j_1 = m_1+1}^{m_2} \sum_{j_2 = m_2+1}^{m_3} \theta_{j_1} \theta_{j_2} \theta_{|j_1 - j_2|} 
 &= \sum_{r \in \mathbb{N}} \theta_r \sum_{i = m_2+1-r}^{m_2} \mathbb{I}_{i \geq m_1+1} \mathbb{I}_{i+r \leq m_3} \theta_i \theta_{i+r} \\
 &\leq \frac{1}{2} \sum_{r \leq r_0} |\theta_r| \sum_{i = (m_2+1-r)\vee (m_1+1)}^{m_2 \wedge (m_3 - r)} \theta_i^2 + \theta_{i+r}^2 \\ 
&\quad + \frac{1}{2} \sum_{r > r_0} |\theta_r| \left[ \sum_{j_1 = m_1+1}^{m_2} \theta_{j_1}^2 + \sum_{j_2 = m_2+1}^{m_3} \theta_{j_2}^2 \right] \\
&\leq \frac{\Norm{\theta}_{\ell^1}}{2} \max_{1 \leq r \leq r_0} \sum_{i = (m_2+1-r)\vee (m_1+1)}^{m_2 \wedge (m_3 - r)} \theta_i^2 + \theta_{i+r}^2 \\
&\quad + \frac{1}{2} \sum_{r > r_0} |\theta_r| \sum_{i = a_x \Deltak + 1}^{b_x \Deltak} \theta_{k_* + i}^2.
\end{align*}
By claim \ref{3.lem_ub_varphi_fen} in appendix, for any $k \in \{m_1 + 1,\ldots, m_3\} \subset \{a_x\Deltak + 1,\ldots, b_x \Deltak\}$,
$\theta_k^2 \leq \kappa_3 (1+x)^2 n^{- u_2} \Deltal$, therefore 
\[ \sum_{j_1 = m_1+1}^{m_2} \sum_{j_2 = m_2+1}^{m_3} \theta_{j_1} \theta_{j_2} \theta_{|j_1 - j_2|} 
\leq r_0 \frac{\Norm{\theta}_{\ell^1}}{2} \kappa_3 (1+x)^2 n^{- u_2} \Deltal 
+ \frac{1}{2} \sum_{r > r_0} |\theta_r| \sum_{i = a_x \Deltak + 1}^{b_x \Deltak} \theta_{k_* + i}^2.  \]
By hypothesis \ref{3.inthm_hyp_ub_sum_varphi} of section \ref{3.sec.hyp},
\[ \sum_{j \geq r_0 +1} |\theta_j| \leq \sum_{j = r_0 +1}^{+\infty} \sqrt{\sum_{i = j}^{+\infty} \theta_i^2} \leq 
\sum_{j = r_0 + 1}^{+ \infty} \frac{c_1}{(j-1)^{1 + \frac{\delta_1}{2}}} \leq \frac{2c_1}{\delta_1} r_0^{\frac{- \delta_1}{2}}. \]
Thus, by lemma \ref{3.lem_ub_fen_risk},
\begin{equation} \label{3.eq_ub_covar_thet}
 \sum_{j_1 = m_1+1}^{m_2} \sum_{j_2 = m_2+1}^{m_3} \theta_{j_1} \theta_{j_2} \theta_{|j_1 - j_2|} 
 \leq \frac{r_0 \Norm{\theta}_{\ell^1}}{2} \kappa_3 (1+x)^2 n^{- u_2} \Deltaor  
 + \frac{2c_1}{\delta_1} r_0^{- \frac{\delta_1}{2}} 4(1+x) \Deltaor.
\end{equation}
Let $r_0 = \lceil n^{\frac{2u_2}{2 + \delta_1}} \rceil \leq 2n^{\frac{2 u_2}{2 + \delta_1}}$
and $u = \frac{\delta_1(u_2)}{2 + \delta_1} > 0$. For all $n \geq 2$,

\begin{equation} \label{3.eq_ub_sum_covar_phi}
 \sum_{j_1 = m_1+1}^{m_2} \sum_{j_2 = m_2+1}^{m_3} \theta_{j_1} \theta_{j_2} \theta_{|j_1 - j_2|}  \leq 
 \left[ \Norm{\theta}_{\ell^1} \kappa_3 +  8 \frac{c_1}{\delta_1} \right](1+x)^2 n^{-u} \Deltaor,
\end{equation}
which proves equation \eqref{3.inlem_ub_cov_bias}.

Moreover,
\begin{align*}
 \sum_{j_1 = m_1 + 1}^{m_2} \sum_{j_2 = m_2+1}^{m_3} \theta_{|j_1 - j_2|}^2 &= 
 \sum_{r \in \mathbb{N}} \theta_r^2 \left| \left\{j_1: (m_1 + 1 \leq j_1 \leq m_2) \wedge (m_2+1 \leq j_1 + r \leq m_3) \right\} \right| \\
 &\leq \sum_{r \in \mathbb{N}} \theta_r^2 \left[(m_3 - m_1) \wedge r \right] \\
 &\leq r_0 \sum_{r = 0}^{r_0} \theta_r^2 + (m_3 - m_1) \sum_{r > r_0} \theta_r^2 \\
 &\leq r_0 \Norm{s}^2 + (b_x - a_x) \Deltak \sum_{r > r_0} \theta_r^2  \\
 &\leq r_0 \Norm{s}^2 + 2(1+x) \Deltak \frac{c_1}{r_0^{2 + \delta_1}}, 
\end{align*}
by hypothesis \ref{3.inthm_hyp_ub_sum_varphi} of Theorem \ref{3.thm_approx_ho} and lemma \ref{3.lem_ub_fen_risk}.
Let now $r_0 = \big \lceil  \Deltak^{\frac{1}{3}} \big \rceil$. Since $\Deltak \geq 1$, it follows that:
\begin{align*}
 \frac{1}{n_t} \sum_{j_1 = m_1 + 1}^{m_2} \sum_{j_2 = m_2+1}^{m_3} \theta_{|j_1 - j_2|}^2 
 &\leq \frac{\Deltak^{\frac{1}{3}} + 1}{n_t} \Norm{s}^2 + 2 c_1 (1+x) \frac{\Deltak}{n_t} \left( \Deltak \right)^{\frac{-2}{3}} \\ 
 &\leq \left[2\frac{\Norm{s}^2}{(\Deltak)^{\frac{2}{3}}} \Deltaor + 2 c_1 (1+x) \Deltaor (\Deltak)^{-\frac{2}{3}} \right] \\
 &\leq \left[2 \Norm{s}^2 + 2 c_1 (1+x) \right] \frac{\Deltaor}{(\Deltak)^{\frac{2}{3}}}.
\end{align*}
On the other hand, $\Deltak \geq \frac{n_t}{n - n_t} \geq  n^{\delta_4}$ by hypothesis \ref{3.inthm_hyp_ub_nv} 
of Theorem \ref{3.thm_approx_ho}. There exists therefore $\kappa(c_1, \Norm{s}^2)$ such that, for any $n$,
\begin{equation} \label{3.eq_ub_sum_covar_var}
 \frac{1}{n_t} \sum_{j_1 = m_1 + 1}^{m_2} \sum_{j_2 = m_2+1}^{m_3} \theta_{|j_1 - j_2|}^2 \leq \kappa (1+x) n^{-\frac{2\delta_4}{3}} \Deltaor,
\end{equation}
which proves equation \eqref{3.inlem_ub_cov_var}.
Since for all $x > 0$, $(b_x - a_x) \leq 2(1+x) \Deltak \leq \kappa(1+x)n$, by  proposition \ref{3.prop_cov_approx1} 
and a union bound, there exists an event $A$ of probability greater than $1 - e^{-y}$ and a constant $\kappa$ such that, 
if $a_x \leq m_1 < m_2 < m_3 \leq b_x$, then
\begin{equation} \label{3.eq_E_m_prop}
\begin{split}
 |E_{m_1,m_2,m_3}| &\leq  \frac{1}{\sqrt{2}} \sum_{j_1 = m_1+1}^{m_2} \sum_{j_2 = m_2+1}^{m_3} \theta_{j_1} \theta_{j_2} \theta_{|j_1 - j_2|} 
 + \frac{1}{2n_t} \sum_{j_1 = m_1+1}^{m_2} \sum_{j_2 = m_2 + 1}^{m_3} \theta_{|j_1 - j_2|}^2 \\ 
 &\quad + \kappa (y + \log(2+x) + \log n)^2 (1+x) n^{-u_3} \Deltaor. 
\end{split}
\end{equation}
From equations \eqref{3.eq_E_m_prop}, \eqref{3.eq_ub_sum_covar_phi} and \eqref{3.eq_ub_sum_covar_var},
equation \eqref{3.inlem_ub_E_m} follows with $u_4 = \min \left(u_3,\frac{2\delta_4}{3}, \frac{\delta_1 u_2}{2 + \delta_1} \right)$.
\end{proof}

Let then $g_n^0 : \left[ \tfrac{-k_*}{\Deltak}; + \infty \right[ \rightarrow \mathbb{R}$ be defined
first for all $\alpha \in \left\{ \tfrac{j}{\Deltak} : j \in \mathbb{N} - k_* \right\}$ by 
\begin{equation} \label{3.eq_def_g0}
 \forall \alpha \in \left\{ \frac{j}{\Deltak} : (j + k_*) \in \mathbb{N} \right\}, g_n^0(\alpha) = \text{sgn}(\alpha) \Var (Z_\alpha),
\end{equation}
then for all $\alpha \in \left[ \tfrac{-k_*}{\Deltak}; + \infty \right[$ by linear interpolation (hence in general,
$g_n^0(\alpha) \neq \Var(Z_\alpha)$). 
Let $K(g_n^0)$ be given by definition \ref{3.def_K}, then by equation \eqref{3.eq_struct_cov_Z} and claim \ref{3.lem_ub_E},
with probability greater than $1 - e^{-y}$, for any $x > 0$,
\begin{equation} \label{3.eq_approx_g0}
\begin{split}
 &\max_{(j_1,j_2) \in \{a_x \Deltak,\ldots, b_x \Deltak \}^2} 
 \left| \Cov \left( Z\left( \frac{j_1}{\Deltak} \right), Z\left( \frac{j_2}{\Deltak} \right) \right) 
 - K \left( g_n^0, \frac{j_1}{\Deltak}, \frac{j_2}{\Deltak} \right) \right| \\ 
 &\quad \leq 4\kappa_7(1+x)^2 + (y+ \log n)^2 n^{-u_4} \frac{\Deltaor}{n_v \Deltal^2} \\
 &\quad \leq 4\kappa_7 (1+x)^2 (y + \log n)^2  n^{-u_4}.
\end{split}
\end{equation}

Moreover, for any $j \in \mathbb{Z} \cap [a_x \Deltak; b_x \Deltak]$, by  definition of $Z$,

\begin{equation} \label{3.eq_g0_grid}
\begin{split}
 \text{sgn}(j) g_n^0 \left( \frac{j}{\Deltak} \right)  = \Var \left( Z \left( \frac{j}{\Deltak} \right) \right)
&=  \frac{4}{n_v \Deltal^2} \sum_{i_1 = k_*- (j)_- + 1}^{k_* + (j)_+} \sum_{i_2 = k_*- (j)_- + 1}^{k_* + (j)_+} \hat{\theta}^T_{i_1} \hat{\theta}^T_{i_2} 
 \Bigl[ \frac{\theta_{i_1 + i_2}}{\sqrt{2}} \\ 
 &\quad + \Bigl( \frac{1 - \delta_{i_1,i_2}}{\sqrt{2}} + \delta_{i_1,i_2} \Bigr) \theta_{|i_1 - i_2|} - \theta_{i_1} \theta_{i_2} \Bigr].
\end{split}
 \end{equation}
Moreover, since $\Deltal^2 = \frac{\Deltaor}{n_v}$,
\begin{align*}
 \frac{4}{n_v\Deltal^2} \frac{1}{2} \frac{j}{n_t}
 &= \frac{4}{n_v\Deltal^2}  \frac{1}{2} \frac{j}{\Deltak} \frac{\Deltak}{n_t} \\
 &= \frac{\Deltaor}{n_v\Deltal^2}  2 \frac{j}{\Deltak} \\
 &=  2 \frac{j}{\Deltak}\enspace.
\end{align*}
Thus, by proposition \ref{3.prop_cov_approx1}, with probability greater than $1 - e^{-y}$,
\begin{align*}
 \text{sgn}(j)  g_n^0 \left( \frac{j}{\Deltak} \right) 
 &=   2 \frac{|j|}{\Deltak} + \frac{4}{n_v\Deltal^2} \left(1 - \frac{1}{\sqrt{2}} \right) 
 \sum_{i = k_*- (j)_- + 1}^{k_* + (j)_+} \theta_i^2  \\
 &\quad + \frac{4}{n_v\Deltal^2} \sum_{i_1 = k_*- (j)_- + 1}^{k_* + (j)_+} \sum_{i_2 = k_*- (j)_- + 1}^{k_* + (j)_+} 
 \theta_{i_1} \theta_{i_2} \frac{\theta_{|i_1 - i_2|}}{\sqrt{2}} \\
 &\quad + \frac{4}{n_v\Deltal^2} \frac{1}{2n_t} \sum_{i_1 = k_*-(j)_- + 1}^{k_* + (j)_+} \sum_{i_2 = k_*- (j)_- + 1}^{k_* + (j)_+} \theta_{|i_1 - i_2|}^2 \\
 &\quad \pm 4 \kappa_4 (y + \log n)^2 (1+x) n^{-u_3}. \numberthis \label{3.eq_g0_plus_prop}
\end{align*}
Let $g_n^1$ be defined for all $\alpha = \frac{j}{\Deltak}, j \in \mathbb{Z} \cap [-k_*(n_t);+\infty)$ by
\begin{align}
 \text{sgn}(j) g_n^1 \left( \frac{j}{\Deltak} \right) &= \frac{4}{n_v\Deltal^2} \sum_{i_1 = k_* - (j)_- + 1}^{k_* + (j)_+} 
 \sum_{i_2 = k_* - (j)_- + 1}^{k_* + (j)_+} \theta_{i_1} \theta_{i_2} \frac{\theta_{|i_1 - i_2|}}{\sqrt{2}} \nonumber \\
 &\quad + \frac{4}{n_v\Deltal^2} \left(1 - \frac{1}{\sqrt{2}} \right) \sum_{i = k_*- (j)_- + 1}^{k_* + (j)_+} \theta_i^2 \nonumber \\
 &= \frac{4}{n_v\Deltal^2} \sum_{i_1 = k_* - (j)_- + 1}^{k_* + (j)_+} 
 \sum_{i_2 = k_* - (j)_- + 1}^{k_* + (j)_+} \theta_{i_1} \theta_{i_2} \left( \frac{1 - \delta_{i_1,i_2}}{\sqrt{2}} + \delta_{i_1,i_2} \right) 
 \theta_{|i_1 - i_2|}, \label{3.eq_def_g1} 
\end{align}
and for all $\alpha \in \left[ - \frac{k_*}{\Deltak}; + \infty \right[$  by linear interpolation.

We will now apply lemma \ref{3.lem_approx_non-decreasing} to $g_n^1$.
Let $x > 0$ and $(k_1, k_2) \in \{k_* + a_x \Deltak, \ldots, k_* + b_x \Deltak \}^2 $ be such that $k_1 < k_2$. 
Thus:
\begin{itemize}
 \item If $k_* \leq k_1$,
 \begin{align*}
  g_n^1 \left( \frac{k_2 - k_*}{\Deltak} \right) - g_n^1 \left( \frac{k_1 - k_*}{\Deltak} \right) &= 
  \frac{4}{n_v \Deltal^2} \Biggl( \sum_{i = k_*+1}^{k_2} \sum_{j = k_*+1}^{k_2} 
  \theta_i \theta_j \left( \frac{1 - \delta_{i,j}}{\sqrt{2}} + \delta_{i,j} \right) \theta_{|i-j|} \\ 
&\quad -  \sum_{i = k_*+1}^{k_1} \sum_{j = k_*+1}^{k_1} 
\theta_i \theta_j \left( \frac{1 - \delta_{i,j}}{\sqrt{2}} + \delta_{i,j} \right) \theta_{|i-j|} \Biggr) \\ 
  &= \frac{4}{n_v \Deltal^2} \Biggl( \sum_{i = k_1+1}^{k_2} \sum_{j = k_1+1}^{k_2} 
  \theta_i \theta_j \left( \frac{1 - \delta_{i,j}}{\sqrt{2}} + \delta_{i,j} \right) \theta_{|i-j|} \\ 
  &\quad + 2\sum_{i = k_1+1}^{k_2} \sum_{j = k_*+1}^{k_1} \theta_i \theta_j \frac{\theta_{|i-j|}}{\sqrt{2}} \Biggr).
 \end{align*}
  By lemma \ref{3.lem_theta_posdef} in appendix,
  \[ 0 \leq  \sum_{i = k_1+1}^{k_2} \sum_{j = k_1+1}^{k_2} 
  \theta_i \theta_j \left( \frac{1 - \delta_{i,j}}{\sqrt{2}} + \delta_{i,j} \right) \theta_{|i-j|} 
  \leq \NormInfinity{s} \sum_{i = k_1+1}^{k_2} \theta_i^2. \]
  Thus by equation \eqref{3.inlem_ub_cov_bias} from claim \ref{3.lem_ub_E},
\[ 
 - \kappa_7 (1+x)^2 n^{-u_3} \frac{4 \sqrt{2} }{n_v \Deltal^2} \Deltaor \leq 
 g_n^1 \left( \tfrac{k_2 - k_*}{\Deltak} \right) - g_n^1 \left( \tfrac{k_1 - k_*}{\Deltak} \right) 
 \leq \frac{4 \NormInfinity{s}}{n_v \Deltal^2} \sum_{i = k_1+1}^{k_2} \theta_i^2 +  \kappa_7 (1+x)^2 n^{-u_4} \frac{4\sqrt{2}}{n_v \Deltal^2} \Deltaor. 
 \]
 \item If $k_2 \leq k_*$,
  \begin{align*}
  g_n^1 \left( \frac{k_2 - k_*}{\Deltak} \right) - g_n^1 \left( \frac{k_1 - k_*}{\Deltak} \right) &= 
  \frac{4}{n_v \Deltal^2} \Biggl( \sum_{i = k_1+1}^{k_*} \sum_{j = k_1+1}^{k_*} 
  \theta_i \theta_j \left( \frac{1 - \delta_{i,j}}{\sqrt{2}} + \delta_{i,j} \right) \theta_{|i-j|} \\ 
&\quad -  \sum_{i = k_2+1}^{k_*} \sum_{j = k_2+1}^{k_*} 
\theta_i \theta_j \left( \frac{1 - \delta_{i,j}}{\sqrt{2}} + \delta_{i,j} \right) \theta_{|i-j|} \Biggr) \\ 
  &= \frac{4}{n_v \Deltal^2} \Biggl( \sum_{i = k_1+1}^{k_2} \sum_{j = k_1+1}^{k_2} 
  \theta_i \theta_j \left( \frac{1 - \delta_{i,j}}{\sqrt{2}} + \delta_{i,j} \right) \theta_{|i-j|} \\ 
  &\quad + 2\sum_{i = k_1+1}^{k_2} \sum_{j = k_2+1}^{k_*} \theta_i \theta_j \frac{\theta_{|i-j|}}{\sqrt{2}} \Biggr).
 \end{align*}
 In the same way, by lemma \ref{3.lem_theta_posdef} and equation \eqref{3.inlem_ub_cov_bias} from claim \ref{3.lem_ub_E}, 
\begin{align*}
g_n^1 \left( \tfrac{k_2 - k_*}{\Deltak} \right) - g_n^1 \left( \tfrac{k_1 - k_*}{\Deltak} \right) &\geq - \kappa_7 (1+x)^2 n^{-u_3} \frac{4 \sqrt{2}}{n_v \Deltal^2} \Deltaor, \\ 
g_n^1 \left( \tfrac{k_2 - k_*}{\Deltak} \right) - g_n^1 \left( \tfrac{k_1 - k_*}{\Deltak} \right) 
&\leq \frac{4 \NormInfinity{s}}{n_v \Deltal^2} \sum_{i = k_1+1}^{k_2} \theta_i^2 +  \kappa_7 (1+x)^2 n^{-u_4} \frac{4\sqrt{2}}{n_v \Deltal^2} \Deltaor. 
\end{align*} 
 \item If $k_1 \leq k_2 \leq k_2$, 
 \[ g_n^1 \left( \tfrac{k_2 - k_*}{\Deltak} \right)  - g_n^1 \left( \tfrac{k_1 - k_*}{\Deltak} \right) = 
 g_n^1 \left( \tfrac{k_2 - k_*}{\Deltak} \right)  - g_n(0) + g_n(0) - g_n^1 \left( \tfrac{k_1 - k_*}{\Deltak} \right),  \]
 therefore by the two previous cases,
 \[0 \leq g_n^1 \left( \tfrac{k_2 - k_*}{\Deltak} \right) - g_n^1 \left( \tfrac{k_1 - k_*}{\Deltak} \right) 
 \leq \frac{4 \NormInfinity{s}}{n_v \Deltal^2} \sum_{i = k_1+1}^{k_2} \theta_i^2 + \kappa_7 (1+x)^2 n^{-u_4} \frac{8\sqrt{2}}{n_v \Deltal^2} \Deltaor. \]
\end{itemize}

By definition $\Deltal^2 = \frac{\Deltaor}{n_v}$ 
therefore for any $x > 0$ and $(j_1, j_2) \in [a_x \Deltak; b_x \Deltak]^2$
such that $j_1 \leq j_2$,
\begin{equation} \label{3.eq_bd_diff_g1_1}
 - 4\sqrt{2} \kappa_7 (1+x)^2 n^{-u_4} \leq g_n^1 \left( \tfrac{j_2}{\Deltak} \right) - g_n^1 \left( \tfrac{j_1}{\Deltak} \right)
\leq \frac{4 \NormInfinity{s}}{n_v \Deltal^2} \sum_{i = j_1+1}^{j_2} \theta_{k_* + i}^2 +  8 \sqrt{2} \kappa_7 (1+x)^2 n^{-u_4}.
\end{equation}
Moreover,
for any $j_1,j_2$ such that $0 < j_1 < j_2$, $\theta_{k_* + j_i}^2 \leq \frac{1}{n_t}$ hence
\begin{align*}
 \NormInfinity{s} \sum_{i = j_1+1}^{j_2} \theta_{k_* + i}^2 &\leq 
 \NormInfinity{s} \sum_{j = k_* + j_1+1}^{k_* + j_2} [\theta_j^2 - \frac{1}{n_t}] 
 + \frac{j_2 - j_1}{\Deltak}  \NormInfinity{s} \frac{\Deltak}{n_t}\\
 &= - \NormInfinity{s} \Deltal [f_n\left( \frac{j_2}{\Deltak}  \right) - f_n \left(\frac{j_1}{\Deltak}  \right)] 
 +  \frac{j_2 - j_1}{\Deltak}  \NormInfinity{s} \Deltaor. \numberthis \label{3.eq_ub_sumphi_d}
\end{align*}
For $j_1,j_2$ such that $j_1 < j_2 \leq 0$, $\theta_{k_* + j_i}^2 \geq \frac{1}{n_t}$ hence
\begin{align*}
 \NormInfinity{s} \sum_{i = j_1+1}^{j_2} \theta_{k_* + i}^2 &\leq 
 \NormInfinity{s} \sum_{j = k_* + j_1+1}^{k_* + j_2} [\theta_j^2 - \frac{1}{n_t}] 
 + \frac{j_2 - j_1}{\Deltak}  \NormInfinity{s} \frac{\Deltak}{n_t}\\
 &= \Deltal \NormInfinity{s} [f_n\left( \frac{j_1}{\Deltak}  \right) - f_n \left(\frac{j_2}{\Deltak}  \right)] 
 +   \frac{j_2 - j_1}{\Deltak}  \NormInfinity{s} \Deltaor. \numberthis \label{3.eq_ub_sumphi_g}
\end{align*}
By equations \eqref{3.eq_bd_diff_g1_1}, \eqref{3.eq_ub_sumphi_d} and \eqref{3.eq_ub_sumphi_d}, it follows that,
for any $(j_1,j_2) \in \bigl( [a_x \Deltak; b_x \Deltak] \cap \mathbb{Z} \bigr)^2$,
\begin{align}
 g_n^1 \left( \frac{j_2}{\Deltak} \right)  - g_n^1 \left( \frac{j_1}{\Deltak} \right) 
 &\leq - 4\frac{\NormInfinity{s}}{n_v \Deltal} \left[f_n \left( \frac{j_2}{\Deltak} \right) - f_n \left( \frac{j_1}{\Deltak} \right) \right]
 + 4 \NormInfinity{s} \frac{j_2 - j_1}{\Deltak} + 8\sqrt{2} \kappa_7 (1+x)^2 n^{-u_4} \nonumber \\ 
 &\leq - 4 \frac{\NormInfinity{s}}{n_v \Deltal}  \left[ f_n \left( \frac{j_2}{\Deltak} \right) - f_n \left( \frac{j_1}{\Deltak} \right) \right] 
 + 4\NormInfinity{s} \frac{j_2 - j_1}{\Deltak} + 8 \sqrt{2} \kappa_7 (1+x)^2 n^{-u_4}. \label{3.eq_ub_diff_g1_grid}
\end{align}

To extend the lower bound given by equation \eqref{3.eq_bd_diff_g1_1}, notice that
for any $(\alpha_1, \alpha_2) \in [a_x;b_x]^2$ such that $\alpha_1 < \alpha_2$,
\begin{itemize}
 \item if $\lfloor \alpha_1 \Deltak \rfloor = \lfloor \alpha_2 \Deltak \rfloor \leq \alpha_1 \Deltak < \alpha_2 \Deltak \leq \lfloor \alpha_1 \Deltak \rfloor + 1$,
by linearity of $g_n^1$ on $[\lfloor \alpha_1 \Deltak \rfloor; \lfloor \alpha_1 \Deltak \rfloor + 1]$,
\[ g_n^1(\alpha_2) - g_n^1 (\alpha_1) \geq - \left[ g_n^1 \left( \frac{\lfloor \alpha_1 \Deltak \rfloor + 1}{\Deltak} \right) 
- g_n^1 \left( \frac{\lfloor \alpha_1 \Deltak \rfloor}{\Deltak} \right) \right]_-, \]
\item otherwise, $\lfloor \alpha_1 \Deltak \rfloor + 1 \leq \lfloor \alpha_2 \Deltak \rfloor$, therefore by linearity of 
$(u,v) \mapsto g_n^1(u) - g_n^1(v)$ on $\frac{1}{\Deltak} \left[\lfloor \alpha_1 \Deltak \rfloor;\lfloor \alpha_1 \Deltak \rfloor+1 \right] 
\times \frac{1}{\Deltak} [\lfloor \alpha_2 \Deltak \rfloor; \lfloor \alpha_2 \Deltak \rfloor + 1]$,
\begin{align*}
   g_n^1(\alpha_2) - g_n^1(\alpha_1) \geq \min \Bigl\{& g_n^1(u) - g_n^1(v) : \\ 
   &(u,v) \in \bigl\{ \frac{\lfloor \alpha_2 \Deltak \rfloor}{\Deltak}
; \frac{\lfloor \alpha_2 \Deltak \rfloor + 1}{\Deltak} \bigr\} \times \bigl\{ \frac{\lfloor \alpha_1 \Deltak \rfloor}{\Deltak}
; \frac{\lfloor \alpha_1 \Deltak \rfloor + 1}{\Deltak} \bigr\} \Bigr\}. 
\end{align*}
\end{itemize}
In all cases,
\begin{equation} \label{3.eq_ext_lb_g1}
 g_n^1(\alpha_2) - g_n^1 (\alpha_1) \geq - \max \left\{ \left[ g_n^1 \left( \frac{j_2}{\Deltak} \right) 
 - g_n^1 \left( \frac{j_1}{\Deltak} \right) \right]_- : j_1 \leq j_2, (j_1,j_2) \in \{a_x \Deltak, \ldots, b_x \Deltak \}^2 \right\}.
\end{equation} 
Thus by equation \eqref{3.eq_bd_diff_g1_1}, for any $x > 0$:
\begin{equation} \label{3.eq_lb_diff_g1}
 \forall (\alpha_1,\alpha_2) \in [a_x;b_x]^2, \alpha_1 < \alpha_2 \implies 
 g_n^1 (\alpha_2) - g_n^1(\alpha_1) \geq - 4 \sqrt{2} \kappa_7 (1+x)^2 n^{-u_4}.
\end{equation}
By the same argument applied to the function 
\[\alpha \mapsto g_n^1(\alpha) + 4\frac{\NormInfinity{s}}{n_v \Deltal}  f_n(\alpha) - 4 \NormInfinity{s} \alpha, \]
which is piecewise linear on the partition $\left\{ \left[ \frac{j}{\Deltak}; \frac{j+1}{\Deltak} \right[ : j \in \{a_x \Deltak, \ldots, b_x \Deltak \} \right\}$,
equation \eqref{3.eq_ub_diff_g1_grid} extends to $[a_x;b_x]$ for any $x > 0$:
\begin{equation} \label{3.eq_ub_diff_g1}
\begin{split}
    \forall (\alpha_1,\alpha_2) \in [a_x;b_x]^2, g_n^1 (\alpha_2) - g_n^1(\alpha_1) 
 &\leq 4\frac{\NormInfinity{s}}{n_v \Deltal}  [f_n(\alpha_1) - f_n(\alpha_2)] + 4 \NormInfinity{s}[\alpha_2 - \alpha_1] \\
 &\quad + 8 \sqrt{2} \kappa_7 (1+x)^2 n^{-u_4}. 
\end{split}
\end{equation}
Let $\varepsilon_d : \alpha \mapsto \inf_{x \in \mathbb{R}_+: b_x \geq \alpha} 8 \sqrt{2} \kappa_7 (1+x)^2 n^{-u_4}$.
The function $\varepsilon_d$ is non-decreasing by definition, and $\varepsilon_d(0) \geq 8\kappa_7 n^{-u_4} > 0$.
Furthermore, by  equations \eqref{3.eq_lb_diff_g1} and \eqref{3.eq_ub_diff_g1}, 
\begin{align*}
  &\forall (\alpha_1,\alpha_2) \in \mathbb{R}_+^2, \alpha_1 \leq \alpha_2 \implies  \\
  &\quad -\varepsilon_d(\alpha_2) 
\leq g_n^1 (\alpha_2) - g_n^1(\alpha_1) \leq 4\frac{\NormInfinity{s}}{n_v \Deltal}  [f_n(\alpha_1) - f_n(\alpha_2)] + 4\NormInfinity{s}[\alpha_2 - \alpha_1] 
+ \varepsilon_d(\alpha_2). 
\end{align*}
In this situation, lemma \ref{3.lem_approx_non-decreasing} applies with $g = g_n^1$,
$h_+ = - \frac{8 \NormInfinity{s}}{n_v \Deltal} f_n + 8 \NormInfinity{s} \text{Id}$ and $\varepsilon = 2\varepsilon_d$. Note that $h_+$ is indeed non-decreasing since by equation \eqref{eq_deriv_fn}, $f_n' \leq \frac{\Deltak}{n_t \Deltal}$ on $\mathbb{R}_+$ which leads to $\frac{f_n'}{n_v \Deltal} \leq \frac{\Deltak}{n_v n_t \Deltal^2} = 1$.
This guarantees existence of a non-decreasing function $g_{n,+}^2: \mathbb{R}_+ \rightarrow \mathbb{R}_+$ such that $g_{n,+}^2(0) = 0$,
\[ \sup_{\alpha \in \mathbb{R}_+} \frac{|g_n^1(\alpha) - g_{n,+}^2(\alpha)|}{2\varepsilon_d(\alpha)} \leq 6 \]
and for all $\alpha_1,\alpha_2$ such that $\alpha_1 \leq \alpha_2$,
\[ g_{n,+}^2(\alpha_2) - g_{n,+}^2(\alpha_1) \leq 8 \frac{\NormInfinity{s}}{n_v \Deltal}  [f_n(\alpha_1) - f_n(\alpha_2)] + 8 \NormInfinity{s}[\alpha_2 - \alpha_1] \]
Symetrically, let $\varepsilon_g: \alpha \mapsto \inf_{x \in \mathbb{R}_+: - a_x \geq \alpha} 8 \sqrt{2} \kappa_7 (1+x)^2 n^{-u_4}$, defined on $\bigl[0; \tfrac{k_*(n_t)}{\Deltak} \bigr]$.
$\varepsilon_g$ is non-decreasing by definition. Furthermore, $\varepsilon_g(0) \geq 8\kappa_7 n^{-u_4} > 0$. 
By equations \eqref{3.eq_lb_diff_g1} and \eqref{3.eq_ub_diff_g1}, 
\begin{align*}
  \forall (\alpha_1,\alpha_2) \in \mathbb{R}_+^2, \alpha_1 \leq \alpha_2 \implies -\varepsilon_g(\alpha_2) 
&\leq g_n^1 ( - \alpha_1) - g_n^1(- \alpha_2) \\
&\leq 4\frac{\NormInfinity{s}}{n_v \Deltal}  [f_n(- \alpha_2) - f_n(- \alpha_1)] \\ 
&\quad + 4\NormInfinity{s}[\alpha_2 - \alpha_1] 
+ \varepsilon_g(\alpha_2). 
\end{align*}
In this situation, lemma \ref{3.lem_approx_non-decreasing} applies with $g = - g_n^1 (-\cdot)$,
$h_+ = \frac{8 \NormInfinity{s}}{n_v \Deltal} f_n(-\cdot) + 8 \NormInfinity{s} \text{Id}$, $\varepsilon = 2\varepsilon_g$.
It guarantees existence of a function $g_{n,-}^2: \bigl[0; \tfrac{k_*(n_t)}{\Deltak} \bigr] \rightarrow \mathbb{R}_+$ such that $g_{n,-}^2(0) = 0$,
\[ \sup_{\alpha \in \left[0; \tfrac{k_*(n_t)}{\Deltak} \right]} \frac{| - g_n^1(-\alpha) - g_{n,-}^2(\alpha)|}{2\varepsilon_g(\alpha)} \leq 6 \]
and for any $\alpha_1,\alpha_2$ such that $\alpha_1 \leq \alpha_2$,
\[ g_{n,-}^2(\alpha_2) - g_{n,-}^2(\alpha_1) \leq 8 \frac{\NormInfinity{s}}{n_v \Deltal}  [f_n(- \alpha_2) - f_n(- \alpha_1)] + 8 \NormInfinity{s}[\alpha_2 - \alpha_1]. \] 
Let then $g_n^2: \alpha \mapsto g_{n,+}^2(\alpha) \mathbb{I}_{\alpha \geq 0} - g_{n,-}^2(-\alpha) \mathbb{I}_{\alpha < 0}$ 
and $\varepsilon(\alpha) = \varepsilon_d(\alpha) \mathbb{I}_{\alpha \geq 0} + \varepsilon_g(- \alpha) \mathbb{I}_{x < 0}$, which yields
\begin{equation}
 \NormInfinity{\frac{g_n^2 - g_n^1}{2\varepsilon}} \leq 6
\end{equation}
and 
\begin{equation} \label{3.eq_ub_diff_g2}
 \forall (\alpha_1,\alpha_2) \in \mathbb{R}^2, g_n^2(\alpha_2) - g_n^2(\alpha_1) 
 \leq 8 \frac{\NormInfinity{s}}{n_v \Deltal}  [f_n(\alpha_1) - f_n(\alpha_2)] + 8 \NormInfinity{s}[\alpha_2 - \alpha_1].
\end{equation}
By definition of $\varepsilon$, for any $x > 0$ and any $\alpha \in [a_x,b_x]$,
$\varepsilon(\alpha) \leq 8 \sqrt{2}\kappa_7 (1+x)^2 n^{-u_4}$, hence
\begin{equation} \label{3.eq_ub_g1_g2}
 \forall x > 0, \forall \alpha \in [a_x;b_x], |g_n^2(\alpha) - g_n^1(\alpha)| \leq 96 \sqrt{2} \kappa_7 (1+x)^2 n^{-u_4}.
\end{equation}

 Let then:
 \begin{equation} \label{3.eq_def_g}
  g_n: \alpha \mapsto g_n^2(\alpha) + 4 \Norm{s}^2 \alpha .
 \end{equation}
Since $g_n^2$ is non-decreasing, $g_n(\alpha_2) - g_n(\alpha_1) \geq 4 \Norm{s}^2 [\alpha_2 - \alpha_1]$, which proves
equation \ref{3.inthm_lb_diff_gn} of Theorem \ref{3.thm_approx_ho}. 
Moreover, equation \eqref{3.eq_ub_diff_g2} yields equation \eqref{3.eq.inthm_ub_diff_gn} of Theorem \ref{3.thm_approx_ho}.

Let now $x > 0$ be fixed until the end of this section.
By definition of $g_n^1$ (equation \eqref{3.eq_def_g1}), equations \eqref{3.eq_g0_plus_prop}, \eqref{3.eq_ub_g1_g2}
and since the functions $g_n^0$
and $\alpha \mapsto 4 \Norm{s}^2 \alpha$ are piecewise linear on the partition 
$\left(\left[ \frac{j}{\Deltak}, \frac{j+1}{\Deltak} \right) \right)_{a_x \Deltak \leq j \leq b_x \Deltak - 1 }$, 
with probability greater than $1 - e^{-y}$,
\begin{align*}
 \NormInfinity{g_n^0 - g_n} &\leq \NormInfinity{g_n^1 - g_n^2} + \kappa_4 (y + \log n)^2 (1+x) n^{-u_3} \\
 &\quad + \max_{a_x \Deltak \leq j \leq b_x \Deltak} \left| \frac{\text{sgn}(j)}{n_v\Deltal^2} \frac{4}{2n_t} \sum_{i_1 = k_* + (j)_- + 1}^{k_* + (j)_+} 
 \sum_{i_2 = k_* + (j)_- + 1}^{k_* + (j)_+} \theta_{|i_1 - i_2|}^2 - \frac{(4\Norm{s}^2 - 2) j}{\Deltak} \right| \\
 &\leq \max_{a_x \Deltak \leq j \leq b_x \Deltak} \left| \frac{4}{n_v\Deltal^2} \frac{1}{2 n_t} \sum_{i_1 = k_* + (j)_- + 1}^{k_* + (j)_+} 
 \sum_{i_2 = k_* + (j)_- + 1}^{k_* + (j)_+} \theta_{|i_1 - i_2|}^2 - 4 \left( \Norm{s}^2 - \tfrac{1}{2} \right) \frac{|j|}{\Deltak} \right| \\
 &\quad + 96 \sqrt{2} \kappa_7 (1+x)^2 n^{-u_4} + \kappa_4 (y + \log n)^2 (1+x) n^{-u_3} \numberthis \label{3.eq_approx_g_g0_1}.
\end{align*}
It remains to bound the $\max$. 
By parity in $j$ of the sum, one can assume $0 \leq j \leq \max(|a_x|, |b_x|) \Deltak$ instead of $a_x \Deltak \leq j \leq b_x \Deltak|]$.
Let therefore $j \in \{0,\ldots,\max(|a_x|, |b_x|) \Deltak\}$, then
\begin{align*}
 \frac{1}{2n_t} \sum_{i_1 = k_* + 1}^{k_* + j} \sum_{i_2 = k_* + 1}^{k_* + j} \theta_{|i_1 - i_2|}^2 &= 
 \frac{|j|}{2n_t} + \frac{1}{2n_t} \sum_{r \in \mathbb{N}^*} 2 \left| \left\{i : k_* \leq i \leq i + r \leq k_* + j\right\} \right| \theta_r^2 \\
 &= \frac{|j|}{2n_t} + \frac{1}{n_t} \sum_{r = 1}^{+ \infty} (j - r)_+  \theta_r^2. \numberthis \label{3.eq_simplif_double_sum}
\end{align*}
Furthermore, for all $r_0 \in \mathbb{N}^*$,
\begin{align*}
 \left| \frac{1}{n_t} \sum_{r = 1}^{+ \infty} (j - r)_+ \theta_r^2 - \frac{j}{n_t} (\Norm{s}^2 - 1) \right| &\leq 
 \frac{1}{n_t} \sum_{r = 1}^{+ \infty} \theta_r^2 \left| (j-r)_+ - j \right| \\
 &\leq \frac{r_0}{n_t} \sum_{r = 1}^{r_0} \theta_r^2 + \frac{j}{n_t} \sum_{r = r_0 + 1}^{+ \infty} \theta_r^2 \\
 &\leq \Norm{s}^2 \frac{r_0}{n_t} + \max (|a_x|, |b_x|) \frac{\Deltak}{n_t} \times \frac{c_1}{r_0^2},
\end{align*}
by hypothesis \ref{3.inthm_hyp_ub_sum_varphi} of section \ref{3.sec.hyp}. By setting  
$r_0 = \lceil (\Deltak)^{\frac{1}{3}} \rceil \leq 2 (\Deltak)^{\frac{1}{3}}$ (since $\Deltak \geq 1$), it follows that
\begin{align*}
 \left| \frac{1}{n_t} \sum_{r = 1}^{+ \infty} (j - r)_+ \theta_r^2 - \frac{j}{n_t} (\Norm{s}^2 - 1) \right| &\leq 
 \left[2\Norm{s}^2 +  c_1 \max (|a_x|, |b_x|) \right] \frac{(\Deltak)^{\frac{1}{3}}}{n_t} \\
 &\leq \left[2\Norm{s}^2 + 2 c_1(1+x) \right] (\Deltak)^{- \frac{2}{3}} \Deltaor \text{ by  lemma } \ref{3.lem_ub_fen_risk}.
\end{align*}
Let $\kappa = \kappa(c_1, \Norm{s})$.
Since by hypothesis \ref{3.inthm_hyp_ub_nv} of Theorem \ref{3.thm_approx_ho}, 
$\Deltak \geq \frac{n_t}{n - n_t} \geq n^{\delta_4}$, 
\[ \left| \frac{1}{n_t} \sum_{r = 1}^{+ \infty} (j - r)_+ \theta_r^2 - \frac{j}{n_t} (\Norm{s}^2 - 1) \right|  \leq \kappa (1+x) n^{- \frac{2\delta_4}{3}} \Deltaor. \]
By equation \eqref{3.eq_simplif_double_sum} and since $\frac{j}{n_t} = \frac{j}{\Deltak} \Deltaor$, 
for any $j \in [0; \max(|a_x|, |b_x|) \Deltak]$.
\[
 \left| \frac{1}{2n_t} \sum_{i_1 = k_* + 1}^{k_* + j} \sum_{i_2 = k_* + 1}^{k_* + j} \theta_{|i_1 - i_2|}^2 
 - \frac{j}{\Deltak} \left(\Norm{s}^2 - \frac{1}{2} \right) \Deltaor \right|
 \leq \kappa (1+x) n^{- \delta_4} \Deltaor.
\]
By equation \eqref{3.eq_approx_g_g0_1} and since $\frac{\Deltaor}{n_v\Deltal^2} = 1$, 
it follows that, with probability greater than $1 - e^{-y}$,
\begin{equation}
 \NormInfinity{g_n^0 - g_n} \leq 96 \sqrt{2} \kappa_7 (1+x)^2 n^{-u_4} + \kappa_4 (y + \log n)^2 (1+x) n^{-u_3} + 4 \kappa (1+x) n^{- \delta_4}.
\end{equation}
Let $\kappa = 96 \sqrt{2} \kappa_7 + \kappa_4 + 4\kappa$ and $u_5 = \min(u_4,u_3,\delta_4)$, it then follows from  definition \ref{3.def_K} of $K$ 
that with probability greater than $1 - e^{-y}$,
\[ \NormInfinity{K(g_n^0) - K(g_n)} \leq \NormInfinity{g_n^0 - g_n} \leq \kappa (1+x)^2 (y + \log n)^2 n^{-u_5}. \]
By equation \eqref{3.eq_approx_g0}, it follows that that, with probability greater than $1 - e^{-y}$, 
for any $(j_1,j_2) \in \{a_x \Deltak, \ldots, b_x \Deltak \}^2$,
\begin{align*}
\left| \Cov \left( Z \left( \frac{j_1}{\Deltak} \right), Z \left( \frac{j_2}{\Deltak} \right) \right)
 - K(g_n) \left( \frac{j_1}{\Deltak}, \frac{j_2}{\Deltak} \right) \right| &\leq 4\kappa_7 (1+x)^2 (y+ \log n)^2  n^{-u_3} \\ 
 &\quad + \kappa (1+x)^2 (y + \log n)^2 n^{-u_5} \\
 &\leq \kappa (1+x)^2 (y + \log n)^2 n^{-u_5} \numberthis \label{3.eq_approx_g_grid} 
\end{align*}
by setting $\kappa = \kappa + 4\kappa_7$ and since  $u_5 \leq u_3$.
Claim \ref{3.claim_def_g} follows. 
It remains to upper bound $g_n$ on $[a_x;b_x]$ in order to check equation \ref{3.inthm_ub_gn} of Theorem \ref{3.thm_approx_ho}.
This is the subject of the following lemma.
\begin{lemma} \label{3.lem_ub_gn}
 For any $\alpha \in \mathbb{R}$, 
 \[ |g_n(\alpha)| \leq 20 \NormInfinity{s} f_n(\alpha) + 12 \NormInfinity{s}  \leq  \max \left( 40 \NormInfinity{s} f_n(\alpha), 24 \NormInfinity{s} \right) . \]
In particular,
for all $x > 0$, $\max(|g_n(a_x)|,|g_n(b_x)|) \leq 20 \NormInfinity{s} (1+x)$.
\end{lemma}

\begin{proof}
Since $\NormInfinity{s} \geq \Norm{s}^2 \geq 1$ and $n_v \Deltal \leq 1$ by lemma \ref{3.lem_odgs}, point \ref{3.inthm_ub_diff_gn} of Theorem \ref{3.thm_approx_ho} which we 
already proved implies that for any $\alpha \in \mathbb{R}$,  
\[ |g_n(\alpha))| \leq 8 \NormInfinity{s} f_n(\alpha) +  12\NormInfinity{s} |\alpha|. \] If $|\alpha| < 1$, then
\[ |g_n(\alpha))| \leq 8 \NormInfinity{s}|f_n(\alpha)|  + 12 \NormInfinity{s} \leq \max \left( 16 \NormInfinity{s} f_n(\alpha), 24 \NormInfinity{s} \right), \] 
else
 $|f_n(\alpha) - f_n(1)| \geq |\alpha| - 1$, therefore $|\alpha| \leq f_n(\alpha) + 1$, which yields 
 \[ |g_n(\alpha))| \leq 20 \NormInfinity{s} f_n(\alpha) + 12 \NormInfinity{s} \leq \max \left( 40 \NormInfinity{s} f_n(\alpha), 24 \NormInfinity{s} \right) . \]
\end{proof}

\subsection{Construction of a Wiener process \texorpdfstring{$W$}{W} such that \texorpdfstring{$W \circ g_n$}{W o gn} approximates \texorpdfstring{$Z$}{Z}}
Let $E_y$ be the event of probability greater than $1 - e^{-y}$ on which the equations of claim \ref{3.claim_def_g}
are satisfied. Let $x > 0$.
Given $D_n^T \in E_y$, $Z^1$ is a piecewise linear gaussian process on the partition 
$([\tfrac{j}{\Deltak}; \tfrac{j+1}{\Deltak}))_{a_x \Deltak \leq j \leq b_x \Deltak}$, such that for any $j \in \{a_x \Deltak,\ldots,b_x \Deltak\}$,
\begin{equation}
  \max_{(j_1, j_2) \in \{0,\ldots,b_x \Deltak\}^2} \left| \Cov \left(Z \left( \tfrac{j_1}{\Deltak} \right), Z \left( \tfrac{j_2}{\Deltak} \right) \right) 
  - K(g_n) \left(\tfrac{j_1}{\Deltak}, \tfrac{j_2}{\Deltak} \right) \right|
  \leq \kappa_6(1+x)^2 [y + \log n]^2 n^{-u_5},
\end{equation}
where $K(g_n)$ is given by definition \ref{3.def_K}.
Since $g_n$ is non-decreasing, $K(g_n)(s,t) = \Cov(W_{g_n(s)}, W_{g_n(t)})$ for any two-sided Wiener process $W$ on $\mathbb{R}$ 
such that $W_0 = 0$. In particular, $K(g_n)$ is a positive-definite function.
Furthermore, by definition, $\forall (\alpha_1,\alpha_2) \in [a_x;b_x]^2$, 
\[K(g_n)(\alpha_1,\alpha_1) + K(g_n)(\alpha_2,\alpha_2) - 2K(g_n)(\alpha_1,\alpha_2) 
= |g_n(\alpha_2) - g_n(\alpha_1)|. \]
Moreover, for all $j \in \{a_x \Deltak,\ldots, b_x \Deltak - 1\}$, since $n_v \Deltal \leq 1$, 
\begin{align*}
|g_n(\tfrac{j+1}{\Deltak}) - g_n(\tfrac{j}{\Deltak})| &\leq 8\NormInfinity{s} |f_n(\tfrac{j+1}{\Deltak}) - f_n(\tfrac{j}{\Deltak})| 
+ \frac{12 \NormInfinity{s}}{\Deltak} \text{ by equations } \eqref{3.eq_def_g} \text{ and } \eqref{3.eq_ub_diff_g2} \\
&\leq 8 \kappa_3 \NormInfinity{s} (1+x)^2 n^{-u_2} + 12 \NormInfinity{s} \frac{n - n_t}{n_t} \text{ by claim } \ref{3.lem_ub_varphi_fen}  \\
&\leq 8 \kappa_3 \NormInfinity{s} (1+x)^2 n^{-u_2} + 12 \NormInfinity{s} n^{-\delta_4} 
\text{ by hypothesis } \ref{3.inthm_hyp_ub_nv}.
\end{align*}
Finally, by  lemma \ref{3.lem_ub_gn} and since $g_n$ is non-decreasing, 
\[ \sup_{\alpha \in [a_x;b_x]} K(g_n)(\alpha,\alpha) \leq \max(|g_n(a_x)|,|g_n(b_x)|) \leq 20 \NormInfinity{s} (1+x). \]
In this situation, proposition \ref{3.prop_approx_proc_gauss} in the appendix (applied to $Y = Z^1$, $K_X = K(g_n)$
with $h = g_n$) guarantees the existence of a continuous gaussian process $Z^2(D_n^T)$,
with variance-covariance function $K(g_n)$ and such that for some constant $\kappa$ and for $u = \min (u_5,u_2,\delta_4)$,
\begin{equation} \label{3.eq_approx_Z1_Z2}
\forall y > 0, \forall D_n^T \in E_y, \mathbb{E} \left[ \sup_{a_x \leq t \leq b_x} |Z^1(t) - Z^2(t)| |D_n^T  \right] \leq \kappa (1+x)^{\frac{7}{6}} 
[y + \log n]^{\frac{2}{3}} \times n^{-\frac{u}{12}}.
\end{equation}
Since the conditional distribution of $Z^2(D_n^T)$ given $D_n^T$ is entirely determined by the function $g_n$ which does not depend on $D_n^T$, 
$Z^2$ is independent from $D_n^T$. 
Moreover, since $g_n$ increases, $W = Z^2 \circ g_n^{-1}$ is a continuous, centered gaussian process with 
covariance function
\begin{equation}
\Cov(Z_s,Z_t) = K(g_n)(g_n^{-1}(s), g_n^{-1}(t)) = \begin{cases}
                                                    &s \wedge t \text{ if } 0 \leq s,t \\
                                                    & - (s \vee t) \text{ if } s,t \leq 0 \\
                                                    &0 \text{ else },
                                                   \end{cases}
\end{equation}
it is therefore a two-sided Wiener process on $[g_n(a_x);g_n(b_x)]$ taking value $0$ at $0$. $W$ can be extended
to $\mathbb{R}$ by placing independent Wiener processes $W_g,W_d$ on its left and on its right, by the equations
$W(u) = W(g_n(a_x)) + W_g(u) - W_g(g_n(a_x))$ for $u <a_x$, $W(u) = W(g_n(b_x)) + W_d(u) - W_d(g_n(b_x))$ for $u > b_x$.
Thus, by  claim \ref{3.claim_strong_approx_proc_emp} and equation \eqref{3.eq_approx_Z1_Z2}, 
with probability greater than $1 - 2e^{-y}$,
\begin{align*}
 \mathbb{E} \left[ \sup_{a_x \leq t \leq b_x} |Z(t) - W_{g_n(t)}| |D_n^T  \right] 
 &= \mathbb{E} \left[ \sup_{a_x \leq t \leq b_x} |Z(t) - Z^2(t)| |D_n^T  \right] \\ 
 &\leq 
 \mathbb{E} \left[ \sup_{a_x \leq t \leq b_x} |Z^1(t) - Z(t)| | D_n^T \right] + 
 \mathbb{E} \left[ \sup_{a_x \leq t \leq b_x} |Z^1(t) - W_{g_n(t)}| |D_n^T \right] \\
 &\leq \kappa (1+x)^{\frac{7}{6}} [y + \log n]^{\frac{2}{3}} n^{-\frac{u_2}{12}} + \kappa_5(c_1) (1+ y) (1+x)^{\frac{3}{2}} n^{-\frac{\delta_5}{3}} \\
 &\leq \kappa (1+y) (1+x)^{\frac{3}{2}} n^{-u},
\end{align*}

for all $u < \min \left(\frac{u_5}{12},\frac{u_2}{12}, \frac{\delta_4}{12},  \frac{\delta_5}{3}  \right)$ and a constant $\kappa(u)$.
Finally, by  claim \ref{3.claim_approx_ex_risk},
with probability greater than $1 - 3e^{-y}$,
\begin{align*}
 &\sup_{\alpha \in [a_x;b_x]} \left| \ho{T}(\alpha) - [f_n(\alpha) - W_{g_n(\alpha)}] \right| \\ 
 &\quad \leq 
 \sup_{\alpha \in [a_x;b_x]} \left| L(\alpha) 
 - f_n(\alpha) \right| 
   + \sup_{\alpha \in [a_x;b_x]} |Z(\alpha) - W_{g_n(\alpha)}| \\
 &\quad \leq \kappa_1 (1+x)[\log(n) + \log(2+x) + y]^2 n^{- \min(\frac{1}{12}, \frac{\delta_4}{2})} + \kappa (1+y) (1+x)^{\frac{3}{2}} n^{-u} \\
 &\quad \leq \kappa (1+y)^2 (1+x)^{\frac{3}{2}} n^{-u_1},
\end{align*}
for all $u_1 < \min \left(\frac{u_5}{12},\frac{u_2}{12}, \frac{\delta_4}{12},  \frac{\delta_5}{3} \right)$
and a constant $\kappa$. 
This proves Theorem \ref{3.thm_approx_ho}.


\section{Appendix}
\begin{lemma} \label{3.lem_lim_var}
 Let $X$ be a random variable belonging to $[-1;1]$, with pdf 
 $s$. For all $j \in \mathbb{N}$, let $\theta_j = \langle s, \psi_j \rangle$. 
 Then 
 \begin{align*}
  \Var \left( \psi_j(X) \right) &\underset{j \to + \infty}{\longrightarrow} 1 \\
  \forall k_0 \leq k, \sum_{j = k_0}^k \left|\Var(\psi_j) - 1 \right| &\leq \Norm{\theta}_{\ell^1} = \sum_{j = 0}^{+ \infty} |\langle s, \psi_j \rangle|.
 \end{align*}
\end{lemma}

\begin{proof}
  $\mathbb{E} [\psi_j(X)] = \int_{0}^1 \psi_j(x)  s(x) dx = \theta_j$. Moreover,
 $\psi_j(X)^2 =  2 \cos^2(2 \pi j X) =  1 + \cos(2 \pi j X)$, therefore 
 $\Var(\cos(\pi jX)) = 1 + \frac{\theta_j}{\sqrt{2}} - \theta_j^2$, therefore since $|\theta_j| \leq \sqrt{2}$, 
 $|\Var(\cos(jX)) - 1| \leq \left|\sqrt{2} - \frac{1}{\sqrt{2}} \right| |\theta_j| \leq |\theta_j|$.
\end{proof}

\begin{lemma} \label{3.lem_ineq_inf}
 Let $f: \mathbb{R_+} \rightarrow \mathbb{R}_+$ be a function, $g,h : \mathbb{R}_+ \rightarrow \mathbb{R}$
 be two non-increasing functions. 
 Then 
 \[ \inf_{x \in \mathbb{R}_+} \left\{ f(x) + g(x) + h(x) \right\} \leq \inf_{x \in \mathbb{R}_+} \left\{ f(x) + g(x) \right\} 
 + \inf_{x \in \mathbb{R}_+} \left\{ f(x) + h(x) \right\}. \]
\end{lemma}

\begin{proof}
 Let $\delta > 0$. Let $x_g$ be such that $f(x_g) + g(x_g) \leq \delta + \inf_{x \in \mathbb{R}_+} \left\{ f(x) + g(x) \right\}$.
 Let $x_h$ be such that $f(x_h) + h(x_h) \leq \inf_{x \in \mathbb{R}_+} \left\{ f(x) + h(x) \right\}$.
 Let $x_* = \max(x_g,x_h)$.
 If $x_* = x_g$, then
  \begin{align*}
   f(x_*) + g(x_*) + h(x_*) &\leq \inf_{x \in \mathbb{R}_+} \left\{ f(x) + g(x) \right\} + \delta + h(x_*) \\
   &\leq  \inf_{x \in \mathbb{R}_+} \left\{ f(x) + g(x) \right\} + \delta + h(x_h) \text{ since } h \text{ is non-increasing} \\
   &\leq \inf_{x \in \mathbb{R}_+} \left\{ f(x) + g(x) \right\} + \delta + f(x_h) + h(x_h) \\
   &\leq \inf_{x \in \mathbb{R}_+} \left\{ f(x) + g(x) \right\} + \inf_{x \in \mathbb{R}_+} \left\{ f(x) + h(x) \right\} + 2\delta
  \end{align*}
  Symetrically, if $x_* = x_h$,
  then $f(x_*) + g(x_*) + h(x_*) \leq \inf_{x \in \mathbb{R}_+} \left\{ f(x) + g(x) \right\} + \inf_{x \in \mathbb{R}_+} \left\{ f(x) + h(x) \right\} + 2\delta$.
As a result, 
\begin{align*}
   \inf_{x \in \mathbb{R}_+} \left\{ f(x) + g(x) + h(x) \right\} \leq f(x_*) + g(x_*) + h(x_*) &\leq \inf_{x \in \mathbb{R}_+} \left\{ f(x) + g(x) \right\} \\
&\quad + \inf_{x \in \mathbb{R}_+} \left\{ f(x) + h(x) \right\} + 2\delta.  
\end{align*}
Since no assumptions were made about $\delta > 0$, lemma \ref{3.lem_ineq_inf} is proved.
\end{proof}

\begin{proposition} \label{3.prop_approx_rsk}
For any integers $k_0 \leq k$, 
with probability greater than $1 - e^{-y}$:
\[ \left| \sum_{j = k_0 + 1}^{k} (\hat{\theta}_j^T - \theta_j)^2 - \frac{|k - k_0|}{n_t} \right| \leq \frac{\Norm{\theta}_{\ell^1}}{n_t} 
+ C \sqrt{y + \log n} \left[ \frac{\sqrt{|k - k_0| }}{n_t} + \frac{|k - k_0|}{n_t^{\frac{5}{4}}} \right].\]
In particular, there exists a constant $\kappa_1 = \kappa_1(\NormInfinity{s}, c_1, \Norm{\theta}_{\ell^1})$ such that 
for any $\alpha_1, \alpha_2$ such that $(\alpha_1 \Deltak, \alpha_2 \Deltak) \in \mathbb{N}^2$ and $\alpha_1 < \alpha_2$,   
with probability greater than $1 - e^{-y}$,
\begin{equation}
 \left| \sum_{j = k_* + \alpha_1 \Deltak}^{k_* + \alpha_2 \Deltak} (\hat{\theta}_j^T - \theta_j)^2 - [\alpha_2 - \alpha_1]\Deltaor \right| 
\leq  \kappa_1 (\alpha_2 - \alpha_1) [\log n + y]^2 \times n^{- \min(\frac{1}{12}, \frac{\delta_4}{2})} \Deltal (n). \label{3.eq_var_odg}
\end{equation}
\end{proposition}

\begin{proof} 
Let $(k_0,k) \in \mathbb{N}^2$ be such that $k_0 < k$.
 The proof rests on lemma 14 of \citet{Arl_Ler:2012:penVF:JMLR} 
 applied to $S_m = \langle \psi_{k_0+1}, \ldots, \psi_k \rangle$. 
Let us compute $b_m = \sup_{u \in \mathbb{R}^{|k - k_0|}: 
\Norm{u} \leq 1} \sum_{j = k_0}^k u_j \psi_j(x) \leq \sup_x \sqrt{\sum_{j = k_0}^k \psi_j^2(x)} \leq \sqrt{|k-k_0|}$ and
\[ \mathcal{D}_k = \sum_{j = k_0 + 1}^k \Var \left( \psi_j(X) \right) = |k-k_0| \pm \frac{\Norm{\theta}_{\ell^1}}{n_t} \]
(by lemma \ref{3.lem_lim_var}). Furthermore, $\mathcal{D}_k \leq \sqrt{2} |k - k_0|$ since $\psi_j = \sqrt{2} \cos(2\pi j \cdot): 
[0;1] \rightarrow [-\sqrt{2}; \sqrt{2}]$.
By \cite[lemma 14]{Arl_Ler:2012:penVF:JMLR}, with probability greater than $1 - e^{-y}$, for any $\varepsilon > 0$, 
\[  \Bigl| \sum_{j = k_0 + 1}^k (\hat{\theta}_j^T - \theta_j)^2 - \frac{\mathcal{D}_k}{n_t}  \Bigr| 
\leq \varepsilon \frac{\mathcal{D}_k}{n_t} + \kappa \left( \frac{\NormInfinity{s}[\log n + y]}{(\varepsilon \wedge 1)n_t} 
+ \frac{|k - k_0|[\log n + y]^2}{(\varepsilon \wedge 1)^3 n_t^2}\right). \]
Let $\varepsilon_1 = \sqrt{\frac{\NormInfinity{s} (\log n + y)}{|k - k_0|}} \wedge 1$ .
If $\varepsilon_1 = 1$, then $|k - k_0| \leq \NormInfinity{s} (y + \log n)$ therefore 
$\varepsilon_1 \frac{|k - k_0|}{n_t} + \kappa \frac{\NormInfinity{s}[\log n + y]}{(\varepsilon_1 \wedge 1)n_t} \leq 
(1+\kappa)\frac{\NormInfinity{s}(y + \log n)}{n_t}$. 
If $\varepsilon_1 < 1$, then $\varepsilon_1 \frac{|k - k_0|}{n_t} + \kappa \frac{\NormInfinity{s}[\log n + y]}{(\varepsilon_1 \wedge 1)n_t} 
= (1 + \kappa) \sqrt{\NormInfinity{s}(y + \log n)} \frac{\sqrt{|k - k_0|}}{n_t}$.
In all cases, if $k > k_0$,
\begin{equation} \label{3.eq_eps_1}
 \varepsilon_1 \frac{|k - k_0|}{n_t} + \kappa \frac{\NormInfinity{s}[\log n + y]}{(\varepsilon_1 \wedge 1)n_t} 
 \leq (1 + \kappa) \NormInfinity{s} (y + \log n) \frac{\sqrt{|k - k_0|}}{n_t}.
\end{equation}

Let $\varepsilon_2 = \frac{\sqrt{\log n + y}}{n_t^{\frac{1}{4}}} \wedge 1$.
If $\frac{\sqrt{y + \log n}}{n_t^{\frac{1}{4}}} \geq 1 = \varepsilon_2$, then 
$\varepsilon_2 \frac{|k - k_0|}{n_t} + \kappa \frac{|k - k_0|[\log n + y]^2}{(\varepsilon_2 \wedge 1)^3 n_t^2} \leq 
\sqrt{y + \log n} \frac{|k - k_0|}{n_t^{\frac{5}{4}}} + \kappa \frac{|k - k_0|(y+ \log n)^2}{n_t^2} \leq (1+\kappa) (y + \log n)^2 \frac{|k - k_0|}{n_t^{\frac{5}{4}}}$.
If $\varepsilon_2 = \frac{\sqrt{y + \log n}}{n_t^{\frac{1}{4}}} < 1$, then
\begin{align*}
 \varepsilon_2 \frac{|k - k_0|}{n_t} + \kappa \frac{|k - k_0|[\log n + y]^2}{(\varepsilon_2 \wedge 1)^3 n_t^2} &= \sqrt{y + \log n} \frac{|k - k_0|}{n_t^{\frac{5}{4}}}
+ \kappa (y + \log n)^2 \frac{|k - k_0|}{n_t^2} \frac{n_t^{\frac{3}{4}}}{(y + \log n)^{\frac{3}{2}}} \\
&\leq (1+\kappa)\sqrt{y + \log n} \frac{|k - k_0|}{n_t^{\frac{5}{4}}}.
\end{align*}
In all cases,
\begin{equation} \label{3.eq_eps_2}
 \varepsilon_2 \frac{|k - k_0|}{n_t} + \kappa \frac{|k - k_0|[\log n + y]^2}{(\varepsilon_2 \wedge 1)^3 n_t^2} 
 \leq (1+\kappa) (y + \log n)^2 \frac{|k - k_0|}{n_t^{\frac{5}{4}}}.
\end{equation}
By lemma \ref{3.lem_ineq_inf},
\begin{align*}
 \Bigl| \sum_{j = k_0 + 1}^k (\hat{\theta}_j^T - \theta_j)^2 - \frac{\mathcal{D}_k}{n_t}  \Bigr| &\leq 
 \inf_{\varepsilon \geq 0} \left\{ \varepsilon \frac{\mathcal{D}_k}{n_t} + \kappa \frac{\NormInfinity{s}[\log n + y]}{(\varepsilon \wedge 1)n_t} \right\} \\
&\quad + \inf_{\varepsilon \geq 0} \left\{ \varepsilon \frac{\mathcal{D}_k}{n_t} + \kappa \frac{|k - k_0|[\log n + y]^2}{(\varepsilon \wedge 1)^3 n_t^2} \right\} \\
 &\leq \varepsilon_1 \frac{|k - k_0|}{n_t} + \kappa \frac{\NormInfinity{s}[\log n + y]}{(\varepsilon_1 \wedge 1)n_t} + \varepsilon_2 \frac{|k - k_0|}{n_t} \\
 &\quad + \kappa \frac{|k - k_0|[\log n + y]^2}{(\varepsilon_2 \wedge 1)^3 n_t^2}  + (\varepsilon_1 + \varepsilon_2) \frac{\Norm{\theta}_{\ell^1}}{n_t} \\
 &\leq (1 + \kappa) \NormInfinity{s} (y + \log n) \frac{\sqrt{|k - k_0|}}{n_t} \\
 &\quad + (1+\kappa) (y + \log n)^2 \frac{|k - k_0|}{n_t^{\frac{5}{4}}} 
  + \frac{2\Norm{\theta}_{\ell^1}}{n_t},
\end{align*}
by  equations \eqref{3.eq_eps_1}, \eqref{3.eq_eps_2}.
In conclusion, on an event $E_y$ of probability greater than $1 - e^{-y}$, 
\begin{align*}
 \Bigl| \sum_{j = k_0 + 1}^k (\hat{\theta}_j^T - \theta_j)^2 - \frac{|k - k_0|}{n_t}  \Bigr| &\leq 
 \Bigl| \sum_{j = k_0 + 1}^k (\hat{\theta}_j^T - \theta_j)^2 - \frac{\mathcal{D}_k}{n_t}  \Bigr| + \frac{\Norm{\theta}_{\ell^1}}{n_t} \\
 &\leq \frac{3\Norm{\theta}_{\ell^1}}{n_t} + (1 + \kappa) \NormInfinity{s} (y + \log n) \\ 
 &\quad \times \left[ \frac{\sqrt{|k - k_0|}}{n_t} 
 + (y + \log n) \frac{|k - k_0|}{n_t^{\frac{5}{4}}} \right] . \numberthis \label{3.eq_ub_var_mod}
\end{align*}

If $k_0 = k_* + \alpha_1 \Deltak$ and $k = k_* + \alpha_2 \Deltak$, then by hypothesis
\ref{3.inthm_hyp_ub_nv} of Theorem \ref{3.thm_approx_ho},
\begin{equation}
 \frac{\sqrt{|k - k_0|}}{n_t} = \sqrt{\alpha_2 - \alpha_1} \sqrt{\frac{\Deltak}{n_v n_t}} \sqrt{\frac{n_v}{n_t}} 
= \sqrt{\alpha_2 - \alpha_1} \sqrt{\frac{n - n_t}{n_t}} \Deltal 
\leq \sqrt{\alpha_2 - \alpha_1} n^{- \frac{\delta_4}{2}} \Deltal. \label{3.eq_ub_term1_var}
\end{equation}
Furthermore,
\begin{align*}
 \frac{|k - k_0|}{n_t^{\frac{5}{4}}} &= (\alpha_2 - \alpha_1) \frac{\Deltaor}{n_t^{\frac{1}{4}}} \\
&= (\alpha_2 - \alpha_1) \sqrt{\frac{\Deltaor}{n_v}} \frac{\sqrt{n_v \Deltaor}}{n_t^{\frac{1}{4}}} \\
&\leq (\alpha_2 - \alpha_1) \Deltal \frac{\sqrt{2 n_v \oracle(n_t) + 1}}{n_t^{\frac{1}{4}}}.
\end{align*}
Let $k_1 = \lceil n_t^{\frac{1}{3 + \delta_1}} \rceil$, so that $n_t^{\frac{1}{3 + \delta_1}} \leq k_1 \leq 2n_t^{\frac{1}{3 + \delta_1}}$.
By hypothesis \ref{3.inthm_hyp_ub_sum_varphi} 
of Theorem \ref{3.thm_approx_ho}, $\sum_{j = k+1}^{+ \infty} \theta_j^2 \leq \frac{c_1}{k^{2 + \delta_1}}$ therefore 
\[ \oracle(n_t) \leq \inf_{k \in \mathbb{N}^*} \frac{c_1}{k^{2 + \delta_1}} + \frac{k}{n_t} \\ 
\leq \frac{c_1}{k_1^{2 + \delta_1}} + \frac{k_1}{n_t} \\ 
\leq \frac{c_1}{n_t^{\frac{2}{3 + \delta_1}}} + \frac{2 n_t^{\frac{1}{3 + \delta_1}}}{n_t} \\
\leq \frac{2 + c_1}{n_t^{\frac{2}{3 + \delta_1}}}. \]
Thus $1 + 2n_v \oracle(n_t) \leq 1 + 2n_t \oracle(n_t) \leq (5 + 2c_1) n_t^{\frac{1 + \delta_1}{3 + \delta_1}} $, hence
\begin{align*}
 \frac{|k - k_0|}{n_t^{\frac{5}{4}}} &\leq (\alpha_2 - \alpha_1) \Deltal \sqrt{5 + 2c_1}
 \frac{n_t^{\frac{1+\delta_1}{6 + 2\delta_1}}}{n_t^{\frac{1}{4}}} \\
&\leq (\alpha_2 - \alpha_1) \sqrt{5 + 2c_1} n_t^{- \frac{1}{12}} \Deltal \\
&\leq (\alpha_2 - \alpha_1) \sqrt{5 + 2c_1} \frac{2^{\frac{1}{12}}}{n^{\frac{1}{12}}} \Deltal. \numberthis \label{3.eq_ub_term2_var}
\end{align*}

Finally, $\frac{\Norm{\theta}_{\ell^1}}{n_t} = \frac{n_v}{n_t} \frac{\Norm{\theta}_{\ell^1}}{n_v} 
\leq \Norm{\theta}_{\ell^1} \frac{n - n_t}{n_t} \Deltal \leq \Norm{\theta}_{\ell^1} n^{- \delta_4} $.
Equation \eqref{3.eq_var_odg} follows from equations \eqref{3.eq_ub_var_mod}, \eqref{3.eq_ub_term1_var} and \eqref{3.eq_ub_term2_var}.
\end{proof}
\begin{lemma} \label{3.lem_ub_double_sum}
 Let $(c_{i,j})_{(i,j) \in \mathbb{N}^2}$ be real coefficients.  Let $I_1, I_2 \subset \mathbb{N}$ be two finite sets. 
 Let $(\theta_j)_{j \in \mathbb{N}}$ be a sequence.
 Let $C = \max \left\{ \sup_{i \in I_1} \sum_{j \in I_2} |c_{i,j}|, \sup_{i \in I_2} \sum_{j \in I_1} |c_{i,j}|
  \right\}$. Then
 \[ \sum_{i \in I_1} \left( \sum_{j \in I_2} c_{i,j} \theta_j\right)^2 \leq C^2 \sum_{j \in I_2} \theta_j^2 \]
 and 
 \[ \left| \sum_{i \in I_1} \sum_{j \in I_2}  \theta_i \theta_j c_{i,j} \right| \leq C \max \left\{ \sum_{i \in I_1} \theta_i^2, \sum_{j \in I_2} \theta_j^2 \right\}. \]
\end{lemma}

\begin{proof}
 Let $C_i = \sum_{j \in I_2}|c_{i,j}|$. Then
 \begin{align*}
  \sum_{i \in I_1} \left( \sum_{j \in I_2} c_{i,j} \theta_j\right)^2 &= \sum_{i \in I_1} C_i^2 \left(\frac{1}{C_i} \sum_{j \in I_2} \text{sgn}(c_{i,j})|c_{i,j}| \theta_j\right)^2 \\
  &\leq \sum_{i \in I_1} \frac{C_i^2}{C_i} \sum_{j \in I_2} |c_{i,j}| \theta_j^2 \text{ by l'inégalité of Jensen} \\
  &\leq \left( \max_{i \in I_1} C_i \right) \sum_{j \in I_2} \theta_j^2 \sum_{i \in I_1} |c_{i,j}| \\
  &\leq C^2 \sum_{j \in I_2} \theta_j^2.
 \end{align*}
This proves the first equation. 
Furthermore,
\begin{align*}
 \left| \sum_{i \in I_1} \sum_{j \in I_2}  \theta_i \theta_j c_{i,j} \right| &\leq \sum_{i \in I_1} \sum_{j \in I_2}  \frac{\theta_i^2 + \theta_j^2}{2} |c_{i,j}| \\
&= \frac{1}{2} \sum_{i \in I_1} \theta_i^2 \sum_{j \in I_2} |c_{i,j}| + \frac{1}{2}\sum_{j \in I_2} \theta_j^2 \sum_{i \in I_1} |c_{i,j}| \\
&\leq C\max \left\{ \sum_{i \in I_1} \theta_i^2, \sum_{j \in I_2} \theta_j^2 \right\},
 \end{align*}
 which proves the second equation.
\end{proof}

\begin{lemma} \label{3.lem_lb_inf_fenetre}
 Under the assumptions of Theorem \ref{3.thm_approx_ho}, there exists a constant $\kappa(c_1,c_2) > 0$ such that
 for any $x \geq 0$,
 \[ k_* + a_x \Deltak \geq \frac{\kappa}{(1+x)^{\frac{1}{\delta_2}}} n_t^\frac{2}{3 \delta_2}. \]
\end{lemma}

\begin{proof}
By hypothesis \ref{3.inthm_hyp_lb_sum_varphi} of Theorem \ref{3.thm_approx_ho}, 
 \begin{align*}
  c_2 (k_* + a_x \Deltak)^{-\delta_2} &\leq \sum_{j = k_* + a_x \Deltak + 1}^{+\infty} \theta_j^2 \\
  &\leq \sum_{j = k_* + a_x \Deltak + 1}^{k_*} \left[ \theta_j^2 - \frac{1}{n_t} \right] + |a_x| \Deltaor
  + \sum_{j = k_* + 1}^{+ \infty} \theta_j^2 \\
  &\leq \Deltal f_n(a_x) + |a_x|\Deltaor + \oracle(n_t). \numberthis \label{3.eq_bne_inf_fenetre_g} 
 \end{align*}
 By definition, $f_n(a_x) \leq x$ and by lemma \ref{3.lem_ub_fen_risk}, $|a_x| \leq 2(1+x)$.
 Furthermore, by lemma \ref{3.lem_odgs},
 $\Deltal \leq \Deltaor \leq 2\oracle(n_t) + \frac{1}{n_v}.$
 Since by hypothesis \ref{3.inthm_hyp_lb_nv} of Theorem \ref{3.thm_approx_ho}, 
 $n_v \geq n^{\frac{2}{3} + \delta_5}$, it follows that: $\Deltaor \leq 2 \oracle(n_t) + \frac{1}{n^{\frac{2}{3} + \delta_5}}$.
 Equation \eqref{3.eq_bne_inf_fenetre_g} thus yields
 \[ c_2 (k_* + a_x \Deltak)^{-\delta_2} \leq 6 (1+x) \left[\oracle(n_t) + \frac{1}{n^{\frac{2}{3}}} \right]. \]
On the other hand, by hypothesis \ref{3.inthm_hyp_ub_sum_varphi} of Theorem \ref{3.thm_approx_ho}, 
\[ \oracle(n_t) \leq \min_{k \in \mathbb{N}} \frac{c_1}{k^2} + \frac{k}{2n_t} 
\leq \frac{3 c_1^{\frac{1}{3}}}{n_t^{\frac{2}{3}}}. \]
It follows finally that, for some constant $\kappa(c_1,c_2)$,
\[ k_* + a_x \Deltak \geq \frac{\kappa}{(1+x)^{\frac{1}{\delta_2}}} n_t^\frac{2}{3 \delta_2}. \]
 \end{proof}

\begin{claim} \label{3.lem_ub_varphi_fen}
Let $u_2 = \min \left( \frac{2\delta_3}{3 \delta_2}, \delta_4 \right)$.
Let $x$ be a non-negative real number.
Let $a_x,b_x$ be such that $a_x \leq 0 \leq b_x$ and $\max(f_n(a_x),f_n(b_x)) \leq x$.
Assume also that $a_x \Deltak - 1 \geq \frac{-k_*}{\Deltak}$.
There exists a constant $\kappa_3 \geq 0$ such that 
for all $j \in [a_x \Deltak; b_x \Deltak + 1]$,
\begin{align}
 \left| f_n \left( \tfrac{j}{\Deltak} \right) - f_n \left( \tfrac{j-1}{\Deltak} \right) \right| &\leq \kappa_3 (1+x)^2 n^{- u_2} \\
 \theta_{k_* + j}^2 &\leq \kappa_3 (1+x)^2 n^{- u_2} \Deltal.
\end{align}
\end{claim}

\begin{proof}
 By hypothesis \ref{3.inthm_hyp_lb_rapport_varphi} of Theorem \ref{3.thm_approx_ho}, 
 for all $k \geq 1$, $\theta_{k + k^{\delta_3}}^2 \geq c_3 \theta_{k - k^{\delta_3}}^2$.
 Thus, for all $k \geq 1$ and any $j \in [k - k^{\delta_3}; k + k^{\delta_3}]$,
\begin{align*}
 \max \left( \theta_k^2, \frac{1}{n_t} \right) 
 &\leq \max \left( c_3 \theta_j^2, \frac{1}{n_t} \right) \\
 &\leq \frac{1 + c_3}{n_t} + c_3 \left| \theta_j^2 - \frac{1}{n_t} \right|. \numberthis \label{3.eq_ub_delta_fn}
\end{align*}
Let $k \in [k_* + a_x \Deltak; k_* + b_x \Deltak + 1]$.
Assume without loss of generality (up to a change in the constant $\kappa_9$) that $x \geq 1$.
Thus by  lemma \ref{3.claim_bd_diff_fn}, $\max(-a_x,b_x) \geq 1$.
\begin{itemize}
 \item If $|b_x| \geq 1$, then two cases can be distinguished.
 \begin{itemize}
  \item If $k \leq k_* + \frac{\Deltak}{2}$, then $k + k^{\delta_3} \wedge \frac{\Deltak}{2}
  \leq k_* + \Deltak \leq k_* + b_x \Deltak$, therefore by definition of $a_x, b_x$,
  \[  2x \Deltal \geq \Deltal [f_n(a_x) + f_n(b_x)] = \sum_{j = k_* + a_x \Deltak + 1}^{k_* + b_x \Deltak} \left| \theta_j^2 - \frac{1}{n_t} \right|
 \geq \sum_{j = k+1}^{k + k^{\delta_3} \wedge \frac{\Deltak}{2}} \left| \theta_j^2 - \frac{1}{n_t} \right|  
 . \]
 \item If $k_* + \frac{\Deltak}{2} < k \leq k_* + b_x \Deltak + 1$, then $k - k^{\delta_3} \wedge \frac{\Deltak}{2} \geq k_*$, therefore
 \[  2x \Deltal \geq \sum_{j = k_* + a_x \Deltak + 1}^{k_* + b_x \Deltak} \left| \theta_j^2 - \frac{1}{n_t} \right|
 \geq \sum_{j = k - k^{\delta_3} \wedge \frac{\Deltak}{2}}^{k-1} \left| \theta_j^2 - \frac{1}{n_t} \right|.  \]
 \end{itemize}

 \item If $|a_x| \geq 1$, then we likewise consider two possibilities.
 \begin{itemize}
  \item If $k > k_* - \frac{\Deltak}{2}$, then $k - k^{\delta_3} \wedge \frac{\Deltak}{2}
  > k_* - \Deltak \geq k_* + a_x \Deltak$, therefore by definition of $a_x, b_x$, 
  \[  2x \Deltal \geq \sum_{j = k_* + a_x \Deltak + 1}^{k_* + b_x \Deltak} \left| \theta_j^2 - \frac{1}{n_t} \right|
 \geq \sum_{j = k - k^{\delta_3} \wedge \frac{\Deltak}{2}}^{k-1} \left| \theta_j^2 - \frac{1}{n_t} \right|.  \]
 \item If $k \leq k_* - \frac{\Deltak}{2}$, then $k + k^{\delta_3} \wedge \frac{\Deltak}{2} \leq k_*$, therefore
 \[  2x \Deltal \geq \sum_{j = k_* + a_x \Deltak + 1}^{k_* + b_x \Deltak} \left| \theta_j^2 - \frac{1}{n_t} \right|
 \geq \sum_{j = k + 1}^{k + k^{\delta_3} \wedge \frac{\Deltak}{2}} \left| \theta_j^2 - \frac{1}{n_t} \right|.  \]
 \end{itemize} 
\end{itemize} 
  In all cases,  by equation \eqref{3.eq_ub_delta_fn},
 \[ \left( k^{\delta_3} \wedge \frac{\Deltak}{2} \right) \max \left( \theta_k^2, \frac{1}{n_t} \right) \leq  k^{\delta_3} \wedge \frac{\Deltak}{2}
 \frac{1 + c_3}{n_t} + 2x \Deltal, \]
in other words
\[ \max \left( \theta_k^2, \frac{1}{n_t} \right) \leq
 \frac{1 + c_3}{n_t} + \frac{2x \Deltal}{k^{\delta_3} \wedge \frac{\Deltak}{2}}. \]
Furthermore, by hypothesis \ref{3.inthm_hyp_ub_nv} of Theorem \ref{3.thm_approx_ho},
$\Deltak \geq \frac{n_t}{n - n_t} \geq n^{\delta_4}$, and by lemma
\ref{3.lem_lb_inf_fenetre}, 
\[ k^{\delta_3} \geq \left(k_* + a_x \Deltak \right)^{\delta_3} \geq \frac{\kappa}{(1+x)^{\frac{\delta_3}{\delta_2}}} n_t^\frac{2\delta_3}{3 \delta_2} . \] 
Let $u_2 = \min \left( \frac{2\delta_3}{3 \delta_2}, \delta_4 \right)$
Since $\delta_3 \leq \delta_2$, there exists therefore a constant $\kappa$ such that 
\[ \max \left( \theta_k^2, \frac{1}{n_t} \right) \leq \kappa (1+x)^2 n^{-u_2} \Deltal. \]
In conclusion, for all $j \in \{a_x \Deltak, \ldots, b_x \Deltak + 1 \}$,
\begin{align*}
\theta_{k_* + j}^2 &\leq \max \left( \theta_{k_* + j}^2, \frac{1}{n_t} \right) \leq \kappa (1+x)^2 n^{-u_2} \Deltal \\
 \left| f_n \left( \tfrac{j}{\Deltak} \right) - f_n \left( \tfrac{j-1}{\Deltak} \right) \right| &= \frac{1}{\Deltal}|\theta_{j + k_*}^2 - \frac{1}{n_t}| \\
 &\leq \frac{1}{\Deltal} \max \left( \theta_{k_* + j}^2, \frac{1}{n_t} \right) \\
 &\leq \kappa (1+x)^2 n^{-u_2}.
\end{align*}
This proves claim \ref{3.lem_ub_varphi_fen}.
\end{proof}

\begin{proposition} \label{3.prop_cov_approx}
   Let $P$ be the probability measure with pdf $s$ on $[0;1]$. Let $\theta_j = \langle s, \psi_j \rangle = P(\psi_j)$ 
   and $\theta_j^2 = \theta_j^2$, and assume that they satisfy the hypotheses of Theorem \ref{3.thm_approx_ho}.
  Let $\hat{\theta}_j^T = P^T \psi_j$. Let $I^1_k, I^2_k \subset \{ k_* + a_x \Deltak + 1,\ldots, k_* + b_x \Deltak \}$
  be two intervals.  Then the statistics
  \[U_{I^1_k, I^2_k} = \sum_{i \in I^1_k} \sum_{j \in I^2_k} \hat{\theta}_i^T \hat{\theta}_j^T 
  [P(\psi_i \psi_j) - P \psi_i P \psi_j]  \]
  can be approximated in the following way. There exists two constants $\kappa_4$ and $u_3 > 0$ such that, with probability greater than $1 - e^{-y}$,
  \begin{align*}
   U_{I^1_k, I^2_k} &=  \frac{1}{2} \frac{|I^1_k \cap I^2_k|}{n_t} +
   \left(1 - \frac{1}{\sqrt{2}} \right) \sum_{i \in I^1_k \cap I^2_k} \theta_i^2 +
  \frac{1}{\sqrt{2}} \sum_{i \in I^1_k} \sum_{j \in I^2_k} \theta_i \theta_j\theta_{|i-j|} 
  + \frac{1}{2 n_t} \sum_{i \in I^1_k} \sum_{j \in I^2_k} \theta_{|i-j|}^2  \\
  &\quad \pm \kappa_4 (y + \log n)^2 (1+x) n^{-u_3} \Deltaor.
  \end{align*}
 \end{proposition}

\begin{proof}
First, by  lemma \ref{3.lem_ub_fen_risk},
\begin{equation} \label{3.eq_sum_square_fen}
 \max \left(\sum_{i \in I^1_k} \theta_i^2 , \sum_{j \in I^2_k} \theta_j^2 \right) \leq \sum_{j = k_* + a_x \Deltak + 1}^{k_* + b_x \Deltak} \theta_j^2 \leq 4(1+x) \Deltaor.
\end{equation}
Let $c_{i,j} = \frac{\theta_{i+j}}{\sqrt{2}} + \bigl( \frac{1 - \delta_{i,j}}{\sqrt{2}} + \delta_{i,j} \bigr) \theta_{|i-j|} - \theta_i \theta_j $.
$U_{I^1_k, I^2_k}$ can be expressed as the sum of $6$ terms: $U_{I^1_k, I^2_k} = V_1 + V_2 + V_3 + V_4 + V_5 + V_6 $, where
\begin{align*}
 V_1 &= \sum_{i \in I^1_k} \sum_{j \in I^2_k} \theta_i \theta_j \left[ \frac{\theta_{i+j}}{\sqrt{2}} 
 + \bigl( \frac{1 - \delta_{i,j}}{\sqrt{2}} + \delta_{i,j} \bigr) \theta_{|i-j|} - \theta_i \theta_j \right] \\
 V_2 &= (P^T - P) \sum_{i \in I^1_k} \psi_i \sum_{j \in I^2_k} \theta_j c_{i,j} \\
 V_3 &= (P^T - P) \sum_{j \in I^2_k} \psi_j \sum_{i \in I^1_k} \theta_i c_{i,j} \\
 V_4 &= \frac{1}{\sqrt{2}} \sum_{i \in I^1_k} \sum_{j \in I^2_k} (P^T - P) \psi_i (P^T - P) \psi_j \theta_{|i-j|} \\
 V_5 &= \left(1 - \frac{1}{\sqrt{2}} \right) \sum_{j \in I^1_k \cap I^2_k} \left( \Ethet{j}{T} - \theta_j \right)^2 \\
 V_6 &= \sum_{i \in I^1_k} \sum_{j \in I^2_k} (P^T - P) \psi_i (P^T - P) \psi_j \left[ \frac{\theta_{i+j}}{\sqrt{2}} - \theta_i \theta_j \right]
\end{align*}
The first term is 
\[ V_1 = \left(1 - \frac{1}{\sqrt{2}} \right) \sum_{i \in I^1_k \cap I^2_k} \theta_i^2 + 
\frac{1}{\sqrt{2}} \sum_{i \in I^1_k} \sum_{j \in I^2_k} \theta_i \theta_j \theta_{|i-j|} 
+ \sum_{i \in I^1_k} \sum_{j \in I^2_k} \theta_i \theta_j \left[\frac{\theta_{i+j}}{\sqrt{2}} - \theta_i \theta_j \right]. \]
For all $i \in I^1_k$,
\[ \sum_{j \in I^2_k} \frac{|\theta_{i+j}|}{\sqrt{2}}  + |\theta_i| |\theta_j| \leq 2\sum_{j \geq k_* + a_x \Deltak + 1} |\theta_j|. \]
Furthermore, for all $k \geq 2$, by hypothesis \ref{3.inthm_hyp_ub_sum_varphi} of Theorem \ref{3.thm_approx_ho},
\begin{equation} \label{3.eq_rem_sum_abs_thet}
 \sum_{j \geq k} |\theta_j| \leq \sum_{j = k}^{+\infty} \sqrt{\sum_{i = j}^{+\infty} \theta_i^2} \leq 
\sum_{j = k}^{+ \infty} \frac{c_1}{(j-1)^{1 + \frac{\delta_1}{2}}} \leq \frac{2c_1}{\delta_1} (k-1)^{\frac{- \delta_1}{2}}.
\end{equation}
Since $k_* + a_x \Deltak \geq \frac{\kappa}{(1+x)^{\frac{1}{\delta_2}}} n_t^\frac{2}{3 \delta_2}$ by lemma \ref{3.lem_lb_inf_fenetre},
there is a constant $\kappa(c_1,c_2)$ such that
\begin{equation} \label{3.eq_sum_reste_cov}
 \sum_{j \in I^2_k} \frac{|\theta_{i+j}|}{\sqrt{2}}  + |\theta_i| |\theta_j|  
\leq \kappa \frac{(1+x)^{ \frac{\delta_1}{2\delta_2}}}{n_t^{\frac{\delta_1}{3\delta_2}}}.
\end{equation}
The same argument applies to $\sum_{i \in I^1_k} \frac{|\theta_{i+j}|}{\sqrt{2}}  + |\theta_i| |\theta_j|$.
Thus, by lemma \ref{3.lem_ub_double_sum}, 
\[\sum_{i \in I^1_k} \sum_{j \in I^2_k} \theta_i \theta_j \left[\theta_{i+j} - \theta_i \theta_j \right] 
\leq 2\kappa \frac{(1+x)^{ \frac{\delta_1}{2\delta_2}}}{n_t^{\frac{\delta_1}{3\delta_2}}} 
\left[ \sum_{i \in I^1_k} \theta_i^2 + \sum_{j \in I^2_k} \theta_j^2 \right].  \]
By equation \eqref{3.eq_sum_square_fen}, it follows that for a certain constant $\kappa(c_1, c_2)$,
\[ \sum_{i \in I^1_k} \sum_{j \in I^2_k} \theta_i \theta_j \left[\frac{\theta_{i+j}}{\sqrt{2}}  - \theta_i \theta_j \right]  \leq 
\kappa \frac{(1+x)^{1 + \frac{\delta_1}{2\delta_2}}}{n_t^{\frac{\delta_1}{3\delta_2}}} \Deltaor. \]
Thus
\begin{equation} \label{3.eq_V1}
 V_1 = \left(1 - \frac{1}{\sqrt{2}} \right) \sum_{i \in I^1_k \cap I^2_k} \theta_i^2 + 
 \frac{1}{\sqrt{2}} \sum_{i \in I^1_k} \sum_{j \in I^2_k} \theta_i \theta_j \theta_{|i-j|} 
 \pm \kappa \frac{(1+x)^{1 + \frac{\delta_1}{2\delta_2}}}{n_t^{\frac{\delta_1}{3\delta_2}}} \Deltaor
\end{equation}
Bernstein's inequality applies to $V_2$ and $V_3$. By symmetry, let us only consider $V_2$. 
Its variance satisfies the following inequality.
\begin{align*}
 \Var \left( \sum_{i \in I^1_k} \psi_i \sum_{j \in I^2_k} \theta_j c_{i,j}\right) &\leq \NormInfinity{s} \Norm{\sum_{i \in I^1_k} \psi_i \sum_{j \in I^2_k} \theta_j c_{i,j}}^2 \\
 &\leq \NormInfinity{s} \sum_{i \in I^1_k} \left( \sum_{j \in I^2_k} \theta_j c_{i,j} \right)^2. 
\end{align*}
Let us now apply lemma \ref{3.lem_ub_double_sum}.
For all $i \in I^1_k$,
\begin{align*}
 \sum_{j \in I^2_k} |c_{i,j}| &\leq \frac{1}{\sqrt{2}}\sum_{j \in I^2_k} |\theta_{i+j}| + \frac{1}{\sqrt{2}}\sum_{j \in I^2_k} |\theta_{|i-j|}| + |\theta_i| \sum_{j \in I^2_k} |\theta_j| \\
 &\leq \left( \sqrt{2} + \sup_{i \in \mathbb{N}} |\theta_i| \right) \sum_{r \in \mathbb{N}} |\theta_r| \\
 &\leq 3 \Norm{\theta}_{\ell^1}
\end{align*}
In the same way, for all $j \in I^2_k$, $\sum_{i \in I^1_k} |c_{i,j}| \leq 3 \Norm{\theta}_{\ell^1}$,
hence by lemma \ref{3.lem_ub_double_sum},
\begin{align*}
 \Var \left( \sum_{i \in I^1_k} \psi_i \sum_{j \in I^2_k} \theta_j c_{i,j}\right) &\leq  3 \Norm{\theta}_{\ell^1} \NormInfinity{s} \sum_{j \in I^2_k} \theta_j^2  \\
 &\leq  12 \Norm{\theta}_{\ell^1} \NormInfinity{s} (1+x) \Deltaor \text{ by equation } \eqref{3.eq_sum_square_fen}. \numberthis \label{3.eq_bern_varbd}
\end{align*}
As for the upper bound on the uniform norm, it follows from lemma \ref{3.lem_ub_double_sum} and the elementary upper bound
$\NormInfinity{\psi_i} \leq \sqrt{2}$ that
\begin{align*}
 \sup_{x \in \mathbb{R}} \left| \sum_{i \in I^1_k} \psi_i(x) \sum_{j \in I^2_k} \theta_j c_{i,j} \right| 
 &\leq \sqrt{\sum_{i \in I^1_k} \left( \sum_{j \in I^2_k} \theta_j c_{i,j} \right)^2} \sup_{x \in \mathbb{R}} \sqrt{\sum_{i \in I^1_k} \psi_i(x)^2} \\
 &\leq 3 \Norm{\theta}_{\ell^1} \sqrt{2|I^1_k|} \sqrt{\sum_{j \in I^2_k} \theta_j^2} \\
 &\leq 3 \Norm{\theta}_{\ell^1} \sqrt{2(b_x - a_x) \Deltak} \sqrt{4(1+x) \Deltaor} \\
 &\leq \kappa (1+x) \sqrt{\Deltak \Deltaor} \text{ by lemma } \ref{3.lem_ub_fen_risk} \numberthis \label{3.eq_bern_supbd}, 
\end{align*}
for some constant $\kappa = \kappa(\Norm{\theta}_{\ell^1})$.
By Bernstein's inequality, there exists an event $E_2(y) \subset \mathbb{R}^{n_t}$
with probability $\mathbb{P}(D_n^T \in E_2(y)) \geq 1 - e^{-y}$ such that, 
for any $D_n^T \in E_2(y)$,
\begin{align*}
 |V_2| &\leq \sqrt{\frac{2y}{n_t}} \sqrt{\Var \left( \sum_{i \in I^1_k} \psi_i \sum_{j \in I^2_k} \theta_j c_{i,j}\right)} 
 + \frac{y}{3n_t} \sup_{x \in \mathbb{R}} \left| \sum_{i \in I^1_k} \psi_i(x) \sum_{j \in I^2_k} \theta_j c_{i,j} \right| \\
 &\leq \sqrt{24 \Norm{\theta}_{\ell^1} \NormInfinity{s}} \sqrt{\frac{y(1+x) \Deltaor}{n_t}} + \frac{\kappa y}{3n_t}(1+x) \sqrt{\Deltak \Deltaor} \text{ by } 
 \eqref{3.eq_bern_varbd}, \eqref{3.eq_bern_supbd}.
\end{align*}
Setting $\kappa = \max \left( \sqrt{24 \Norm{\theta}_{\ell^1} \NormInfinity{s}}, \frac{\kappa}{3} \right)$,
it follows that 
on $E_2(y)$,
\[  |V_2| \leq \kappa \sqrt{y(1+x)} \sqrt{\frac{n_v}{n_t}} \Deltal + \kappa y(1+x) \sqrt{\frac{n_v}{n_t}} \Deltal \sqrt{\Deltaor}. \]
By lemma \ref{3.lem_odgs}, $\Deltaor$ is uniformly bounded: $\Deltaor \leq 2 \sum_{j = 1}^{n_t} \theta_j^2 + \frac{1}{n - n_t} \leq 1 + 2 \Norm{s}^2 \leq 1 + 2\NormInfinity{s}$.
Furthermore, by hypothesis \ref{3.inthm_hyp_ub_nv} of Theorem \ref{3.thm_approx_ho}, 
$\sqrt{\frac{n_v}{n_t}} = \sqrt{\frac{n - n_t}{n_t}} \leq n^{-\frac{\delta_4}{2}}$.
Thus, there exists a constant $\kappa(\Norm{\theta}_{\ell^1}, \NormInfinity{s})$ such that, on $E_2(y)$,
\begin{equation} \label{3.eq_V2}
 |V_2| \leq \kappa y(1+x) n^{-\frac{\delta_4}{2}} \Deltal.
\end{equation}
Symmetrically, there exists an event $E_3(y)$ of probability greater than $1 - e^{-y}$, such that for any $D_n^T \in E_3(y)$,
\begin{equation} \label{3.eq_V3}
 |V_3| \leq \kappa y(1+x) n^{-\frac{\delta_4}{2}} \Deltal.
\end{equation}

Now consider $V_4$. This term can be expressed as a finite sum of sums of squares:

\begin{align*}
 V_4 &= \frac{1}{\sqrt{2}} \sum_{r \in \mathbb{Z}} \sum_{i \in I^1_k \cap (I^2_k-r)} (P^T - P) \psi_i (P^T - P) \psi_{i+r} \theta_{|r|} \\
 &= \frac{1}{4 \sqrt{2}} \sum_{r \in \mathbb{Z}} \theta_{|r|} \sum_{i \in I^1_k \cap (I^2_k-r)} 
 \left[ (P^T - P)(\psi_i + \psi_{i+r}) \right]^2 - \left[(P^T - P)(\psi_i - \psi_{i+r}) \right]^2 . 
\end{align*}
Let $J_0 = \{ j \in \mathbb{N}: \lfloor \frac{j}{r} \rfloor \text{ is even} \}$ and 
$J_1 = \{ j \in \mathbb{N}: \lfloor \frac{j}{r} \rfloor \text{ is odd} \}$. Thus
\[ V_4 = \frac{1}{4 \sqrt{2}} \sum_{r \in \mathbb{Z}} \theta_{|r|} \sum_{(z,\varepsilon) \in \{0;1\} \times \{-1;1\}} \sum_{j \in J_z} \varepsilon(P^T - P)(\psi_i + \varepsilon \psi_{i+r})^2 \mathbb{I}_{I^1_k}(i) \mathbb{I}_{I^2_k}(i+r) .\]

For any fixed $r \neq 0$, $(z,\varepsilon) \in \{0;1\} \times \{-1;1\}$, $\frac{1}{\sqrt{2}} \bigl(\psi_i + \varepsilon \psi_{i+r} \bigr)_{i \in J_z}$ 
is an orthonormal collection of functions, since for any $(i,j) \in J_z^2$,
\begin{align*}
 <\psi_i + \varepsilon \psi_{i+r},\psi_j + \varepsilon \psi_{j+r}> &= <\psi_i,\psi_j> + 
 \varepsilon <\psi_{i+r},\psi_j> +\varepsilon <\psi_{i},\psi_{j+r}> + <\psi_{i+r},\psi_{j+r}> \\
 &= 2\delta_{i,j} + \varepsilon <\psi_{i+r},\psi_j> +\varepsilon <\psi_{i},\psi_{j+r}> \\
 &= 2\delta_{i,j} \text{ since } i,j \in J_z \text{ and } i+r, j+r \in J_{1-z}. 
\end{align*}
\cite[Lemma 14]{Arl_Ler:2012:penVF:JMLR} applied to $S_m = \langle (\psi_i + \varepsilon \psi_{i+r} )_{i \in J_z \cap I^1_k \cap I^2_k} \rangle$ 
for all $(z,\varepsilon) \in \{0;1\} \times \{-1;1\}$ , $r \in \{-n_t,\ldots,n_t\}$ and a union bound yield
an event $E_4(y)$ of probability $\mathbb{P}(D_n^T \in E_4(y)) \geq 1 - e^{-y}$ such that, for some absolute constant $\kappa$ and for all $D_n^T \in E_4(y)$, 
$(z,\varepsilon) \in \{0;1\} \times \{-1;1\}$ and $r \in \mathbb{Z}$,
\begin{align*}
 \sum_{i \in J_z \cap I^1_k \cap (I^2_k-r)} \varepsilon(P^T - P)(\psi_i + \varepsilon \psi_{i+r})^2  &= (1 \pm \delta) \frac{\varepsilon}{n_t} 
 \sum_{i \in J_z \cap I^1_k \cap (I^2_k-r)} \bigl[ \Var(\psi_i) + \Var(\psi_{i+r}) \\ 
 & + 2 \varepsilon \Cov(\psi_i,\psi_{i+r}) \bigr]  
 + \kappa \frac{\NormInfinity{s} [\log(1+r) + \log n_t + y]}{(\delta \wedge 1) n_t} \\ 
 &\quad + \kappa \frac{|I^1_k|[\log (1+r) + \log n_t + y]^2}{(\delta \wedge 1)^3 n_t^2}.
\end{align*}
By summing on $(r, z,\varepsilon) \in \mathbb{Z} \times \{0;1 \} \times \{-1;1\}$ and since $\NormInfinity{\psi_i} \leq \sqrt{2}$, 
it follows that for all $D_n^T \in E_4(y)$,
\begin{align*}
 \left|V_4 - \frac{1 }{n_t} \sum_{r \in \mathbb{Z}} \frac{\theta_{|r|}}{\sqrt{2}} \sum_{i \in I^1_k \cap (I^2_k-r)} c_{i,i+r} \right| 
 &= \left|V_4 - \frac{1 }{n_t} \sum_{r = - n_t}^{n_t} \frac{\theta_{|r|}}{\sqrt{2}} \sum_{i \in I^1_k \cap (I^2_k-r)} c_{i,i+r} \right| \\
 &\leq \frac{\delta}{n_t \sqrt{2}}  \sum_{r \in \mathbb{Z}} |\theta_{|r|}| \sum_{i \in I^1_k \cap (I^2_k-r)} \left[ \Var(\psi_i) + \Var(\psi_{i+r}) \right] \\ 
 &\quad + \kappa \sum_{r \in \mathbb{Z}}\frac{|\theta_{|r|}|}{\sqrt{2}}\times \frac{\NormInfinity{s} [\log n_t \log (1+r) + y]}{(\delta \wedge 1) n_t} \\
&\quad + \kappa \sum_{r \in \mathbb{Z}}\frac{|\theta_{|r|}|}{\sqrt{2}}\times \frac{|I^1_k|[\log n_t + \log(1+r) + y]^2}{(\delta \wedge 1)^3 n_t^2} \\
\end{align*}
By hypothesis \ref{3.inthm_hyp_ub_sum_varphi}, $|\theta_j| \leq \frac{\sqrt{c_1}}{j^{1 + \frac{\delta_1}{2}}}$,
hence the sum $\sum_{r \in \mathbb{Z}} |\theta_{|r|}| \log(1+r)^2$ converges to a finite value $\Norm{\theta}_{1, \log^2}$. 
Moreover, by lemma \ref{3.lem_ub_fen_risk}, $|I^1_k| \leq (b_x - a_x) \Deltak \leq 2(1+x) \Deltak$, hence
\begin{align*}
 \left|V_4 - \frac{1 }{n_t} \sum_{r \in \mathbb{Z}} \frac{\theta_{|r|}}{\sqrt{2}} \sum_{i \in I^1_k \cap (I^2_k-r)} c_{i,i+r} \right| &\leq 
6 \Norm{\theta}_{\ell^1}(1+x) \frac{\delta}{n_t}  \Deltak
 + 2\kappa \Norm{\theta}_{1,\log^2}  \frac{\NormInfinity{s} [1 + y]}{(\delta \wedge 1) n_t} \\
 &\quad + 8\kappa (1+x) \Norm{\theta}_{1,\log^2} \frac{\Deltak[1 + y]^2}{(\delta \wedge 1)^3 n_t^2}.
\end{align*}
 There exists therefore a constant $\kappa(\Norm{\theta}_{1,\log^2})$ such that, for all $D_n^T \in E_4(y)$,
 \[ \left|V_4 - \frac{1 }{n_t} \sum_{r \in \mathbb{Z}} \frac{\theta_{|r|}}{\sqrt{2}} \sum_{i \in I^1_k \cap (I^2_k-r)} c_{i,i+r} \right| 
 \leq \kappa \delta (1+x) \Deltaor + \frac{[\log n_t + y]}{(\delta \wedge 1) n_t} 
 + \kappa(1+x)\frac{[\log n_t + y]^2}{(\delta \wedge 1)^3 n_t} \Deltaor. \]
Let now $\delta = \max \left\{ \frac{n - n_t}{n_t}, n^{- \frac{1}{3}} \right\}^{\frac{3}{4}} $.
By hypothesis \ref{3.inthm_hyp_ub_nv} of section \ref{3.sec.hyp},  $\frac{n - n_t}{n_t} \leq n^{-\delta_4}$, 
therefore $\delta \Deltaor \leq n^{-\min\left( \frac{1}{4}, \frac{3\delta_4}{4} \right)} \Deltaor$.
Moreover, $\Deltaor \geq \frac{1}{n_v}$ therefore $\frac{1}{\delta n_t} \leq \left( \frac{n - n_t}{n_t} \right)^{\frac{1}{4}} \frac{1}{n_v} 
\leq n^{-\frac{\delta_4}{4}} \Deltaor$.
Finally, since $\delta \geq n^{-\frac{1}{4}}$ and $n_t \geq \frac{n}{2}$, $\frac{\Deltaor}{\delta^3 n_t} \leq 2 n^{-\frac{1}{4}} \Deltaor$.
Since $\delta_4 \leq 1$, there exists therefore a constant $\kappa$ such that for all $D_n^T \in E_4(y)$,
\begin{equation} \label{3.eq_presque_V4}
 \left|V_4 -\frac{1}{n_t \sqrt{2}} \sum_{i \in I^1_k} \sum_{j \in I^2_k} \theta_{|i-j|} c_{i,j} \right| 
 \leq \kappa (1+x) [\log n_t + y]^2 n^{- \frac{\delta_4}{4}} \Deltaor.
\end{equation}
Moreover, since $c_{i,j} = \frac{\theta_{i+j}}{\sqrt{2}} + \left( \frac{1 - \delta_{i,j}}{\sqrt{2}} + \delta_{i,j} \right) 
\theta_{|i - j|} - \theta_i \theta_j$ and $\theta_0 = 1$,
\begin{align*}
 \frac{1}{n_t} \sum_{i \in I^1_k} \sum_{j \in I^2_k} \frac{\theta_{|i-j|}}{\sqrt{2}} c_{i,j}
 &= \sum_{i \in I^1_k \cap I^2_k} \frac{1}{\sqrt{2}} \left(1 - \frac{1}{\sqrt{2}} \right) \frac{1}{n_t} 
 +\frac{1}{n_t} \sum_{i \in I^1_k} \sum_{j \in I^2_k} \frac{\theta_{|i-j|}^2}{2} \\
&\quad + \frac{1}{n_t} \sum_{i \in I^1_k} \sum_{j \in I^2_k} \frac{\theta_{|i-j|} \theta_{i+j}}{2} 
 - \frac{1}{n_t} \sum_{i \in I^1_k} \sum_{j \in I^2_k} \frac{\theta_{|i-j|}}{\sqrt{2}} \theta_i \theta_j.
\end{align*}
Since for all $j \in \mathbb{N}$, $|\theta_j| \leq 1$,
\begin{align*}
   \left| \frac{1}{n_t} \sum_{i \in I^1_k} \sum_{j \in I^2_k} \frac{\theta_{|i-j|}}{\sqrt{2}} c_{i,j}
 - \left( 1 - \frac{1}{\sqrt{2}} \right) \frac{|I^1_k \cap I^2_k|}{n_t \sqrt{2}} - \frac{1}{n_t} \sum_{i \in I^1_k} \sum_{j \in I^2_k} \frac{\theta_{|i-j|}^2}{2}  \right| 
 &\leq \frac{2}{n_t} \left(\sum_{r \in \mathbb{N}} |\theta_r| \right)^2 \\
 &\leq 2 \frac{n - n_t}{n_t} \frac{1}{n_v} \Norm{\theta}_{\ell^1}^2 \\
 &\leq 2 \Norm{\theta}_{\ell^1}^2 n^{-\delta_4} \Deltaor \numberthis \label{3.eq_rem_V4}
\end{align*}
since $\Deltaor \geq \frac{1}{n_v}$ and $\frac{n - n_t}{n_t} \geq n^{-\delta_4}$, by hypothesis \ref{3.inthm_hyp_ub_nv}
of Theorem \ref{3.thm_approx_ho}.
From equations \eqref{3.eq_presque_V4} and \eqref{3.eq_rem_V4}, it follows that, for some constant
$\kappa(\Norm{\theta}_{1,\log^2})$,
\begin{equation} \label{3.eq_V4}
 \left|V_4 - \left( 1 - \tfrac{1}{\sqrt{2}} \right) \frac{|I^1_k \cap I^2_k|}{n_t \sqrt{2}} 
 - \frac{1}{2n_t} \sum_{i \in I^1_k} \sum_{j \in I^2_k} \theta_{|i-j|}^2  \right| 
 \leq \kappa (1+x) [\log n_t + y]^2 n^{-\frac{\delta_4}{4}} \Deltaor. 
\end{equation}

$V_5$ can be expressed as
\[ V_5 = \left(1 - \frac{1}{\sqrt{2}} \right) \sum_{j \in I^1_k \cap I^2_k} \left( \Ethet{j}{T} - \theta_j \right)^2, \]
therefore by  proposition \ref{3.prop_approx_rsk}, there exists an event $E_5(y)$ of probability greater than $1 - e^{-y}$ 
such that for all $D_n^T \in E_5(y)$,
\[ V_5 = \left(1 - \frac{1}{\sqrt{2}} \right) \sum_{i \in I^1_k \cap I^2_k} \frac{1}{n_t} 
\pm \left(1 - \frac{1}{\sqrt{2}} \right) \kappa_1 (b_x - a_x) [\log n + y]^2 n^{- \min(\frac{1}{12}, \frac{\delta_4}{2})} \Deltal (n) \]
It follows by lemma \ref{3.lem_ub_fen_risk} that  on $E_5(y)$,
\begin{equation} \label{3.eq_V5}
 \left| V_5 - \left(1 - \frac{1}{\sqrt{2}} \right) \frac{|I^1_k \cap I^2_k|}{n_t} \right| \leq 
 4\kappa_1 (1+x) [\log n + y]^2 n^{- \min(\frac{1}{12}, \frac{\delta_4}{2})} \Deltal (n). 
\end{equation}

Finally, $V_6$ can be bounde in the following manner.
\begin{align*}
 V_6 &\leq \frac{1}{2} \sum_{i \in I^1_k} \sum_{j \in I^2_k} \left[(P^T - P)^2 \psi_i + (P^T - P)^2 \psi_j \right] 
 \left[ \frac{|\theta_{i+j}|}{\sqrt{2}} + |\theta_i| |\theta_j| \right] \\
 &= \frac{1}{2} \sum_{i \in I^1_k} (P^T - P)^2 \psi_i \sum_{j \in I^2_k} \left[ \frac{|\theta_{i+j}|}{\sqrt{2}} 
 + |\theta_i| |\theta_j| \right] + \frac{1}{2} \sum_{j \in I^2_k} (P^T - P)^2 \psi_j \sum_{i \in I^1_k} \left[ \frac{|\theta_{i+j}|}{\sqrt{2}} 
 + |\theta_i| |\theta_j| \right]. \\
\end{align*}
Thus, by equation \eqref{3.eq_sum_reste_cov},
\begin{align*}
 V_6 &\leq \kappa \frac{(1+x)^{ \frac{\delta_1}{2\delta_2}}}{n_t^{\frac{\delta_1}{3\delta_2}}} \times 
 \left[ \frac{1}{2} \sum_{i \in I^1_k} (P^T - P)^2 \psi_i + \frac{1}{2} \sum_{j \in I^2_k} (P^T - P)^2 \psi_j \right] \\
 &\leq \kappa \frac{(1+x)^{ \frac{\delta_1}{2\delta_2}}}{n_t^{\frac{\delta_1}{3\delta_2}}} 
 \times \sum_{j = k_* + a_x \Deltak + 1}^{k_* + b_x \Deltak} \left( \Ethet{j}{T} - \theta_j \right)^2. \numberthis \label{3.eq_maj_V5}
\end{align*}
By proposition \ref{3.prop_approx_rsk}, there exists an event $E_6(y)$ of probability
greater than $1 - e^{-y}$, such that for any $D_n^T \in E_6(y)$,
\begin{align*}
 \sum_{j = k_* + a_x \Deltak + 1}^{k_* + b_x \Deltak} \left( \Ethet{j}{T} - \theta_j \right)^2 &\leq 
 (b_x - a_x) \Deltaor + \kappa_1 (b_x - a_x) (y + \log n)^2 n^{- \min(\frac{1}{12}, \frac{\delta_4}{2})} \Deltal \\
 &\leq 2 \max(\kappa_1,1) (b_x - a_x) (y + \log n)^2 \Deltaor \\
 &\leq 8 \max(\kappa_1,1) (1+x) (y + \log n)^2 \Deltaor \text{ by  lemma } \ref{3.lem_ub_fen_risk}.
\end{align*}
It follows by equation \eqref{3.eq_maj_V5} that on $E_6(y)$, for a certain constant $\kappa(\kappa_1,\delta_1, c_1, \kappa_6)$,
\begin{equation} \label{3.eq_V6}
 V_6 \leq \kappa [y + \log n]^2 \frac{(1+x)^{ \frac{\delta_1}{2\delta_2}}}{n_t^{\frac{\delta_1}{3\delta_2}}} \Deltaor.
\end{equation}
Combining equations \eqref{3.eq_V1}, \eqref{3.eq_V2}, \eqref{3.eq_V3}, \eqref{3.eq_V4}, \eqref{3.eq_V5}, \eqref{3.eq_V6}
on the event $\cap_{i = 2}^6 E_i(\log 6 + y)$ yields the result.
\end{proof}

\begin{lemma} \label{3.lem_theta_posdef}
Let $s \in L^{\infty}([0;1])$ be a probability density function.
 For all $j \in \mathbb{N}$, let $\theta_j = \langle s, \psi_i \rangle$, where $\psi_0(x) = 1$ and
 $\psi_j(x) = \sqrt{2} \cos(2j \pi x)$ for all $j \in \mathbb{N}^*$.
 Thus for any finite set $I \subset \mathbb{N}$ and for all functions $u \in \mathbb{R}^{I}$, 
 \[0 \leq \sum_{i \in I} \sum_{j \in I} u(i) u(j) \left( \frac{1 - \delta_{i,j}}{\sqrt{2}} + \delta_{i,j} \right) 
 \theta_{|i-j|} \leq \NormInfinity{s} \sum_{i \in I} u(i)^2. \]
\end{lemma}

\begin{proof}
 Let $X \sim s$ be a random variable with distribution $s(x) dx$ on $[0; 1]$.
For any $x \in \mathbb{R}$, and any $i \neq j$,
\[ \psi_i(x) \psi_j(x) = 2 \cos(2i \pi x) \cos(2j \pi x) = \cos(2(i+j)\pi x) + \cos(2(i-j) \pi x) 
= \frac{\psi_{i+j} + \psi_{|i-j|}}{\sqrt{2}}. \] 
If $i \neq j$, then
 $\Cov(\psi_i(X), \psi_j(X)) = \frac{\theta_{i+j} + \theta_{|i - j|}}{\sqrt{2}} - \theta_i \theta_j$.
 If $i = j$, $\Var (\psi_i(X)) = 1 + \frac{\theta_{2i}}{\sqrt{2}} - \theta_i^2$.
 Let $u \in \mathbb{R}^{I}$, $k \in \mathbb{N}$ and $t_k = \sum_{i \in I} u(i) \psi_{i + k}$, then
 \begin{align*}
  \Var (t_k(X)) &= \sum_{i \in I} \sum_{j \in I} u(i) u(j) 
  \left[ \left( \frac{1 - \delta_{i,j}}{\sqrt{2}} + \delta_{i,j} \right) \theta_{|i-j|} + \frac{\theta_{i + j + 2k}}{\sqrt{2}} - \theta_{i+k} \theta_{j + k} \right] \\
 \end{align*}
Furthermore, $\lim_{n \to +\infty} \theta_{n} = 0$, hence
\[ \lim_{k \to +\infty} \Var (t_k(X)) = \sum_{i \in I} \sum_{j \in I} u(i) u(j) \left( \frac{1 - \delta_{i,j}}{\sqrt{2}} + \delta_{i,j} \right)  \theta_{|i-j|}. \]
It immediately follows that $\sum_{i \in I} \sum_{j \in I} u(i) u(j) \left( \frac{1 - \delta_{i,j}}{\sqrt{2}} + \delta_{i,j} \right)  \theta_{|i-j|}  \geq 0$.
Moreover, for all $k \in \mathbb{N}$,
\begin{align*}
 \Var (t_k(X)) &\leq \mathbb{E} \left[ t_k(X)^2 \right] \\ 
 &= \int_{0}^{1} t_k(x)^2 s(x) dx \\
 &\leq \NormInfinity{s} \Norm{t_k}^2 \\
 &\leq \NormInfinity{s} \sum_{i \in I} u(i)^2.
\end{align*}
Thus 
\[ \sum_{i \in I} \sum_{j \in I} u(i) u(j) \left( \frac{1 - \delta_{i,j}}{\sqrt{2}} + \delta_{i,j} \right)  \theta_{|i-j|} 
\leq \NormInfinity{s} \sum_{i \in I} u(i)^2. \]
\end{proof}

\begin{lemma} \label{3.lem_approx_non-decreasing}
Let $\varepsilon: \mathbb{R}_+ \rightarrow \mathbb{R}_+$ be a non-decreasing function such that $\varepsilon(0) > 0$
and $h_+: \mathbb{R}_+ \rightarrow \mathbb{R}_+$ be a continuous, non-decreasing function. 
 Let $g_0: \mathbb{R}_+ \rightarrow \mathbb{R}$ be a continuous function such that, for any $s < t$,
 \[ - \varepsilon(\max(s,t)) \leq g_0(t) - g_0(s) \leq \max \{h_+(t) - h_+(s), \varepsilon(\max(s,t)) \}. \]
 Assume that $\varepsilon(0) > 0$.
 Then there exists a continuous, non-decreasing function $g: \mathbb{R}_+ \rightarrow \mathbb{R}_+$ such that $g_0(0) = g(0)$,
 \[ \NormInfinity{\frac{g_0 - g}{\varepsilon}} \leq 6, \]
 and moreover
 \[ \forall x,y, |g(y) - g(x)| \leq |h_+(y) - h_+(x)|.  \]
\end{lemma}

\begin{proof}
Assume to begin with that $\varepsilon$ is right-continuous.
 Let $r > 0, \delta > 0$. We define by induction a sequence $(x_i)_{i  \in \mathbb{N}}$ and a function $g$ on $[x_i;x_{i+1}]$. 
 Let $x_0 = 0$ and $g(x_0) = g_0(x_0)$. 
 For any $i \in \mathbb{N}$, assuming $x_i$ and $g(x_i)$ have been defined, let
 \begin{equation}
  \begin{split}
   x_{i+1} &=  \inf \left\{x \geq x_i: g_0(x) \geq g_0(x_i) + 2\varepsilon(x_i) \text{ ou } \varepsilon(x) \geq \frac{3}{2} \varepsilon(x_i) \right\} \\
   \forall x \in (x_i,x_{i+1}], g(x) &= 
   \begin{cases}
    &g(x_i) \text{ if } \varepsilon(x_{i+1}) \geq \frac{3}{2} \varepsilon(x_i) \\                    
    &g(x_i) + \frac{g_0(x_{i+1}) - g_0(x_i)}{h_+(x_{i+1}) - h_+(x_i)} [h_+(x) - h_+(x_i)] \text{ else}. 
   \end{cases}
  \end{split}
 \end{equation}
  If $x_{i+1} = +\infty$, the above definitions still make sense and the induction stops.
 Notice first that for any $x \in [x_i;x_{i+1})$, 
 $g_0(x) - g_0(x_i) \leq [h_+(x) - h_+(x_i)] \vee \varepsilon(x) \leq [h_+(x) - h_+(x_i)] \vee \frac{3}{2} \varepsilon(x_i)$.
 Thus,by continuity of $g_0$,
 \[ g_0(x_{i+1}) - g_0(x_i) \leq [h_+(x_{i+1}) - h_+(x_i)] \vee \frac{3}{2} \varepsilon(x_i). \]
 By assumption, $\varepsilon$ is right-continuous, therefore 
 if $\varepsilon(x_{i+1}) < \frac{3}{2} \varepsilon(x_i)$, it must be that $\inf \{x \geq x_i : \varepsilon(x) \geq \frac{3}{2} \varepsilon(x_i) \} > x_{i+1}$.
 Then by definition of $x_{i+1}$ and continuity of $g_0$, 
 $g_0(x_{i+1}) = g_0(x_i) + 2\varepsilon(x_i)$, therefore 
 \[ 2\varepsilon(x_i) = g_0(x_{i+1}) - g_0(x_i) \leq [h_+(x_{i+1}) - h_+(x_i)] \vee \frac{3}{2} \varepsilon(x_i), \]
 which implies that
 \begin{equation}
  0 < 2\varepsilon(x_i) = g_0(x_{i+1}) - g_0(x_i) \leq [h_+(x_{i+1}) - h_+(x_i)] .
 \end{equation}
This proves that $g$ is well defined. $g$ is non-decreasing and continuous
since $h_+$ has these properties. If $\varepsilon(x_{i+1}) < \frac{3}{2} \varepsilon(x_i)$, then the previous equation implies that
\[\forall i \in \mathbb{N}, \forall (x,y) \in (x_i,x_{i+1}], x \leq y \implies g(y) - g(x) \leq h_+(y) - h_+(x), \]
else $g$ is constant on $]x_i;x_{i+1}]$ and the above equation is trivially true.
Hence, since $g,h_+$ are non-decreasing and continuous, 
\[ \forall (x,y) \in \mathbb{R}, x \leq y \implies g(y) - g(x) \leq h_+(y) - h_+(x). \]
We will now prove by induction that for all $i \in \mathbb{N}^*$,
\begin{equation} \label{3.eq_hyp_rec_lem_approx_non-decreasing} 
 0 \leq g_0(x_i) - g(x_i) \leq 6\varepsilon(x_{i-1}). 
\end{equation}
Base case: This equation is true for $i = 1$ since $x_0 = 0$ and $g(0) = g_0(0) = 0$,
therefore by definition of $g,x_1$, 
$0 \leq g(x_1) \leq g_0(x_1) \leq 2\varepsilon(x_0)$.

Inductive step: Assume that equation \eqref{3.eq_hyp_rec_lem_approx_non-decreasing} is true for some $i \in \mathbb{N}$. 
Then by definition of $x_{i+1}$ and $g$,
\begin{itemize}
 \item If $\varepsilon(x_{i+1}) \geq \frac{3}{2} \varepsilon(x_i)$, then $g(x_{i+1}) = g(x_i)$ therefore 
 $g_0(x_{i+1}) - g(x_{i+1}) = g_0(x_{i+1}) - g_0(x_i) + g_0(x_i) - g(x_i)$. By the induction hypothesis and the definition of $x_{i+1}$,
 \[0 \leq g_0(x_{i+1}) - g(x_{i+1}) \leq 2\varepsilon(x_i) + 6 \varepsilon(x_{i-1}) \leq 
 2\varepsilon(x_i) + 6 \times \frac{2}{3} \varepsilon(x_i) \leq 6 \varepsilon(x_i), \]
 which proves equation \eqref{3.eq_hyp_rec_lem_approx_non-decreasing} for $i+1$.
 \item Otherwise, by definition of $g$, $g(x_{i+1}) = g(x_i) + [g_0(x_{i+1}) - g_0(x_i)]$ therefore by the induction hypothesis
 and since $\varepsilon$ is non-decreasing,
 \[ 0 \leq g_0(x_{i+1}) - g(x_{i+1}) = g_0(x_i) - g(x_i) \leq 6\varepsilon(x_{i-1}) \leq 6\varepsilon(x_i). \]
 This proves equation \eqref{3.eq_hyp_rec_lem_approx_non-decreasing} for $i+1$.
\end{itemize}
By induction, equation \eqref{3.eq_hyp_rec_lem_approx_non-decreasing} is therefore true for all $i \in \mathbb{N}$ (such that $x_i < +\infty$).
Let now $i \in \mathbb{N}$ and $x \in (x_i,x_{i+1}]$. By definition of $g$,
\[ g(x_i) \leq g(x) \leq g(x_i) + (g_0(x_{i+1}) - g_0(x_i))_+ . \]
By equation \eqref{3.eq_hyp_rec_lem_approx_non-decreasing} and definition of $x_{i+1}$,
\begin{align*}
 g(x) - g_0(x) &\leq g(x) - g_0(x_i) \\
 &\leq g(x_i) - g_0(x_i) + (g_0(x_{i+1}) - g_0(x_i))_+ \\
 &\leq 2\varepsilon(x_i) \\
 &\leq 2\varepsilon(x).
\end{align*}
Moreover, by equation \eqref{3.eq_hyp_rec_lem_approx_non-decreasing} and definition of the $x_{i}$,
\begin{align*}
 g(x) - g_0(x) &\geq g(x_i) - g_0(x_{i+1}) \\
 &\geq g(x_i) - g_0(x_i) - [g_0(x_{i+1}) - g_0(x_i)] \\
 &\geq - 6 \varepsilon(x_{i-1}) - 2\varepsilon(x_i) \\
 &\geq - 6 \times \frac{2}{3} \varepsilon(x_i) - 2\varepsilon(x_i) \\
 &\geq - 6 \varepsilon(x_i) \\
 &\geq -6\varepsilon(x).
\end{align*}
It has been proved that for all $i \in \mathbb{N}$ such that $x_i$ is finite,
\[ \forall x \in ]x_i;x_{i+1}], |g(x) - g_0(x)| \leq 6\varepsilon(x). \]
Il must now be proved that $\lim_{n \to +\infty} x_n = +\infty$.
Since $\varepsilon$ is non-decreasing and right-continuous, by definition of $x_n$, 
$g_0(x_{n+1}) \geq g_0(x_n) + 2\varepsilon(x_n) \geq g_0(x_n) + \varepsilon(0)$ ou $\varepsilon(x_{n+1}) \geq \frac{3}{2}\varepsilon(x_n)$.
Since $\varepsilon(0) > 0$ by assumption, this implies that $\max(g_0,\varepsilon)(x_n) \to + \infty$.
The function $\max(g_0,\varepsilon)$ is non-decreasing, thus it is bounded on every interval of the form $[0;x]$, which implies that $x_n \to +\infty$.
This proves the proposition under the assumption that $\varepsilon$ is right-continuous.

In the general case, let $\varepsilon_+ : x \mapsto \inf_{y > x} \varepsilon(y)$, which is non-decreasing and right-continuous.
Since $\varepsilon$ is non-decreasing, $\varepsilon_+ \geq \varepsilon$, therefore the assumptions of the proposition hold
with $\varepsilon_+$ instead of $\varepsilon$. By the right-continuous case of the proposition, which we already proved, 
there exists a non-decreasing function $g$ such that $\NormInfinity{\frac{g - g_0}{\varepsilon_+}} \leq 6$ and 
\[ \forall x,y, x \leq y \implies g(y) - g(x) \leq h_+(y) - h_+(x). \]
Let $x \in \mathbb{R}_+$. For all $y < x$, $|g(y) - g_0(y)| \leq 6 \varepsilon_+(y)$, therefore by continuity of $g,g_0$,
\[ |g(x) - g_0(x)| \leq 6 \sup_{y < x} \varepsilon_+(y) = 6 \sup_{y < x} \inf_{y' > y} \varepsilon(y') \leq 6 \varepsilon(x). \]
This proves the proposition in the general case.

\end{proof}

\begin{proposition} \label{3.prop_approx_proc_gauss}
Let $([x_i, x_{i+1}))_{1 \leq i \leq M-1}$ be a partition of the interval $[a,b)$. 
Let $Y: \{x_1,\ldots, x_M \} \rightarrow \mathbb{R}$ be such that $\left( Y(x_j) \right)_{1 \leq j \leq M}$
is a zero-mean gaussian vector. Abusing notation, we also denote by $Y$ the extension of $Y$
to $[a;b]$ by linear interpolation.
 Let $K_Y: [a;b]^2 \rightarrow \mathbb{R}$ be the variance-covariance function of $Y$.
 Let $h: [a;b] \rightarrow \mathbb{R}$ be a continuous, increasing function and let
 $K_X: [a;b]^2 \rightarrow \mathbb{R}$ be a positive semi-definite function such that:
 \[\forall (s,t) \in [a;b]^2, \ |K_X(s,s) + K_X(t,t) - 2K_X(s,t)| \leq |h(s) - h(t)|. \]
Assume that there exists constants $L > 0$ and $\varepsilon \in [0;1]$ such that:
\begin{itemize}
 \item $\sup_{t \in [a;b]} \sqrt{K_X(t,t)} \leq L$
 \item For any $i \in \{1, \ldots, M-1\}, h(x_{i+1}) - h(x_i) \leq \varepsilon$
 \item $\max_{(i,j) \in \{1, \ldots, M\}^2} \left| K_X(x_i,x_j) - K_Y(x_i,x_j) \right| \leq \varepsilon$.
\end{itemize}
 There exists a universal constant $\kappa$ and a measurable function $f : C([a;b], \mathbb{R}) \rightarrow C([a;b], \mathbb{R})$  
 such that for all random variables $\nu \sim \mathcal{U}([0;1])$ independent from $Y$, 
 $X = f(Y,\nu)$ is a zero-mean gaussian process with variance-covariance function $K_X$ 
 and moreover,
 \[ \mathbb{E} \left[ \sup_{a \leq t \leq b} |X_t - Y_t| \right] \leq \kappa
 \sqrt{(1+L) \log M \left[(h(b)-h(a)) \vee 1 \right]} \varepsilon^\frac{1}{12}. \]
\end{proposition}

\begin{proof}
We assume without loss of generality that $h(b) - h(a) \geq 1$.
We shall moreover use the following notation. 
For $A,B$ two symmetric matrices, $A \prec B$ means that $B - A$ is positive definite. 
$\Norm{A}_{op}$ denotes the matrix operator norm corresponding to the euclidean norm, 
i.e $\Norm{A}_{op} = \sup_{x: \Norm{x}_2 \leq 1} \Norm{Ax}_2$.
We will need the following lemmas:
\begin{lemma} \label{3.lem_norm_ineq_matrice}
 For all $A \in \mathbb{R}^{m \times m}$, $\Norm{A}_{op} \leq m \max_{1 \leq i,j \leq m} |A_{i,j}|$.
\end{lemma}
\begin{proof}
 Let $v \in \mathbb{R}^m$ be such that $\sum_{i = 1}^m v_i^2 = 1$.
 By the Cauchy-Schwartz inequality,
 \[\Norm{Av}^2 = \sum_{i = 1}^m \left( \sum_{j = 1}^m A_{i,j} v_j \right)^2 
 \leq \sum_{i = 1}^m \sum_{j = 1}^m A_{i,j}^2 \leq m^2 \max_{i = 1,\ldots,m} A_{i,j}^2. \]
 This is true for any $v$, which proves lemma \ref{3.lem_norm_ineq_matrice}.
\end{proof}

Lemma \ref{3.lem_trace_ineq} below is a special case of Mc-Carthy's trace inequality \citep[Lemma 2.6]{McCarthy1967}.
\begin{lemma} \label{3.lem_trace_ineq}
 Let $A,B$ be two symmetric, positive semi-definite matrices, then
 \[ \text{Tr}(\sqrt{A+B}) \leq \text{Tr}(\sqrt{A}) + \text{Tr}(\sqrt{B}). \]
\end{lemma}

 The hypotheses imply that $h$ is bijective from $[a;b]$ to $[h(a);h(b)]$. Let $m \in \mathbb{N}$.
 For all $j \in \{1, \ldots, m\}$, let 
 \begin{equation} \label{3.eq_def_tj}
  t_j = \max \left\{ x_i | i \in \{1, \ldots, M\}, h(x_i) \leq h(a) + \frac{j-1}{m-1}[h(b) - h(a)] \right\}.
 \end{equation} 
 Let $K_{X,m} = \bigl( K_X(t_i,t_j) \bigr)_{1 \leq i,j \leq m}$ and $K_{Y,m} = \bigl( K_Y(t_i,t_j) \bigr)_{1 \leq i,j \leq m}$. The 
 Wasserstein distance between two gaussian vectors is known \citep{Olkin1982}: there exists a coupling 
 $\tilde{X}^m, \tilde{Y}^m$ of the distributions $\mathcal{N} \left(0, K_{X,m} \right)$ and $\mathcal{N} \left(0, K_{Y,m} \right)$ such that:
 \[\mathbb{E} \left[ \sum_{i = 1}^m (\tilde{X}^m_i - \tilde{Y}^m_i)^2 \right] = \text{Tr}\left( K_{X,m} + K_{Y,m} - 2 \bigl( K_{X,m}^\frac{1}{2} K_{Y,m} K_{X,m}^\frac{1}{2} \bigr)^\frac{1}{2} \right). \]
Thus
\begin{align*}
 K_{X,m}^2 &= K_{X,m}^\frac{1}{2} K_{Y,m} K_{X,m}^\frac{1}{2} + K_{X,m}^\frac{1}{2} \bigl(K_{Y,m} - K_{X,m} \bigr) K_{X,m}^\frac{1}{2} \\
 &\prec K_{X,m}^\frac{1}{2} K_{Y,m} K_{X,m}^\frac{1}{2} + \Norm{K_{Y,m} - K_{X,m}}_{op} K_{X,m}.
\end{align*}
By lemma \ref{3.lem_trace_ineq}, 
\[ \text{Tr}(K_{X,m}) \leq \text{Tr} \left( \bigl(K_{X,m}^\frac{1}{2} K_{Y,m} K_{X,m}^\frac{1}{2} \bigr)^\frac{1}{2} \right) + \Norm{K_{Y,m} - K_{X,m}}_{op}^\frac{1}{2} \text{Tr}\left(K_{X,m}^\frac{1}{2} \right).  \]
By the same argument (exchangeing $X$ and $Y$),
\[\text{Tr}(K_{Y,m}) \leq \text{Tr} \left( \bigl(K_{Y,m}^\frac{1}{2} K_{X,m} K_{Y,m}^\frac{1}{2} \bigr)^\frac{1}{2} \right) + \Norm{K_{Y,m} - K_{X,m}}_{op}^\frac{1}{2} \text{Tr}\left(K_{Y,m}^\frac{1}{2} \right). \]
It follows that
\begin{align*}
 \mathbb{E} \left[ \sum_{i = 1}^m (\tilde{X}^m_i - \tilde{Y}^m_i)^2 \right] &\leq \Norm{K_{Y,m} - K_{X,m}}_{op}^\frac{1}{2} \text{Tr}\left(K_{X,m}^\frac{1}{2} + K_{Y,m}^\frac{1}{2} \right) \\
 &\leq m^\frac{1}{2} \NormInfinity{K_{Y,m} - K_{X,m}}^\frac{1}{2} \left(\sqrt{\text{Tr}(K_{X,m})} + \sqrt{\text{Tr}(K_{Y,m})} \right) \\
 &\leq m^\frac{1}{2} \NormInfinity{K_{Y,m} - K_{X,m}}^\frac{1}{2} m^\frac{1}{2} \max_{1 \leq i,j \leq m} \left\{ \sqrt{K_X(t_i,t_i)} + \sqrt{K_Y(t_j,t_j)} \right\} \\
 &\leq 2 m \sqrt{\varepsilon} \sqrt{L^2 + \varepsilon}. \numberthis \label{3.eq_ub_wass_dist} 
\end{align*}
By the transfer principle (Kallenberg, Theorem 5.10), 
there exists $f_1$ such that for all uniform random variables variables $\nu_1$ independent from $Y$,
$(\tilde{X}^m, \tilde{Y}^m) \sim \left(f_1(Y^m, \nu_1), Y^m \right)$.

Let $X_0$ be a gaussian process with variance-covariance function $K_X$.
Let $W_0 = X_0 \circ h^{-1}$.
For any $(s,t) \in [h(a);h(b)]^2$,
$W_0(t) - W_0(s)$ is a centred gaussian random variable, hence for all $r > 0$, there exists a universal constant $C(r)$ such that
\begin{align*}
 \mathbb{E} \left[ (W_0(t) - W_0(s))^r \right] &\leq C(r) \mathbb{E}\left[ (W_0(t) - W_0(s))^2 \right]^\frac{r}{2}  \\
 &\leq C(r) \Bigl(K_X(h^{-1}(t),h^{-1}(t)) + K_X(h^{-1}(s),h^{-1}(s)) \\ 
 &\hphantom{C(r) \Bigl(} - 2K_X(h^{-1}(s),h^{-1}(t)) \Bigr)^\frac{r}{2} \\
 &\leq C(r) |t-s|^{\frac{r}{2}}.
\end{align*}

By the Kolmogorov continuity theorem \cite[Chapter 1, Theorem 2.1]{Revuz1999}, 
applied to $x \mapsto \frac{W_0 \left(h(a) + (h(b) - h(a))x \right)}{\sqrt{h(b) - h(a)}}$,
there exists a continuous version $W_1$ of $W_0$ such that for any $\theta \in [0;1)$
and all $(s,t) \in [h(a);h(b)]^2$,
\[ \mathbb{E} \left[ \left( \sup_{s \neq t} \frac{|W_1(t) - W_1(s)|}{|t-s|^{\theta \left(\frac{1}{2} - \frac{1}{r} \right)}} \right)^r \right] 
\leq \frac{(h(b)-h(a))^{\frac{r}{2}}}{(h(b) - h(a))^{\theta r \left(\frac{1}{2} - \frac{1}{r} \right)}} B(\theta,r) <  +\infty, \]
where $B(\theta,r)$ is a universal constant.
Let $X_1 = W_1 \circ h$, which is still a gaussian process, with variance-covariance function $K_X$. 
Then, since by assumption $h(b) - h(a) \geq 1$, for all $r \geq 2$,
\begin{equation} \label{3.eq_cont_X1}
 \mathbb{E} \left[ \left( \sup_{s \neq t} \frac{|X_1(t) - X_1(s)|}{|h(t)-h(s)|^{\theta \left(\frac{1}{2} - \frac{1}{r} \right)}} \right)^r \right] 
\leq [h(b)-h(a)]^{\frac{r}{2}} B(\theta,r) <  +\infty.
\end{equation}
The $C([a;b],\mathbb{R})$-valued process $X_1$ induces a probability distribution $Q$ on the Borel space $C([a;b],\mathbb{R})$.  
Furthermore, $(X_1(t_j))_{1 \leq j \leq m} \sim \tilde{X}^m \sim  f_1(Y^m,\nu_1)$. 
By (Kallenberg, Theorem 5.10), there exists a measurable function $f_2$ such that for all uniform random variables
$\nu_3$ independent from $Y^m,\nu_1$, $(X_1, (X_1(t_j))_{1\leq j \leq m}) \sim (f_2(f_1(Y^m,\nu_1), \nu_3), f_1(Y^m,\nu_1))$.
Let $X = f_2(f_1(Y^m,\nu_1), \nu_3)$ and $X^m = (X(t_j))_{1 \leq j \leq m}$. Almost surely, 
\begin{claim} \label{3.list_prop_X}
\begin{enumerate} 
 \item $X^m = (X(t_j))_{1 \leq j \leq m} = f_1(Y^m,\nu_1)$ p.s, so
 \item $(X^m,Y^m) \sim (\widetilde{X}^m, \widetilde{Y}^m)$, in particular by equation \eqref{3.eq_ub_wass_dist},
 \[ \mathbb{E} \left[ \Norm{X^m - Y^m}^2 \right] \leq 2m \sqrt{\varepsilon} \sqrt{L^2 + \varepsilon}. \]
 \item $X \sim X^1$ as a random continuous function, in particular by equation \eqref{3.eq_cont_X1}
 with $\theta = \frac{3}{4}$ and $r = 6$,
 \begin{equation} \label{3.eq_cont_X}
  \forall \delta > 0, \mathbb{E} \left[ \sup_{(s,t) \in [a;b]: |h(t) - h(s)| \leq \delta} |X(t) - X(s)|^6 \right]^{\frac{1}{6}} 
 \leq \sqrt{h(b) - h(a)} B(\tfrac{3}{4},6)^{\frac{1}{6}} \delta^{\frac{1}{4}}.
 \end{equation}
\end{enumerate}
\end{claim}
By abuse of notation, denote $X^m, Y^m$ the random processes obtained by linear interpolation between the points
 $(t_j, X^m_j)$ and $(t_j, Y^m_j)$, respectively. 
For all $t \in [a;b]$, there exists $j \in \{1, \ldots, m\}$ such that $t_j \leq t \leq t_{j+1}$, therefore
$|h(t) - h(t_j)| \leq h(t_{j+1}) - h(t_j)$, since $h$ is non-decreasing. 
By definition of $t_{j+1}$ (equation \eqref{3.eq_def_tj}), 
$h(t_{j+1}) \leq h(a) + \frac{j}{m-1} (h(b) - h(a))$ and furthermore, there exists $i \in \{1, \ldots, M\}$ 
such that $x_i = t_j$. By equation \eqref{3.eq_def_tj} which defines $t_j$, $h(x_{i+1}) > h(a) + \frac{j-1}{m-1} (h(b) - h(a))$,
which yields 
\begin{align*}
   |h(t) - h(t_j)| &\leq h(a) + \frac{j}{m-1} (h(b) - h(a)) - h(x_{i+1}) + h(x_{i+1}) - h(x_i) \\ 
   &\leq \frac{h(b) - h(a)}{m-1} 
+ h(x_{i+1}) - h(x_i). 
\end{align*}
By assumption, $h(x_{i+1}) - h(x_i) \leq \varepsilon$, which yields
\begin{equation} \label{3.eq_ub_err_t_tj}
 \forall t \in [a;b], \exists j(t) \in \{1, \ldots, m\}, t_j \leq t \leq t_{j+1} \text{ and } |h(t) - h(t_{j(t)})| \leq \frac{h(b) - h(a)}{m} + \varepsilon.
\end{equation}
Since $Y$ is piecewise linear on the partition $([x_i,x_{i+1}))_{1 \leq i \leq M-1}$, 
\begin{align*}
 \sup_{a \leq t \leq b} |X(t) - Y(t)| &\leq \sup_{t \in [a;b]}|X(t) - X(t_{j(t)})|
 + \max_{j \in \{1, \ldots, m\}} |X(t_j) - Y(t_j)|  \\
 &\quad + \sup_{t \in [a;b]}|Y(t) - Y(t_{j(t)})| \\
 &\leq \sup_{(s,t): |h(s)-h(t)| \leq \varepsilon + \frac{h(b)-h(a)}{m}} |X(s) - X(t)|
 + \sqrt{ \sum_{j = 1}^m |X(t_j) - Y(t_j)|^2} \\
 &\quad + \max_{j \in \{1, \ldots, m\}} \max \left\{ |Y(x_i) - Y(t_j)| : 
 i \in \{1, \ldots, M\} \cap [t_j;t_{j+1}] \right\}.
\end{align*}
Thus, by claim \ref{3.list_prop_X},
\begin{align*}
 &\mathbb{E} \left[ \sup_{a \leq t \leq b} |X(t) - Y(t)| \right] \\
 &\quad \leq \mathbb{E} \Bigl[ \sup_{(s,t): |h(s)-h(t)| \leq \varepsilon + \frac{h(b)-h(a)}{m}} |X(s) - X(t)| \Bigr] 
 + \mathbb{E} \Bigl[ \sum_{i = 1}^m (X^m_i - Y^m_i)^2  \Bigr]^\frac{1}{2} \\
 &\quad \quad + \mathbb{E} \left[ \max_{j \in \{1, \ldots, m\}} \max \left\{ |Y(x_i) - Y(t_j)|: 
i \in \{1, \ldots, M\} \cap [t_j;t_{j+1}] \right\} \right] \\
 &\quad \leq \sqrt{h(b) - h(a)} B(\tfrac{3}{4},6)^{\frac{1}{6}} \left(\varepsilon + \frac{h(b)-h(a)}{m} \right)^{\frac{1}{4}}
 + \left( 2m \sqrt{\varepsilon} \sqrt{L^2 + \varepsilon} \right)^{\frac{1}{2}} \\
 &\quad \quad + \sqrt{2 \log M} \max_{j \in \{1, \ldots, m\}} \max \left\{ \sqrt{\mathbb{E} [|Y(x_i) - Y(t_j)|^2]} : 
 i \in \{1, \ldots, M\} \cap [t_j;t_{j+1}] \right\}. \numberthis \label{3.eq_ub_approx_XY}
\end{align*}
Furthermore, for any $(i,j) \in [1;M]^2$,
\begin{align*}
 \mathbb{E}[(Y(x_i) - Y(x_j))^2] &= K_Y(x_i,x_i) + K_Y(x_j,x_j) -2K_Y(x_i,x_j) \\
 &\leq K_X(x_i,x_i) + K_X(x_j,x_j) -2K_X(x_i,x_j) \\
 &\quad + 4\max_{(r,s) \in [1;M]^2} |K_X(x_r,x_s) - K_Y(x_r,x_s)|\\
 &\leq |h(x_i)-h(x_j)| + 4\varepsilon.
\end{align*}
Setting $\kappa = B(\tfrac{3}{4},6)^{\frac{1}{6}}$,
it follows by equation \eqref{3.eq_ub_approx_XY} and the non-decreasing nature of $h$ that
\begin{align*}
\mathbb{E} \left[ \sup_{a \leq t \leq b} |X_t - Y_t| \right] &\leq \sqrt{h(b) - h(a)} B(\tfrac{3}{4},6)^{\frac{1}{6}} \left(\varepsilon + \frac{h(b)-h(a)}{m} \right)^{\frac{1}{4}} 
+ \sqrt{2m(L+1)} \varepsilon^{\frac{1}{4}} \\
&\quad + \sqrt{2 \log M} \sqrt{h(t_{j+1}) - h(t_j) + 4\varepsilon} \\
&\leq \kappa \sqrt{h(b) - h(a)} \varepsilon^{\frac{1}{4}} + \kappa \frac{[h(b) - h(a)]^{\frac{3}{4}}}{m^\frac{1}{4}} 
+ \sqrt{2m(L+1)} \varepsilon^{\frac{1}{4}} \\
&\quad + \sqrt{2 \log M} \sqrt{\frac{h(b) - h(a)}{m} + 5\varepsilon} \text{  by equation } \eqref{3.eq_ub_err_t_tj}.
\end{align*}
Let now $m = \left \lceil \frac{h(b)-h(a)}{\varepsilon^{\frac{1}{3}}} \right \rceil$. 
Since by assumption $\varepsilon \leq 1$, $h(b) - h(a) \geq 1$ it follows finally, by keeping only the largest powers of
$[h(b)-h(a)], \varepsilon, L$ and $\log M$, that
\[ \mathbb{E} \left[ \sup_{a \leq t \leq b} |X_t - Y_t| \right] \leq \kappa \sqrt{h(b)-h(a)} \sqrt{2(L+1) \log M}
\varepsilon^{\frac{1}{12}}. \]
for an absolute constant $\kappa$.
\end{proof}

\bibliographystyle{plainnat}
\bibliography{biblio_these.bib}
\end{document}